\documentclass[a4paper,10pt]{amsart}
\usepackage[utf8]{inputenc}
\usepackage{amsmath, amssymb, amsthm, array, marginnote, xcolor, enumerate, footnote, enumerate, MnSymbol, bbm, mathtools, multirow}
\usepackage[
backend=biber,
style=alphabetic,
sorting=nyt,
giveninits=true,
maxbibnames=9
]{biblatex}
\renewbibmacro{in:}{}
\usepackage{svg, float, subcaption, bbm}
\usepackage{tikz, tikz-cd, ytableau}
\usetikzlibrary{matrix, arrows.meta, shapes.geometric}

\usepackage{xcolor}
\colorlet{myred}{red!10}
\colorlet{mycyan}{cyan!10}

\textheight9in  \def\DATE{\today}
\hoffset - 2cm  \hbadness=100000

\def\MZ{{M\bbZ}}
\def\S{{\mathbb{S}}}
\def\bbZ{{\mathbb Z}}

\def\oC{{\mathcal{C}}}
\def\oP{{\mathcal{P}}}
\def\oQ{{\mathcal{Q}}}
\def\oS{{\mathcal{S}}}
\def\oA{{\mathcal{A}}}
\def\oK{{\mathcal{K}}}
\def\bbk{{\mathbb{K}}}
\def\W{{\mathcal{W}}}
\def\A{{\mathcal{A}}}
\def\Lat{{\mathtt{Lat}}}
\def\LatOper{{\mathtt{LatOper}}}
\def\Poly{{\mathtt{Poly}}}

\def\cPoly{{\mathtt{cPoly}}}
\def\Sets{{\mathtt{Sets}}}
\def\rada#1#2{#1,\ldots,#2}
\def\Rada#1#2#3{{#1}_{#2},\ldots,{#1}_{#3}}
\def\Free{{\mathbb F}}
\def\L{{\mathbb L}}

\def\Inv{{\mathbb F}^{inv}(\{a,b\})}
\def\inv{{\text{inv}}}
\def\id{{\mathbbm 1}}
\def\zero{\mathbf{0}}
\def\one{\mathbf{1}}

\def\lowprod{{%
    \setbox0\hbox{$\bigcircle$}%
    \rlap{\hbox to \wd0{\hss$\vee$\hss}}\box0
}}

\def\upprod{{%
    \setbox0\hbox{$\bigcircle$}%
    \rlap{\hbox to \wd0{\hss$\wedge$\hss}}\box0
}}

\def\luprod{{%
    \setbox0\hbox{$\bigcircle$}%
    \rlap{\hbox to \wd0{\hss$\wedge$\hss}}\box0
}}

\theoremstyle{definition}
\newtheorem{definition}{Definition}[section]

\theoremstyle{plain}
{
\newtheorem{lemma}{Lemma}[section]
\newtheorem{proposition}{Proposition}[section]
\newtheorem{theorem}{Theorem}[section]
}
\theoremstyle{definition}
{
\newtheorem{example}{Example}[section]
\newtheorem{corollary}{Corollary}[section]
\newtheorem{remark}{Remark}[section]
}

\makeatletter
\newcommand{\leqnomode}{\tagsleft@true\let\veqno\@@leqno}
\newcommand{\reqnomode}{\tagsleft@false\let\veqno\@@eqno}

\makeatletter
\newcommand{\addresseshere}{%
  \enddoc@text\let\enddoc@text\relax
}
\makeatother

\author[D. Bashkirov]{Denis Bashkirov} \email{bashkirov@math.cas.cz}
\address{Czech Academy of Sciences, Institute of Mathematics, {\v Z}itn{\'a} 25,
         115 67 Prague 1, Czech Republic}
\title{Lattice operads and operad filtrations}

\begin{document}
\maketitle
\begin{abstract}
We explore the concept of filtrations of operads using lattice-valued operads as indexing objects. The framework generalizes the familiar $\mathbb{Z}$-indexed filtrations of associative algebras and operads, while also naturally extending to a broader range of examples. A key feature of lattice operads is the distributivity of partial compositions with respect to meets and joins. We show that several well-known combinatorial lattices, such as Tamari lattices, form operads exhibiting this property. Further examples include the operad of integer partitions supported on Young's lattice, and operads of integer compositions of types $A$, $B$, and $D$, which we relate to operads associated with families of regular polytopes. We discuss the partial compatibility of the weak order on the symmetric group with the structure of the permutations operad.
\end{abstract}
\section{Introduction}
Filtrations and gradings, being a standard tool for studying structural properties of associative algebras, rings and modules, naturally extend to the setting of operads. While most examples and applications of operad filtrations found in the literature primarily concern filtrations indexed by integers, this paper aims to explore a broader notion of filtration, defined by more flexible indexing objects, such as poset-valued or, more specifically, lattice-valued operads. These indexing objects are intended to align the natural partial order on the set of all linear subspaces of each component 
$\oP(n)$ of a given operad 
$\oP$
 with the operadic compositions within 
$\oP$, providing a finer perspective on the structure of operads.

This construction encompasses, in particular, the standard case of $\bbZ$-indexed filtrations, as well as filtrations by the multiindex operad $M\bbZ$, as used in \cite{superbig} to study algebras of multilinear differential operators. The concept of filtration by an ordered operad, along with applications related to distributive laws for Koszul operads, was originally developed by A. Khoroshkin and V. Dotsenko \cite{khoroshkin, dotsenko}; see example \ref{InvOp} for a discussion of their \textit{operad of inversions}.

% While examples and applications of operad filtrations discussed in  literature concern primarily the case of filtrations indexed by the integers, one of the aims of the present paper is to elaborate on the notion of a filtration defined by means of a more flexible indexing object, that being a poset-valued or, more restrictively, a lattice-valued operad. Such an indexing object is meant to relate, in a reasonably fine way, the natural partial order on the set of all linear subspaces of each component $\oP(n)$ of a given operad $\oP$ with the operadic compositions within $\oP$. The construction encompasses, in particular, both the standard case of the $\mathbb{Z}$-indexed filtrations and filtrations by the multiindex operad $M\mathbb{Z}$ used in \cite{superbig} for studying algebras of multilinear differential operators, cf. Example~\ref{LatOpEx}\eqref{LatOpExMZ}. We would note that the concept of a filtration by an ordered operad, along with some applications pertaining to distributive laws for Koszul operads, is originally due to A.Khoroshkin and V.Dotsenko \cite{khoroshkin, dotsenko};
% see Example \ref{InvOp} for an account of their \textit{operad of inversions}.
% Another aim of the paper is to illustrate how some classical families of lattices of combinatorial origin can be equipped with operadic structures.

The paper is organized as follows. In section \ref{Prelims} we provide some background on different symmetric monoidal structures in the category of polytopes and affine mappings and in the category of lattices and lattice homomorphisms, and recall the constructions of word operads and $M$-associative operads. In section \ref{LatOps}, we introduce the notion of a lattice operad and present some basic results concerning formal properties thereof. A characteristic feature of lattice operads is a certain distributivity property with respect to joins and meets that holds for partial compositions:
 \begin{align*}
   (p\wedge r) \circ_i (q\wedge s) &= (p\circ_i q)\wedge (r\circ_i s)\\
   (p\vee r) \circ_i (q\vee s) &= (p\circ_i q)\vee (r\circ_i s)
\end{align*}

In sections \ref{LatOpEx} and \ref{LatOpsPolys} we present examples of lattice operads.
These include, in particular, an operad of integer partitions supported on the Young's lattice and operads of integer compositions, where, besides ordinary integer compositions, we consider 2-colored and restricted 2-colored compositions as analogues of compositions in types $B$ and $D$.
While it is well-known that Tamari lattices emerge naturally in the operadic context, for instance, in the form of the $1$-skeletons of associahedra, we show that they in fact assemble to a lattice operad; cf. example \ref{TamariOp}.
We discuss partial compatibility of the non-$\Sigma$ operad of permutations with the weak Bruhat order.
\newpage
We would note that there are lattice operads of combinatorial nature associated to the reflection groups of classical types via the corresponding families of regular polytopes and set partitions. 
\begin{center}
\begin{tabular}{c|c|m{3cm}|m{3cm}|m{2.7cm}|m{2cm}|}
\hline
Type & Polytopes & Monoidal \mbox{product} on polytopes & 
Monoidal \mbox{product} on lattices & Operad of & Symmetry\\
\hline
\multirow{3}{*}{A} &
\multirow{3}{*}{Simplicies}
 & Cartesian $\times$ & Lower truncated $\dtimes$ & Power sets & $\S_n$\\
 \cline{3-6}
 & &  Join $*$ & Cartesian $\times$ & Integer compositions & $\bbZ_2\times\bbZ_2$\\
 \hline
 \multirow{3}{*}{B} &
 \multirow{3}{*}{Hypercubes} &
 \multirow{3}{*}{Cartesian $\times$} &
 \multirow{3}{*}{Lower truncated $\dtimes$} & 
 Ternary words & $\mathbb{B}_n$\\
 \cline{5-6} 
 & & & & 2-colored integer compositions & 
 $\bbZ_2\times\bbZ_2$ \\ 
 \cline{1-6}
 % \multirow{3}{*}{C} &
 % \multirow{3}{*}{Hyperoctahedra} &
 % \multirow{3}{*}{Direct sum $\oplus$} &
 % \multirow{3}{*}{Upper truncated $\utimes$} & 
 C & Hyperoctahedra & Direct sum $\oplus$ & Upper truncated $\utimes$ &
 %\multicolumn{2}{|c|}{\multirow{3}{*}{\ditto}} \\ 
 \multicolumn{2}{|c|}{Same as in type $B$} \\ 
 %& & & & \multicolumn{2}{|c|}{} \\ 
 \hline
 D & & &  Lower truncated $\dtimes$ & 2-colored restricted integer compositions  & $\bbZ_2$ \\
 \hline
 I$_2$ & Regular polygons & Disjoint union $\sqcup$ & Disjoint union $\sqcup$ & Totally cyclically ordered sets & $\mathbbm{D}\text{ih}_n$ \\
 \hline
\end{tabular}
\end{center}

The notion of a filtration indexed by a lattice operad is discussed in section \ref{filtrations}. Among some basic technical results, we establish the functorial properties of this construction (Theorems \ref{FiltFuncL} and \ref{FiltPush}) and show (Proposition \ref{reprprop}) that the set-valued functor of all filtrations of a given linear operad is representable. We would note that some of the results presented in sections \ref{LatOps} and \ref{filtrations} can be generalized to the case of semilattice-valued and, more generally, poset-valued operads.

As apparent from the case of associative algebras and modules, for practical reasons it is worth to single out filtrations of a particularly amenable structure, namely, the ones that admit a characterization in terms of a certain generating set (the \textit{standard filtrations}), as well as filtrations with some additional compositional constraints (such as the \textit{$D$-filtrations}). As the organizing principle here, for a given linear operad $\oP$ and an indexing lattice operad $L$, filtrations with properties of interest are introduced via closure operators on the complete lattice of all subspaces of $\oP$ (done arity-wise) indexed by $L$. We present some relevant examples in section \ref{extra}. In the remaining part we discuss the notion of gradings and the associated graded for a filtered operad. We conclude by an outlook on possible applications related to lattice paths combinatorics.

%\subsubsection*{Conventions} Throughout the text the %term \textit{operad} stands for a connected, possibly %non-unital operad. Unless indicated otherwise, such an %operad is assumed to be symmetric. The symbols $\zero$ %and $\one$ denote the smallest and the largest elements %of a lattice under consideration. 

\section{Preliminaries}
\label{Prelims}

\subsection{Symmetric monoidal structures on the category of lattices.}
\label{CatLat}
Let $\Lat$ be the category of lattices and lattice homomorphisms. 
There are four symmetric monoidal structures on $\Lat$ that will be of relevance to us. Those are
\begin{itemize}
\item 
the standard cartesian product $P\times Q$ of $P$ and $Q$ as posets;
\item 
the \textit{lower truncated} (or the \textit{diamond} \cite{fox}) product 
$P\dtimes Q := (P-\{\zero\})\times (Q-\{\zero\}) \cup \{\zero\}$;
\item 
the \textit{upper truncated} product 
$P\utimes Q := (P-\{\one\})\times (Q-\{\one\}) \cup \{\one\}$;
\item 
the \textit{disjoint union} 
$P \sqcup Q := (P-\{\zero, \one\})\sqcup (Q-\{\zero, \one\}) \cup \{\zero, \one\}$.
\end{itemize}
Here and further on in the text, the symbols $\zero$ and $\one$ denote, respectively, the bottom and the top elements of a lattice under consideration, allowing the trivial case of $\zero=\one$. 

At the level of objects, the lower and upper truncated products were introduced by G.Kalai in \cite{kalai}. For the sake of completeness, we present some details concerning the categorical aspects thereof. Specifically, for the lower truncated product, the associators $\alpha_{P,Q,R}: (P\dtimes Q)\dtimes R \to P\dtimes ( Q\dtimes R)$ are defined by
\[
\alpha_{P,Q,R}(t) = 
\begin{cases}
(p, (q, r)),\quad &t=((p, q), r) \in ((P-\{\zero\})\times (Q-\{\zero\}))\times (R-\{\zero\}) \\
\zero,\quad &t=\zero
\end{cases}.
\]
The unit object for this product is the unique two-element lattice $\mathcal{B}_1=\{\zero, \one\}$. 
The left and right unitors are the evident isomorphisms 
\begin{align*}
{\lambda: \mathcal{B}_1 \dtimes P = \{\mathbf{1}\}\times (P-\{\zero\})\cup \{\zero\} \overset{\sim}{\longrightarrow} P}\\
{\rho: P\dtimes \mathcal{B}_1 = (P-\{\zero\}) \times \{\one\}\cup \{\zero\} \overset{\sim}{\longrightarrow} P}
\end{align*}
respectively. The symmetric structure is given by permuting the cartesian factors.

The standard cartesian monoidal structure and the lower truncated product on $\Lat$ are related by means of the following 
\begin{lemma}
\label{LaxMono}
There is a lax monoidal functor $F: (\Lat, \dtimes) \to (\Lat, \times)$ which acts identically on the objects and morphisms of $\Lat$. 
The coherence maps for $F$ are given by $\epsilon: \{\one\} \to \mathcal{B}_1$, where $\epsilon(\one):=\one$, for the units, and by
\begin{align*}
\mu_{P,Q}&: P\times Q \to P\dtimes Q \\
%&F(P)\dtimes F(Q) \to F(P\times Q)\\
%&P\dtimes Q \to P\times Q \\
\mu_{P,Q}((p, q))&:= 
\begin{cases}
(p, q),\quad &p\neq \zero, q\neq \zero\\
\zero, \quad &\text{otherwise} 
\end{cases}
\end{align*}
for all $P, Q\in \Lat$.
\end{lemma}
\begin{proof}
Verifying commutativity of the hexagon diagram
\[
\begin{tikzcd}[column sep={1.5cm,between origins}, row sep={1.2cm,between origins}]
    & (P\times Q)\times R \arrow[rr, swap, "\sim"'] \arrow[ld, "\mu_{P,Q}\times id"'] && P\times (Q\times R) \arrow[rd, "id \times \mu_{Q,R}"] &  \\
    (P\dtimes Q)\times R \arrow[rd, "\mu_{P\dtimes Q, R}"']&  &&  & P\times (Q\dtimes R) \arrow[ld, swap, "\mu_{P,Q\dtimes R}"']\\
    & (P\dtimes Q)\dtimes R \arrow[rr, "\alpha_{P,Q,R}"] && P\dtimes (Q\dtimes R)  & 
\end{tikzcd}
\]
amounts to tracing it separately for $((p, q), r)\in (P\times Q)\times R$ with $p,q,r\neq \zero$ and for the remaining case when at least one of the entries $p, q, r$ is $\zero$. In the former case, both paths lead to $(p,(q,r))$, while in the latter one the resulting value is $\zero$.

Verifying the coherence conditions for the units
\[
\begin{tikzcd}
\{\one\}\times P \arrow[r, "\epsilon\times id"] \arrow[d, swap, "\sim"] & \mathcal{B}_1\times P \arrow[d, "\mu_{\mathcal{B}_1, P}"]\\
P & \mathcal{B}_1\dtimes P \arrow[l, "\lambda"]
\end{tikzcd}
\begin{tikzcd}
P \times \{\one\} \arrow[r, "id\times\epsilon"] \arrow[d, swap, "\sim"] & P \times \mathcal{B}_1\arrow[d, "\mu_{P, \mathcal{B}_1}"]\\
P & P\dtimes \mathcal{B}_1 \arrow[l, "\rho"]
\end{tikzcd}
\]
is equally straightforward and amounts to chasing the diagrams for $p\in P-\{\zero\}$ and for $p=\zero$.
\end{proof}
An analogous result holds for the category $\Lat$ taken with the upper truncated product $\utimes$. For the disjoint union, the associators are inherited from the category of sets with the disjoint union as the symmetric monoidal structure. The unit object in this case is the trivial single-element lattice $\{\one\}$. An analogue of Lemma~\ref{LaxMono} can be established upon considering 
\begin{align*}
\mu_{P,Q}&: P\times Q \to P\sqcup Q \\
\mu_{P,Q}((p, q))&:= 
\begin{cases}
p,\quad &p\neq \zero, q=\one\\
q,\quad &q\neq \zero, p=\one\\
\zero,\quad &\textit{otherwise}
\end{cases}
\end{align*}
and $\epsilon: \{\one\} \overset{\sim}{\longrightarrow} \{\one\}$ as the coherence maps.

% As a corollary of the above, an operad in $\Lat$ with any of the above monoidal structures gives rise to an operad 
\subsection{Categories of polytopes.}
\label{PolyCat}
A few examples of operads in $\Lat$ will be constructed by means of taking face lattices of certain convex polytopes, cf. section ~\ref{LatOpsPolys}. That will require using a proper categorical and symmetric monoidal structure for polytopes in each particular case. We list the options relevant for us below.

\begin{enumerate}
\item 
Let $\Poly$ denote the category of convex polytopes and affine maps. There is a symmetric monoidal structure on $\Poly$ given by the cartesian product of polytopes $P\times Q$.
\item 
Let $\Poly$ be as above, and $P\subset \mathbb{R}^n$, $Q\subset \mathbb{R}^m$ be polytopes of dimensions $n$ and $m$ respectively. The \emph{join product} $P*Q$, defined as the convex hull of $\{(p, 0, 0)| p\in P\}\cup\{(0, q, 1)| q\in Q\}\subset \mathbb{R}^{n+m+1}$, is a symmetric monoidal structure on $\Poly$.
\item 
Let $\Poly$ be as above and $\Poly_0$ be its subcategory consisting of all polytopes $P$ such that for a $n$-dimensional $P$, the origin 
$0\in \mathbb{R}^n$ is contained in its relative interior. 
%Here, $P$ is of dimension $n$. 
%Note that the zero-dimensional polytope $\{pt\}$ is in $\Poly_0$. 
The subcategory $\Poly_0$ carries a symmetric monoidal structure given by the \textit{direct sum} $P\oplus Q$ defined as the convex hull of~${\{(p, 0)| p\in P\}\cup\{(0, q)| q\in Q\}\subset \mathbb{R}^{n+m}}$ for~${P \subset \mathbb{R}^n, Q \subset \mathbb{R}^m}$.
% \item 
% Let $\mPoly_*$ be the category, whose objects are finite disjoint unions $\coprod\limits_{i}P_i\subset \mathbb{R}^n$ of possibly non-convex polytopes such that $0\notin P_i$ for all $i$, and morphisms are continuous mappings that are affine on each $P_i$. 
% Then the \emph{disjoint union} of 
% $\coprod\limits_{i}P_i \subset \mathbb{R}^n$ and $\coprod\limits_{j}Q_j\subset \mathbb{R}^m$ is defined by 
% the natural embedding of
% $\left(\coprod\limits_{i}P_i\right) \sqcup \left(\coprod\limits_{j}Q_j\right)$ into~${\mathbb{R}^n\oplus \mathbb{R}^m=\mathbb{R}^{n+m}}$.
\item 
Following \cite[Definition 7]{diagonal1}, let $\cPoly$ be the category, whose objects are convex polytopes and morphisms are continuous maps $f: P \to Q$ subject to the following condition: there exists a polytopal subcomplex $\mathcal{D}$ of $Q$ such that $f$ sends $P$ homeomorphically to the underlying set $|\mathcal{D}|$ and $f^{-1}(\mathcal{D})$ induces a polytopal subdivision of $P$. The symmetric monoidal structure on $\cPoly$ is given by the cartesian product of polytopes.
\end{enumerate}

\subsection{Word operads and $M$-associative operads.} 
We will make use of two particularly simple constructions of operads.
First, let $(\oC, \times)$ be a concrete cartesian symmetric monoidal category, ${\mathtt{Oper}}_\oC$ be the category of symmetric unital operads in $\oC$ and ${\mathtt{Mon}}_\oC$ be the category of monoids in $\oC$. There is a functor $\W: {\mathtt{Mon}}_\oC \to {\mathtt{Oper}}_\oC$
that sends a monoid $M$ to the symmetric \emph{word operad} $\W M$, where for all ~$n\geq 1$, $\W M(n):=M^{\times n}$ with the natural $\S_n$-action by permutation of the factors. 
The partial compositions~$\circ_i: \W M(m) \times \W M(n) \to \W M(n+m-1)$ are given by 
\begin{align*}
(p_1,\dots, p_m)\circ_i (q_1,\dots, q_n):=
(p_1,\dots, p_{i-1}, p_i\cdot q_1, \dots, p_i\cdot q_n, p_{i+1},\dots, p_m).
\end{align*}
Such an operad $\W M$ is unital quadratic with $M=\W M(1)$ and $(1,1)\in \W M(2)$ being generators 
in arities $1$ and $2$ respectively. We note that the same construction applies when $M$ is a semigroup, although in this case $\W M$ may not be quadratic. 
Despite its apparent simplicity, the structure of a word operad turns out to be quite versatile. For numerous examples of combinatorial operads arising as suboperads of word operads the reader is referred to \cite{giraudo}. 

Now, let $(\oC, \otimes)$ be a symmetric monoidal category and $M$ be an arbitrary object of $\oC$. Then the  \emph{$M$-associative} operad $\A M$ \cite{dotsenko-shadrin-vallette}, or the \emph{$M$-interstice operad} as per \cite{combe-giraudo}, is the non-$\Sigma$ operad in $\oC$ defined by setting $\A M(n):=M^{\otimes (n-1)}$ for all $n\geq 2$. The partial composition~${\circ_i: M^{\otimes (n-1)}\otimes M^{\otimes (m-1)}
\to M^{\otimes (n+m-2)}}$ identifies $\A M(m)=M^{\otimes (m-1)}$ with the factor of~${\A M(n+m-2)=M\otimes \dots \otimes\underbrace{M \otimes \dots M}_{m-1}\otimes \dots M}$ 
starting at the index $i$ for $1\leq i\leq n$. 
Stated less formally, this can be described as inserting $M^{\otimes (m-1)}$ into the gap between the $(i-1)$-th and the $i$-th copies of $M$ within $M^{\otimes (n-1)}$ if $2\leq i\leq n - 1$, prepending it to $M^{\otimes (n-1)}$ if~$i=1$ and appending it to the right of $M^{\otimes (n-1)}$ if $i=n$. If $\oC$ is concrete, then $\A M$ can be characterized as the non-symmetric binary quadratic operad generated by the elements of $M$ subject to the relations of the form~$m \circ_1 m'=m'\circ_2 m$ for all $m,m'\in M$. That is, $\A M$ is an operad of mutually compatible binary associative operations. In particular, the ordinary non-$\Sigma$ set-valued associative operad corresponds to the case of $M=\{*\}$ in $\Sets$. 

While non-symmetric (or \textit{non-$\Sigma$}) in the sense of the common terminology, an $M$-associative operad in a concrete category $\oC$ comes with two mutually commuting group actions on $\A M(n)$ for all $n\geq 2$: these are the diagonal action of $Aut(M)$ and the action of $\mathbb{Z}_2$ that reverses an element of $\A M(n)$ considered as a word of length $n-1$. Such a $\mathbb{Z}_2$-action $\#: \A M\to \A M$ is subject to the equivariance condition that reads 
\begin{align}
\label{AMInvol}
(v\circ_i w)^\#=v^\#\circ_{n-i+1}w^\#
\end{align}
 for all~$m, n\geq 2$, $v\in \A M(m)$,~ $w\in \A M(n)$, $1\leq i\leq m$.
\section{Lattice operads}
\label{LatOps}
In what follows, by a \emph{lattice operad} we mean an operad (possibly, a non-$\Sigma$ operad) in the category $\Lat$ with one of the symmetric monoidal structures described above. 
By default, the results presented below are stated for $\Lat$ with the cartesian product $\times$. Indeed, by virtue of Lemma~\ref{LaxMono} and the subsequent observations, an operad in $(\Lat, \dtimes)$, $(\Lat, \utimes)$ or $(\Lat, \sqcup)$  gives rise to an operad in $(\Lat, \times)$ that captures the same compositional structure.

Given a lattice $L$, the natural partial order on $L$ is defined, as usual, via~${a\leq b \Leftrightarrow a = a\wedge b}$ for all $a, b\in L$. In particular, any homomorphism $f: L\to K$ of lattices is monotonic in the sense that~${a\leq b \Rightarrow f(a)\leq f(b)}$ for all $a, b\in L$. This leads to an immediate observation that for a lattice operad $\oP$, the $\circ_i$-compositions are monotonic in each argument. That is,
   for any $m, n\geq 1$, $1\leq i\leq m$, $a, b\in \mathcal{P}(m)$ and $c\in \mathcal{P}(n)$  such that $a\leq b$,
    \[
    a\circ_i c \leq b\circ_i c,
    \] 
    Indeed, given such $a, b, c$, by definition of the monoidal structure on $(\Lat, \times)$, we have $(a, c) \leq (b, c)$ in~${\oP(m)\times \oP(n)}$.
    Since $\circ_i: \oP(m)\times \oP(n) \to \oP(m+n-1)$ is a lattice homomorphism and, hence, monotonic,  the comparison follows.
    Similarly, for any $1\leq j\leq n$, we have
    \[
    c\circ_j a\leq c\circ_j b.
    \]   
  Furthermore, the partial compositions are distributive in the following sense.
\begin{lemma}
 \label{MonoDist}
  Let $\oP$ be a lattice operad. 
  Then  
   \begin{subequations}
   \begin{align}   
    \label{LatDist1}
    a\circ_i (b\wedge c) &= (a\circ_i b)\wedge (a\circ_i c)\\
    a\circ_i (b\vee c) &= (a\circ_i b)\vee (a\circ_i c)    
    \end{align} 
    \end{subequations}
    for all $m, n\geq 1$, $1\leq i\leq m$, $a\in \mathcal{P}(m)$, $b, c\in \mathcal{P}(n)$,
    and
    \begin{subequations}
    \begin{align}
    \label{LatDist2}
    (a\wedge b)\circ_i c = (a\circ_i c)\wedge (b\circ_i c)\\
    (a\vee b)\circ_i c = (a\circ_i c)\vee (b\circ_i c)
    \end{align}
    \end{subequations}
    for all $m, n\geq 1$, $1\leq i\leq m$, $a, b\in \mathcal{P}(m)$, $c\in \mathcal{P}(n)$. 
\end{lemma}

\begin{proof}
  We will derive the first of the listed equalities. The remaining three are analogous.  
  To this end, we note that since partial compositions are lattice homomorphisms, then
  \[
  \circ_i((p,q) \wedge (r,s))=\circ_i(p,q)\wedge \circ_i(r,s)
  \]
  for any $p,r\in\oP(m)$, $q,s\in \oP(n)$ and $1\leq i\leq m$.
  By definition of the cartesian product in $\Lat$, we have ~${(p,q) \wedge (r,s)=(p\wedge r, q\wedge s)}$ and thus,
  \begin{align}
  \label{LatOpDefMeet}
   (p\wedge r) \circ_i (q\wedge s) = (p\circ_i q)\wedge (r\circ_i s)
  \end{align}
  Upon taking $p=r=a$, $q=b$, $s=c$, we get the desired equality.
\end{proof}
While monotonicity of partial compositions by no means suffices, in general, to establish that an operad $\oP$ with each component $\oP(n)$ being a lattice is an operad in $\Lat$, in practice, verification of the defining relation \eqref{LatOpDefMeet} and its dual form 
  \begin{align}
  \label{LatOpDefJoin}
   (p\vee r) \circ_i (q\vee s) = (p\circ_i q)\vee (r\circ_i s)
  \end{align}
for the join can be slightly simplified by means of the following elementary observation. 
\begin{lemma}
\label{WeakDist}
Let $\oP$ be a set-valued operad such that for each $n\geq 1$, $\oP(n)$ is a lattice, all partial compositions $\circ_i: \oP(m)\times \oP(n) \to \oP(m+n-1)$ are monotonic in each argument with respect to the natural partial orders on both the source and the target, and 
\begin{subequations}
   \begin{align}   
    \label{LatOpIneqMeet}
    (p\wedge r) \circ_i (q\wedge s) \geq (p\circ_i q)\wedge (r\circ_i s)\\
    \label{LatOpIneqJoin}
    (p\vee r) \circ_i (q\vee s) \leq (p\circ_i q)\vee (r\circ_i s)
    \end{align} 
    \end{subequations}
 for all $p,r\in\oP(m)$, $q,s\in \oP(n)$ and $1\leq i\leq m$
Then $\oP$ is a lattice operad
\end{lemma}
\begin{proof}
By monotonicity of $\circ_i$ in each argument,
$$
(p\wedge r) \circ_i (q \wedge s) \leq p\circ_i (q \wedge s) \leq p\circ_i q
$$
and
$$
(p\wedge r) \circ_i (q \wedge s) \leq r\circ_i (q \wedge s) \leq r\circ_i s,
$$
whence the opposite of \eqref{LatOpIneqMeet} follows, yielding \eqref{LatOpDefMeet}.
The argument concerning compatibility with joins is analogous. 
\end{proof}

\begin{definition}
\label{LaxDef}
 A \emph{lax morphism} of lattice operads $f:K\to L$
 is a family of lattice homomorphisms~${f_n:K(n)\to L(n)}$ for all $n\geq 1$, which are required to be $\S_n$-equivariant if $K$ and $L$ are symmetric, such that
 \reqnomode
 \begin{align}
 \label{LaxMorph}
   f(a)\circ_i f(b)\leq f(a\circ_i b)
 \end{align}
 for all $m, n\geq 1$, $1\leq i\leq m$, $a\in K(m)$, $b\in K(n)$.
\end{definition}
% For either one of the symmetric monoidal structures on $\Lat$ introduced above, lattice operads together with lax morphisms form a category that we denote as $\LatOper$, tacitly keeping track on the corresponding symmetric monoidal structure on $\Lat$. 
Lattice operads together with lax morphisms form a category that we denote $\LatOper$.
Indeed, we have the following 
\begin{lemma}
Let $K, L, M$ be lattice operads.
The composition of lax morphisms $K\overset{f}{\to} L\overset{g}{\to} M$ is a lax morphism.
\end{lemma}
\begin{proof}
Let $m, n \geq 1$, $1 \leq i \leq m$ and $a \in K(m)$, $b\in K(n)$.
Since $g$ is lax, we have
\begin{align*}
 (g\circ f)(a) \circ_i (g\circ f)(b) = 
 g(f(a)) \circ_i g(f(b)) \leq g(f(a)\circ_i f(b)). 
\end{align*}
Since $f(a)\circ_i f(b)\leq f(a\circ_i b)$ due to $f$ being lax, and $g$ on the right-hand side is monotonic as a lattice homomorphism $L(m+n-1)\to M(m+n-1)$, then
\[
g(f(a)\circ_i f(b)) \leq g(f(a\circ_i b)) = (g\circ f)(a\circ_i b)
\]
completing the proof.
\end{proof}
The premise of condition \eqref{LaxMorph} and, accordingly, the definition of $\LatOper$ will become more apparent in the context of Theorem \ref{FiltFuncL}.
Meanwhile, we would note that while any morphism of operads in $\Lat$ in the ordinary sense is lax, the converse is not true; cf. remark \ref{TotEx} below for an example.

Some ways of constructing new lattice operads from the given ones are as follows. First, recall that for a lattice~$L$, the poset of intervals $Int(L)$ carries a natural lattice structure with $[a,b]\wedge [c,d]:=[a\vee c, b \wedge d]$ and 
${[a,b]\vee [c,d]:=[a\wedge c, b \vee d]}$. 
\begin{lemma}
Let $\oP$ be a lattice operad. Then
$Int(\oP):=\{Int(\oP)(n)\}_{n\geq 1}$ with the partial compositions
$$[p, r]\circ_i [q,s]:=[p\circ_i q, r\circ_i s]$$
for $[p,r]\in Int(\oP)(m)$,
$[q,s]\in Int(\oP)(n)$ and $1 \leq i\leq m$
is a lattice operad as well.
\end{lemma}
\begin{proof}
A repeated application of Lemma \ref{MonoDist}
leads to a direct verification of the defining conditions \eqref{LatOpDefMeet}, \eqref{LatOpDefJoin}.
\end{proof}
Similarly, we get the following
\begin{lemma}
Let $\oP$ be an operad in the category of posets and monotonic maps. Then $Ord(\oP)=\{Ord(\oP)(n)\}_{n\geq 1}$, where $Ord(P)$ denotes the lattice of order ideals in a poset $P$, is a lattice operad with the partial compositions 
\[
I \circ_i J:=\{x\in \oP(m+n-1)| x\leq p\circ_i q\text{ for some }p\in \oP(m), q\in \oP(n)\}
\]
for $I\in Ord(\oP)(m)$, $J\in Ord(\oP)(n)$ and $1\leq i\leq m$.
\end{lemma}
Finally, we note that any lattice operad $\oP$ gives rise to the dual lattice operad $\oP^{op}$ with $\oP^{op}(n)$ be the lattice dual to $\oP(n)$ and the same partial compositions as in $\oP$ at the level of sets. In particular, a lattice operad $\oP$, where each lattice $\oP(n)$ is orthocomplemented, admits an automorphism, induced by the orthocomplements, as an operad in $\Sets$.  
\section{Examples}
\label{LatOpEx}
In the current and the subsequent section we present examples of lattice operads.
Some of the example have emerged earlier in studies of filtrations of operads \cite{superbig, dotsenko, khoroshkin}, while others arise upon endowing some classical lattices of combinatorial origin by operadic structures with respective symmetries.
\begin{example}
\label{LZEx}
  Let $C\bbZ=\{C\bbZ(n)\}_{n\geq 1}$ be defined by $C\bbZ(n):=\mathbb{Z}$ with the usual order (thus $\wedge$, $\vee$ are the minimum and the maximum respectively) and the trivial action of the symmetric group $\S_n$ for all $n\geq 1$. Define $a\circ_i b:=a+b$ for all $a\in C\bbZ(m)$, $b\in C\bbZ(n)$ and $1\leq i\leq m$. 
  Then $C\bbZ$ is a lattice operad.
  %\item The multiindex operad $M\mathbb{Z}=\{M\mathbb{Z}(n)\}_{n\geq 1}$ in the sense of %\cite{superbig} is a lattice operad.

  \label{CIEx}
  More generally, given a non-empty set $I$, we define the \textit{counting operad $C\bbZ^I$ with respect to $I$} by setting $C\bbZ^I(n):=\bbZ^I$ with the natural lattice structure inherited from $\bbZ$ as before, the trivial $\S_n$-action and the partial compositions defined 
   by setting $(a_p)_{p\in I} \circ_i (b_p)_{p\in I}:=(a_p + b_p)_{p\in I}$ for all $a=(a_p)_{p\in I}\in C\bbZ^I(n)$, $b=(b_p)_{p\in I}\in C\bbZ^I(m)$, $m, n\geq 1$ and $1 \leq i\leq m$. 
  The terminology is to be elucidated later, cf. example \ref{GenCntEx}.
  The lattice operad $C\bbZ^I$ can be identified with the $I$-fold cartesian product of the lattice operad $C\bbZ$ with itself.
\end{example}

\begin{example}
\label{LatOpExMZ}
   For $n\geq 1$, let $M\mathbb{Z}(n)$ be $\mathbb{Z}^n$ with the product order inherited from $\mathbb{Z}$.
   For each $n\geq 1$, $\MZ(n)$ is a lattice with the join and the meet operations being
\[
  (\rada {p'_1}{p'_n})\vee (\rada {p''_1}{p''_n}) :=  (\rada {\max(p'_1, p''_1)}{\max(p'_n, p''_n)})
\]
and
\[
  (\rada {p'_1}{p'_n})\wedge (\rada {p''_1}{p''_n}) :=  (\rada {\min(p'_1, p''_1)}{\min(p'_n, p''_n)})
\]
respectively. 
With the natural permutation action of $\S_n$ and the partial compositions 
\[
\circ_i: \MZ(m)
\times \MZ(n) \to \MZ(m+n-1),
\ n \geq 1,\ 1 \leq i \leq m, 
\] 
defined for 
$(\Rada a1m) \in \MZ(m)$ and $(\Rada b1n) \in \MZ(n)$ by
\[
(\Rada a1m) \circ_i (\Rada b1n) := (\Rada a1{i-1},
\rada {b_1 + a_i}{b_n + a_i},
\Rada a{i+1}m),
\]
the collection $\MZ=\{\MZ(n)\}_{n\geq 1}$ makes up a unital lattice operad of \textit{multi-indicies}.
%which is meant to be reflected in the notation. Otherwise, the operad can be readily seen to be the word operad $\W \mathbb{Z}$ for a totally ordered additive monoid of integers  $\mathbb{Z}$.
Explicit descriptions of components of certain lattice suboperads of $\MZ$ in small arities are presented in the appendix. 
\end{example}

\begin{remark}
\label{TotEx}
 Let $C\bbZ$ and $M\mathbb{Z}$ be as above.
 The mapping $\rho:C\bbZ\to M\mathbb{Z}$ defined by setting
 \[
  \rho(k):=\underbrace{(k,\dots,k)}_{n}
 \]
 for all $n\geq 1$ and $k\in C\bbZ(n)=\bbZ$ is a lax morphism of lattice operads. 
 To see that it is not, in fact, an operad morphism in the ordinary sense, take $1\in C\bbZ(2)$ and observe that $\rho(1)\circ_1 \rho(1)=(1,1)\circ_1 (1,1)=(2,2,1)$, while $\rho(1\circ_1 1)=\rho(\underbrace{1 + 1}_{\in C\bbZ(3)})=(2,2,2).$
\end{remark}

\begin{example}
\label{SubLatOp}
  Let $\mathcal{P}$ be a $\bbk$-linear operad. 
  For each $n\geq 1$, take $Sub(\mathcal{P})(n)$ to be the lattice of all $\bbk$-linear subspaces of $\mathcal{P}(n)$ with the $\S_n$-action induced from $\mathcal{P}(n)$.
  Given $m, n\geq 1$, $1\leq i\leq m$, and two subspaces $V\in Sub(\mathcal{P})(m)$, $U\in Sub(\mathcal{P})(n)$, we define
   $V\circ_i U$ to be the span of $\{a\circ_i b|a\in V, b\in U\}$. 
   Then $Sub(\mathcal{P}):=\{Sub(\mathcal{P})(n)\}_{n\geq 1}$ is a lattice operad.  
\end{example}

\begin{example} [V.Dotsenko, A.Khoroshkin]
   \label{InvOp}
   Let $\Free(\{a,b\})$ be the free symmetric operad in $\Sets$ on two generators $a, b$ of arity $2$. An element $T\in\Free(\{a,b\})(n)$ can be identified with a rooted (planar) binary tree with $n$ leaves and $n-1$ internal vertices labeled by $a$ and $b$. 
   Define the degree $\deg(T)$ to be the number of $a$'s among the labels,  and the number of inversions $\inv(T)$ to be the number of pairs $(v_1, v_2)$ of internal vertices of $T$ such that vertex $v_1$ is labeled by $b$, vertex $v_2$ is labeled by $a$ and $v_2$ is a descendant of $v_1$.    
   
   \begin{center}
   \begin{figure}[H]
   \centering
   \subcaptionbox{$\deg(T)=3, \inv(T)=2$}[.48\textwidth]
   {
    \begin{tikzpicture}[circle, minimum size=2pt, sibling distance=2.4cm, level distance=0.6cm]
    \node [fill=red!40, draw]{$a$}
    child {node {}}
    child {node [fill=cyan!40, draw]{$b$}
      child {node {}}
      child {node [fill=red!40, draw]{$a$}
        child {node [fill=red!40, draw] {$a$}
            child {node {}}
            child {node {}}
        }
        child {node {}}
      }
    };
    \end{tikzpicture}
    }
    %\includegraphics[width=2in]{tree1.eps}
    %\caption{$\deg(T)=3, \inv(T)=2$}
    \subcaptionbox{$\deg(T)=2, \inv(T)=4$}[.48\textwidth]
    {
    \begin{tikzpicture}[circle, minimum size=2pt, sibling distance=2.4cm, level distance=0.6cm]    
    \node [fill=cyan!40, draw]{$b$}
    child {node [fill=cyan!40, draw]{$b$}
        child {node [fill=red!40, draw]{$a$}
            child {node [fill=red!40, draw] {$a$}
            child {node {}}
            child {node {}}
            }
            child {node {}}
        }
        child {node {}}
    }
    child{node {}};
    \end{tikzpicture}
    }
   \end{figure}  
   \end{center}
   \noindent
    For each $n\geq 2$, define
    \[
    \Inv(n) := \{(\deg(T), \inv(T))\in \mathbb{N}_0\times \mathbb{N}_0| T\in \Free(\{a,b\})(n)\}.
    \]
    As a set, $\Inv(n)$ consists of all integer tuples $(d,v)$ such that $0 \leq d \leq n - 1$ and $0 \leq v \leq (n - d - 1)\cdot d$. The figures above depict some tree representing elements $(3, 2)$ and $(2,4)$ from $\Inv(5)$. 
    In particular, figure (B) illustrates how a tree $T$ with the maximum value of $\inv(T)$ for a fixed $\deg(T)$ can be constructed. For each $n\geq 2$, we equip $\Inv(n)$ with the lexicographic order and the trivial $\S_n$-action.
    
    Now, for $m, n\geq 2$, $(d, v)\in \Inv(n)$ and $(e, w)\in \Inv(m)$, the partial compositions in $
    \Inv$ are defined by setting
    \[
      (d, v)\circ_i (e, w) := (d + e, \max\limits_{T',T''}(\inv(T'\circ_i T'')))
    \]        
    for $1 \leq i\leq m$. Here, the maximum is taken over all trees $T'\in \Free(\{a,b\})(n)$, $T''\in \Free(\{a,b\})(m)$
    such that $\deg(T')=d$, $\inv(T')=v$ and $\deg(T'')=e$, $\inv(T'')=w$. That makes $\Inv=\{\Inv(n)\}_{n\geq 2}$ a lattice operad, which is actually a totally ordered set in each arity.
\end{example}

\begin{example}
\label{TamariOp}
Let $PBT$ be the free non-$\Sigma$ operad in $\Sets$ on a single binary generator. 
That is, for ${n\geq 2}$, $PBT(n)$ can be identified with the set of all planar rooted binary trees with $n$ leaves, while the partial composition $T'\circ_i T''$ for $T'\in PBT(m)$, $T''\in PBT(n)$, $1\leq i\leq m$ amounts to grafting the tree $T''$ onto the $i$-th leaf of $T'$ by its root.

Recall that a (left) \textit{rotation} in a binary tree $T$ is a transformation of the form
\begin{figure}[H]
 \begin{tikzpicture}[circle, minimum size=2pt, sibling distance=2cm, level distance=0.8cm]
    \node (A) at (0, 0) [fill=red!40, draw]{}    
    child {node[circle, dashed, draw] {A}}
    child {
    node [fill=red!40, draw] {}
    child {node [circle, dashed, draw] {B}}
    child {node [circle, dashed, draw]{C}}
    };
    \node (B) at (6, 0) [fill=red!40, draw]{}    
    child {
    node [fill=red!40, draw] {}
    child {node [circle, dashed, draw] {A}}
    child {node [circle, dashed, draw]{B}}
    }
    child {node[circle, dashed, draw] {C}}
    ;
    \draw [->] (2.8,-0.7) -- (3.8,-0.7);
    \end{tikzpicture}
\end{figure}
\noindent
performed on a subtree $S$ of $T$, where $A, B$ and $C$ are subtrees of $S$.
By taking rotation as the cover relation, the transitive closure gives $PBT(n)$ a lattice structure, known as the \textit{Tamari lattice of order $n-1$}.
Since a rotation in a tree $T$ changes only the inner vertices structure of $T$ and does not affect the leaves, it commutes with grafting. 
In particular, grafting as the partial composition $(-\circ_i-)$ respects the cover relation in both arguments, making $PBT$ a poset-valued operad. In fact, a stronger statement holds and we show that the partial compositions are compatible with the Tamari lattice structures. That is, $PBT$ is a lattice operad in the sense of our definition.

To this end, we employ the description of the Tamari lattice in terms of the \textit{weight sequences} of binary trees \cite{pallo}.
Namely, given a binary tree $T\in PBT(n)$
 with the leaves enumerated from left to right by $1,2\dots n$, the weight sequence of $T$ is the vector 
$(w_T(1), w_T(2),\dots, w_T(n))$, where $w_T(i)$ is the number of leaves of the largest subtree of $T$ that has leaf $i$ as its rightmost leaf. Such a subtree may be trivial, that is, consisting of a single leaf and no internal vertices, in which case $w_T(i)=1$. 
  \begin{figure}[H]
  {
   \begin{tikzpicture}[circle, minimum size=2pt, 
   level distance=1cm,
   level/.style={sibling distance=4cm/#1}
   ]
    \node [fill=red!40, draw]{5}    
    child 
    {
        node [fill=red!40, draw] {2}
        child
        {
            node [fill=red!40, draw] {1}
            child {}
            child {}
        }
        child 
        {
            node [fill=red!40, draw] {2}
            child
            {
                node [fill=red!40, draw] {1}
                child {}
                child {}
            }
            child 
            {
                node{}
            }
        }
    }
    child
    {
        node [fill=red!40, draw] {1}
        child {node {}}
        child {
                node [fill=red!40, draw] {1}
                child {node {}}
                child {node {}}
        }
    };
    \end{tikzpicture}
  }   
  \caption{A tree $T\in PBT(8)$ with the weight sequence $w_T=(1,2,1,2,5,1,1,8)$.}
  \label{WeightTree}
  \end{figure}
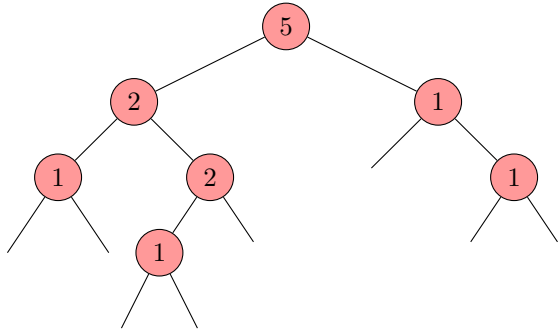
  Algorithmically, the weight sequence of a binary tree $T$ can be computed as follows.
  First, label each vertex $v$ of $T$ by the total number of leaves of the left subtree of $v$, as illustrated in the figure above. Then the weight sequence $w_T$ can be read off from the vertex labels recursively by performing an in-order traversal on $T$ starting from its root: process the left subtree, read the label at the current root and append it to $w_T$, process the right subtree. On a tree with $n$ leaves this procedure results in a sequence of $n-1$ vertex labels placed in a certain order. Additionally, the total number of leaves $n$ is appended to the end of the sequence as $w_T(n)=n$.
  \begin{remark}
The original definition of a weight sequence in \cite{pallo} does not have the final entry $w_T(n)=n$ that we include for our later convenience. Slight variations of this integer vector parametrization of planar rooted binary trees can be found in \cite{HuangTamari, Geyer}.
\end{remark}

  The set $W(n)$ of the weight sequences for all planar rooted binary trees with $n$ leaves can be characterized~\cite[Theorem 1]{pallo} as the set of all positive integer sequences $(w_1,w_2,\dots, w_n)$ such that 
\begin{align}
\label{SlopeCond}
i-j\leq w_i-w_j
\end{align}
 for all $i < n$, ${i - w_i + 1\leq j < i}$, and $w_n=n$.
 Note that the condition $w_n=n$ leads to $w_j\leq j$ for all $1\leq j\leq n$; it follows upon taking ${i=n}$ in the above inequality.
 In particular, this implies $w_1=1$.
 
  The assignment \( T \mapsto w_T \) is bijective. Specifically, given a weight sequence \( w_T = (w_1, w_2, \dots, w_{n-1}, n) \), the corresponding tree \( T \in PBT(n) \) can be recursively reconstructed by identifying its left and right subtrees \( T' \) and \( T'' \), and joining them at the root. The subtrees \( T' \) and \( T'' \) are recovered from their respective weight sequences \( w_{T'} \) and \( w_{T''} \), which are embedded within \( w_T \). These can be obtained by locating the largest entry \( w_r \) in \( w_T \) other than \( n \), which is guaranteed to be unique, and setting \( w_{T'} = (w_1, \dots, w_r) \) and \( w_{T''} = (w_{r+1}, \dots, w_{n-1}, n-r) \).

  \begin{figure}[H]
  {
   \resizebox{0.45\textwidth}{!} {\begin{tikzpicture}[circle, minimum size=2pt, 
level distance=1cm,
level/.style={sibling distance=4cm/#1},
emph/.style={edge from parent/.style={dashed,draw}},
norm/.style={edge from parent/.style={solid,draw}}
]
\node [fill=red!40, draw]{\phantom {\small 5}}    
child[emph]
{   
node [fill=red!40, draw] {2}
child[norm]
{
node [fill=red!40, draw, solid] {1}
child {}
child {}
}
child[norm] 
{
node [fill=red!40, draw, solid] {2}
child
{
node [fill=red!40, draw] {1}
child {}
child {}
}
child 
{
node{}
}
}
}
child[emph]
{
node [fill=red!40, draw] {1}
child[norm] {node {}}
child[norm] {
node [fill=red!40, draw, solid] {1}
child {node {}}
child {node {}}
}
};
\end{tikzpicture}}  
  }   
  \caption{The tree decomposition for $w_T=(1,2,1,2,5,1,1,8)$, $w_{T'}=(1,2,1,2,5)$, ${w_{T''}=(1,1,3)}$}
  \label{WeightTree2}
  \end{figure}
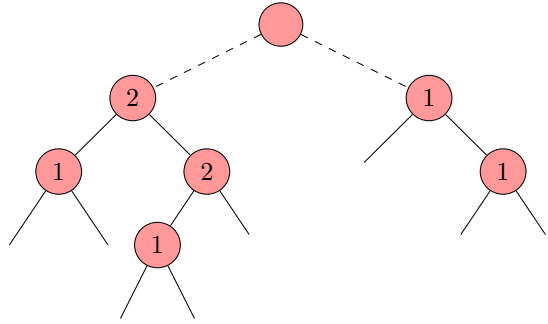
In terms of weight sequences, grafting of planar binary trees is expressed as follows.
  \begin{proposition}
  \label{WScomp}
  Let $u=(u_1,u_2,\dots, u_m)$ and $v=(v_1,v_2,\dots, v_n)$ be the weight sequences of trees ${S\in PBT(m)}$ and ${T\in PBT(n)}$ respectively.
  Then for $1\leq i\leq m$, the weight sequence of $S\circ_i T$ is
  \begin{align}
  \label{WPartComp}
  w = u \circ_i v:= (u_1,\dots, u_{i-1}, \underbracket[0.1ex]{v_1,\dots, v_{n-1},
  v_n+u_{i}-1}, u_{i+1}',u_{i+2}',\dots, u_{m}'),
  \end{align}
  where $u_j'=
  \begin{cases}
  u_j + n - 1, \quad & u_j \geq j-i+ 1\\
  u_j,\quad &\text{otherwise}
  \end{cases}$.  
  \end{proposition}  
  \begin{figure}[H]
  \resizebox{0.5\textwidth}{!} {\begin{tikzpicture}[circle, minimum size=2pt, sibling distance=2cm, level distance=0.7cm]
    \node [fill=red!40, draw] at (0, 0) {1} 
     child
    {
     node  {} 
    }
    child 
    {
        node [fill=red!40, draw] {2}
        child
        {
            node [fill=red!40, draw] {1}
            child {}
            child {}
        }
        child 
        {
            node  {}            
        }
    };
    \node at (2, -0.5) {$\circ_3$};
    \node [fill=red!40, draw] at (4, 0) {2} 
     child
        {
            node [fill=red!40, draw] {1}
            child {}
            child {}
        }
        child 
        {
            node  {}            
        };
   \node at (5.5, -0.5) {$=$};
   \node [fill=red!40, draw] at (7, 0.8) {1} 
     child
    {
     node  {} 
    }
    child 
    {
        node [fill=red!40, draw] {4}
        child
        {
            node [fill=red!40, draw] {1}
            child {}
            child {
            node [fill=red!40, draw] {2}
            child{
            node [fill=red!40, draw] {1}
            child{}
            child{}
            }
            child{}
            }
        }
        child 
        {
            node  {}            
        }
    };
    \end{tikzpicture}}  
   \caption{Grafting in $PBT$ encoded by $(1,1,2,4)\circ_3 (1,2,3)=(1,1,1,2,4,6)$ in $W$}
  \end{figure}
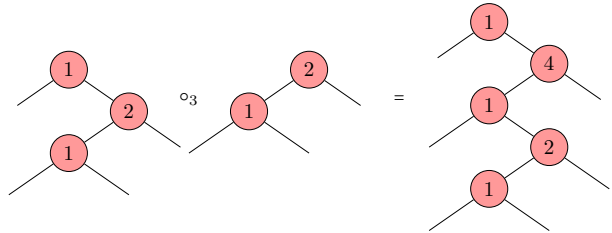
  \begin{proof}    
   We establish the formula for $w_k$ by considering four cases: $k<i$, $i\leq k\leq i+n-2$, $k=i+n-1$ and $k\geq i-n$.
   
   If $k<i$, then the largest subtree $S'$ of $S$ that has leaf $k$ as its rightmost leaf retains this property in $S\circ_i T$, since all the leaves of the grafted copy of $T$ in $S\circ_i T$ have indices in the range $i, i+1, \dots, i+n-1$. Thus, $w_k=u_k$.    
   
   Next, for $i\leq k \leq i+n-2$, note that none of the leaves of $T$, except possibly the $n$-th one, acquires any extra weight in $S\circ_i T$ upon grafting. Indeed, any subtree of $S\circ_i T$ that contains the rightmost leaf $k$ in this range, is contained entirely within the grafted copy of $T$. Thus, $w_{k}=v_{k-i+1}$

   To compute $w_{i+n-1}$, note that if $S'$ is the largest subtree of $S$ that has its rightmost leaf at $i$, then $S'\circ_i T$ is the largest subtree of $S\circ_i T$ that contains the $n$-th leaf of $T$ as its rightmost one upon grafting. Therefore, $w_{i+n-1}$ is equal to the total number of leaves of $S'\circ_i T$. Since the $i$-th leaf of $S$ becomes an internal node in $S'\circ_i T$ upon grafting, then $w_{i+n-1}=|\text{leaves of $T$}|+|\text{leaves of $S'$}|-1=v_n+u_i-1$.

   \begin{center}
  \begin{figure}[H]
  \subcaptionbox{$k<i$}
  {
  \resizebox{0.25\textwidth}{!} {\begin{tikzpicture}[scale=0.6,every node/.style={circle, inner sep=2pt}]
   % Define the side length of the triangles
\def\sidelength{7}
\def\sf{5}
% Coordinates of triangle S vertices
\coordinate (A1) at (0,0);
\coordinate (B1) at (\sidelength/2,3);
\coordinate (C1) at (\sidelength,0);
\coordinate (M1) at (\sidelength/2,2);

% Coordinates of triangle T vertices, with its apex touching the bottom of triangle S
\coordinate (A2) at (2, -3);
\coordinate (B2) at (2+\sidelength/2*0.7, 0);
\coordinate (C2) at (2+\sidelength*0.7,-3);
\coordinate (M2) at (2+\sidelength/2*0.7,-1.4);

\coordinate (A3) at (0.5, 0);
\coordinate (B3) at (0.5+\sidelength/\sf, 6/\sf);
\coordinate (C3) at (0.5+2*\sidelength/\sf, 0);
\coordinate (M3) at (0.5+\sidelength/\sf,0.5);

% Draw triangle S
\draw (A1) -- (B1) -- (C1) -- cycle;
\node at (M1) {$S$};

% Draw triangle T
\draw (A2) -- (B2) -- (C2) -- cycle;
\node at (M2) {$T$};    

% Draw triangle T
\draw[fill=red!30] (A3) -- (B3) -- (C3) -- cycle;
\node at (M3) {$S'$};    

\node[draw, fill=cyan!40, circle] at (B2) {$i$};
\node[draw, fill=red!40, circle] at (C3) {$k$};
\end{tikzpicture}}
  }
  \hspace{0.5cm}
  \subcaptionbox{$i\leq k\leq i+n-2$}
  {
  \resizebox{0.25\textwidth}{!} {\begin{tikzpicture}[scale=0.6,every node/.style={circle, inner sep=2pt}]
   % Define the side length of the triangles
\def\sidelength{7}
\def\sf{5}
% Coordinates of triangle S vertices
\coordinate (A1) at (0,0);
\coordinate (B1) at (\sidelength/2,3);
\coordinate (C1) at (\sidelength,0);
\coordinate (M1) at (\sidelength/2,2);

% Coordinates of triangle T vertices, with its apex touching the bottom of triangle S
\coordinate (A2) at (2, -3);
\coordinate (B2) at (2+\sidelength/2*0.7, 0);
\coordinate (C2) at (2+\sidelength*0.7,-3);
\coordinate (M2) at (2+\sidelength/2*0.7,-1.4);

\coordinate (A3) at (2.5, -3);
\coordinate (B3) at (2.5+\sidelength/\sf, -1.5);
\coordinate (C3) at (2.5+2*\sidelength/\sf, -3);
\coordinate (M3) at (2.5+\sidelength/\sf,-2.5);

% Draw triangle S
\draw (A1) -- (B1) -- (C1) -- cycle;
\node at (M1) {$S$};

% Draw triangle T
\draw (A2) -- (B2) -- (C2) -- cycle;

% Draw triangle T
\draw[fill=red!30] (A3) -- (B3) -- (C3) -- cycle;
\node at (M3) {$S'$};

\node at (M2) {$T$};    
\node[draw, fill=cyan!40, circle] at (B2) {$i$};
\node[draw, fill=red!40, circle] at (C3) {$k$};
\end{tikzpicture}}
  }
  \hspace{0.5cm}
  \subcaptionbox{$k=i+n-1$}
  {
  \resizebox{0.25\textwidth}{!} {\begin{tikzpicture}[scale=0.6,every node/.style={circle, inner sep=2pt}]
   % Define the side length of the triangles
\def\sidelength{7}
\def\sf{5}
% Coordinates of triangle S vertices
\coordinate (A1) at (0,0);
\coordinate (B1) at (\sidelength/2,3);
\coordinate (C1) at (\sidelength,0);
\coordinate (M1) at (\sidelength/2,2);

% Coordinates of triangle T vertices, with its apex touching the bottom of triangle S
\coordinate (A2) at (2, -3);
\coordinate (B2) at (2+\sidelength/2*0.7, 0);
\coordinate (C2) at (2+\sidelength*0.7,-3);
\coordinate (M2) at (2+\sidelength/2*0.7,-1.4);

\coordinate (A3) at (1.7, 0);
\coordinate (B3) at (1.7+\sidelength/\sf, 6/\sf);
\coordinate (C3) at (1.7+2*\sidelength/\sf, 0);
\coordinate (M3) at (1.7+\sidelength/\sf,0.5);

% Draw triangle S
\draw (A1) -- (B1) -- (C1) -- cycle;
\node at (M1) {$S$};

% Draw triangle T
\draw[fill=red!30] (A2) -- (B2) -- (C2) -- cycle;

% Draw triangle T
\draw[fill=red!30] (A3) -- (B3) -- (C3) -- cycle;
\node at (M3) {$S'$};    

\node at (M2) {$T$};    
\node[draw, fill=cyan!40, circle] at (B2) {$i$};
\node[draw, fill=red!40, circle] at (C2) {$k$};
\end{tikzpicture}}
  }  
  \end{figure}
  \end{center}
   To handle the remaining case $k\geq i+n$,
   let $j=k-n+1$ and consider the largest subtree $S''$ of $S$ that has the $j$-th leaf of $S$ as its rightmost one. Then, by definition of $u_j$, all the leaves of $S''$ have indicies spanning the range $j-u_j+1, j-u_j+2,\dots, j$ in $S$.
   Hence, if $j-u_j+1>i$, a copy of $S''$ in $S\circ_i T$ will remain to be the largest subtree, now within $S\circ_i T$, with its rightmost leaf at $k=j+n-1$.
   Otherwise, it will have $T$ attached to it, thus acquiring $n$ extra leaves while losing one as a grafting point. This yields the above formula for~$u_j'=w_{j+n-1}$.
   \begin{center}
  \begin{figure}[H]  
  {
  \resizebox{0.25\textwidth}{!} {\begin{tikzpicture}[scale=0.6,every node/.style={circle, inner sep=2pt}]
   % Define the side length of the triangles
\def\sidelength{7}
\def\sf{5}
% Coordinates of triangle S vertices
\coordinate (A1) at (0,0);
\coordinate (B1) at (\sidelength/2,3);
\coordinate (C1) at (\sidelength,0);
\coordinate (M1) at (\sidelength/2,2);

% Coordinates of triangle T vertices, with its apex touching the bottom of triangle S
\coordinate (A2) at (-0.5, -3);
\coordinate (B2) at (-0.5+\sidelength/2*0.7, 0);
\coordinate (C2) at (-0.5+\sidelength*0.7,-3);
\coordinate (M2) at (-0.5+\sidelength/2*0.7,-1.4);

\coordinate (A3) at (3, 0);
\coordinate (B3) at (3+\sidelength/\sf, 6/\sf);
\coordinate (C3) at (3+2*\sidelength/\sf, 0);
\coordinate (M3) at (3+\sidelength/\sf,0.5);

% Draw triangle S
\draw (A1) -- (B1) -- (C1) -- cycle;
\node at (M1) {$S$};

% Draw triangle T
\draw (A2) -- (B2) -- (C2) -- cycle;
\node at (M2) {$T$};    

% Draw triangle T
\draw[fill=red!30] (A3) -- (B3) -- (C3) -- cycle;
\node at (M3) {$S''$};    

\node[draw, fill=cyan!40, circle] at (B2) {$i$};
\node[draw, fill=red!40, circle] at (C3) {$k$};
\end{tikzpicture}}
  }
  \hspace{0.5cm}
    {
  \resizebox{0.25\textwidth}{!} {\begin{tikzpicture}[scale=0.6,every node/.style={circle, inner sep=2pt}]
   % Define the side length of the triangles
\def\sidelength{7}
\def\sf{5}
% Coordinates of triangle S vertices
\coordinate (A1) at (0,0);
\coordinate (B1) at (\sidelength/2,3);
\coordinate (C1) at (\sidelength,0);
\coordinate (M1) at (\sidelength/2,2);

% Coordinates of triangle T vertices, with its apex touching the bottom of triangle S
\coordinate (A2) at (-0.5, -3);
\coordinate (B2) at (-0.5+\sidelength/2*0.7, 0);
\coordinate (C2) at (-0.5+\sidelength*0.7,-3);
\coordinate (M2) at (-0.5+\sidelength/2*0.7,-1.4);

\coordinate (A3) at (1.2, 0);
\coordinate (B3) at (1.2+\sidelength/\sf, 6/\sf);
\coordinate (C3) at (1.2+2*\sidelength/\sf, 0);
\coordinate (M3) at (1.2+\sidelength/\sf,0.5);

% Draw triangle S
\draw (A1) -- (B1) -- (C1) -- cycle;
\node at (M1) {$S$};

% Draw triangle T
\draw[fill=red!30] (A2) -- (B2) -- (C2) -- cycle;
\node at (M2) {$T$};    

% Draw triangle T
\draw[fill=red!30] (A3) -- (B3) -- (C3) -- cycle;
\node at (M3) {$S''$};    

\node[draw, fill=cyan!40, circle] at (B2) {$i$};
\node[draw, fill=red!40, circle] at (C3) {$k$};
\end{tikzpicture}}
  }
  \caption*{\small{(D)} $k\geq i+n$}
  \end{figure}
  \end{center}
 \end{proof}

  The set $W(n)$ carries the natural product order inherited from $\mathbb{N}^n$. It agrees with the partial order induced by the left rotations of binary trees:
  \begin{lemma}
   Let a binary tree $T''$ be obtained from $T'$ by performing a rotation on some subtree $S$ in $T'$. Then $w_{T'}\leq w_{T''}$ coordinate-wise. 
  \end{lemma}
  \begin{proof}
  Consider a rotation on $S$:
  \begin{figure}[H]
 \begin{tikzpicture}[every node/.style={circle, minimum size=0.8cm, inner sep=0pt}, sibling distance=2cm, level distance=0.8cm]
    \node (A) at (0, 0) [fill=red!40, draw]{$a$}    
    child {node[circle, dashed, draw] {A}}
    child {
    node [fill=red!40, draw] {$b$}
    child {node [circle, dashed, draw] {B}}
    child {node [circle, dashed, draw]{C}}
    };
    \node (B) at (6, 0) [fill=red!40, draw]{$a+b$}    
    child {
    node [fill=red!40, draw] {$a$}
    child {node [circle, dashed, draw] {A}}
    child {node [circle, dashed, draw]{B}}
    }
    child {node[circle, dashed, draw] {C}}
    ;
    \draw [->] (2.8,-0.7) -- (3.8,-0.7);
    \end{tikzpicture}
\end{figure}
 \noindent
  Here, $a$ and $b$ denote the total number of leaves of the subtrees $A$ and $B$ respectively, and we use our vertex labeling convention. Then $w_{T'}$ is of the form
  \[
  w_{T'}=(\dots, \underline{A}, a, \underline{B}, b, \underline{C}, \dots),
  \]
  where $\underline{A}, \underline{B}, \underline{C}$ 
  denote the subsequences of $w_{T'}$ obtained by running the in-order traversal procedure on the subtrees $A$, $B$ and $C$ respectively. The dots denote the entries of $w_{T'}$ for all the leaves not in $S$. Similarly, according to the in-order traversal on $T''$, the sequence $w_{T''}$ takes the form 
  \[
  w_{T''}=(\dots, \underline{A}, a, \underline{B}, a+b, \underline{C}, \dots),
  \]
  where the dots denote the same entries as before. Indeed, the group of dots on the left denotes the entries that have been obtained prior to reaching $S$ during the traversal, would it be before or after the rotation. The group of dots on the right is also not affected by the rotation, since the total number of leaves in $S$ after the rotation remains the same.
  \end{proof}
  Furthermore, $W(n)$ is closed under taking the coordinate-wise minimum $u\wedge v$ of two weight sequences:
  \begin{lemma}
  Let $(u_1,u_2,\dots,u_n)$ and $(v_1,v_2,\dots,v_n)$
  satisfy condition \eqref{SlopeCond}.
  Then so does the sequence ${w_i:=\min(u_i, v_i)}$ for ${1\leq i\leq n}$.
  \end{lemma}
   \begin{proof}
   Let $i<n$. Without loss of generality, we may assume $w_i=\min(u_i,v_i)=u_i$. 
   Let $j$ be in the range $i-w_i+1\leq j<i$. If $u_j\leq v_j$, then $w_j=u_j$, and $w_i-w_j=u_i-u_j\geq i-j$. Otherwise, we have $w_j=v_j\leq u_j$. Then $w_i-w_j=u_i-v_j\geq u_i-u_j\geq i-j$ again.   
   \end{proof}
   Thus, $W(n)$ is a sub-meet-semilattice of $\mathbb{N}^n$ with the product order. At the same time, $W(n)$ is not closed under the natural join on $\mathbb{N}^n$, that being the coordinate-wise maximum of two sequences.    
  To define, following \cite{Geyer}, the join $u\vee v$, we need to introduce a certain normalization operator on positive integer vectors.
  Namely, let ${u=(u_1,u_2,\dots,u_n)\in \mathbb{N}^n}$.   
  Then, inductively, for each $i$ going from $1$ to $n$, set $z_i$ to be the maximum of $u_i$ and $z_j + i - j$ for $i-u_i< j < i$, whenever $j\geq 1$ applies.  By construction, this yields the smallest, in the sense of the product order on $\mathbb{N}^n$, sequence $h(u):=(z_1,z_2,\dots,z_n)$ that satisfies \eqref{SlopeCond} and $u\leq h(u)$. In particular, if $u$ is a weight sequence of a binary tree, then $h(u)=u$. 

   The map $h$ admits the following geometric interpretation.  
  To any positive integer vector $(u_1,u_2,\dots,u_n)$ we  associate a piecewise-linear path on the $xy$-plane with the nodes at $(i,u_i)$ for $1\leq i\leq n$. Geometrically, condition \eqref{SlopeCond} characterizing weight sequences means that for $1\leq i\leq n$, the line segment connecting the node $(i,u_i)$ to any of the $u_i-1$ nodes $(j,u_j)$ immediately preceding $(i,u_i)$ has the slope $\frac{u_i-u_j}{i-j}\geq 1$. 
  Note that this does not preclude the existence of path segments with the zero or negative slope when $u_i=1$.
  The mapping $u \mapsto h(u)$ amounts to adjusting the path by tracing it from left to right and shifting each node up, when necessary, by the minimal distance ensuring that the slope condition holds. 
  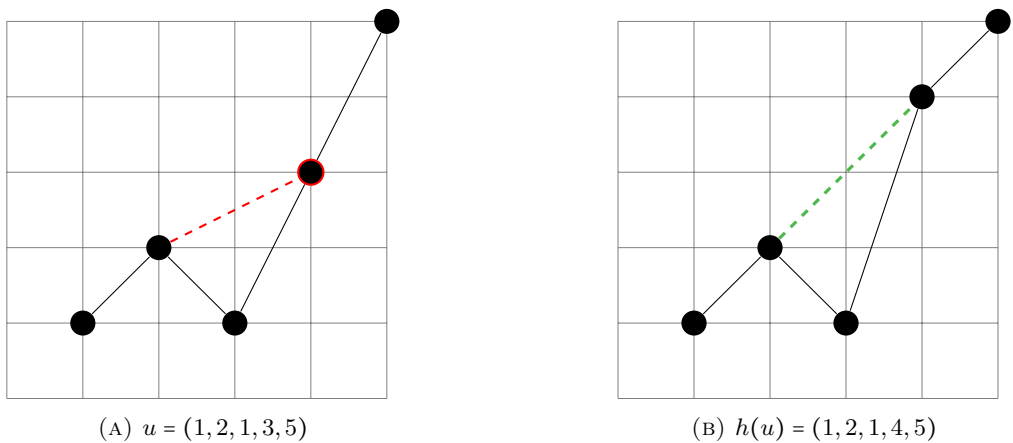
\begin{figure}[H]
   \centering
   \subcaptionbox{$u=(1,2,1,3,5)$}[.48\textwidth]
   {
    \begin{tikzpicture}
    \tikzset{dot/.style={fill=black,circle}}
    \draw[step=1cm,gray,very thin] (0,0) grid (5,5);
    \node[dot] (A1) at (1,1) {};
    \node[dot] (A2) at (2,2) {};
    \node[dot] (A3) at (3,1) {};
    \node[dot, draw=red, thick] (A4) at (4,3) {};
    \node[dot] (A5) at (5,5) {};
    \draw (A1) -- (A2) -- (A3) -- (A4) -- (A5);
    \draw[dashed, red, thick] (A2) -- (A4);
    \end{tikzpicture}
    }
    \subcaptionbox{$h(u)=(1,2,1,4,5)$}[.48\textwidth]
    {
    \begin{tikzpicture}
    \tikzset{dot/.style={fill=black,circle}}
    \draw[step=1cm,gray,very thin] (0,0) grid (5,5);
    \node[dot] (A1) at (1,1) {};
    \node[dot] (A2) at (2,2) {};
    \node[dot] (A3) at (3,1) {};
    \node[dot] (A4) at (4,4) {};
    \node[dot] (A5) at (5,5) {};
    \draw (A1) -- (A2) -- (A3) -- (A4) -- (A5);
    \draw[dashed, green!60!black!70, very thick] (A2) -- (A4);
    \end{tikzpicture}
    }
    \caption{The action of $h$ on a sequence $u$ that is not a weight sequence.}
   \end{figure} 
  
\begin{proposition}
For $n\geq 1$, let 
${L(n):=\{(1,w_2,\dots,w_{n-1},n)|\,1\leq w_i\leq i\}}$.
The collection $L=\{L(n)\}_{n\geq 1}$ is closed under the partial compositions \eqref{WPartComp}.
Furthermore, for $m,n\geq 1$, $a\in L(m)$, $b\in L(n)$ and $1\leq i\leq m$, we have $h(a\circ_i b)=h(a)\circ_i h(b)$

\end{proposition}
\begin{proof}
Let $a\in L(m)$ and $b\in L(n)$ for some $m,n\geq 1$.
Observe that since~$b_n=n$, formula \eqref{WPartComp} can be written in a more uniform way as
\begin{align}
\label{WPartComp2}
w=a \circ_i b:= (a_1,\dots, a_{i-1}, {b_1,\dots, b_{n-1}}, a_i',a_{i+1}',\dots, a_{m}'),
\end{align}
where, as before, $a_j'=a_j+n-1$ if $a_j\geq j-i+1$ and it equals $a_j$ otherwise.
We check that the defining conditions of $L(m+n-1)$ hold for $w$. Indeed, $w_1=a_1=1$, and $w_{m+n-1}=a_m'=m+n-1$, since $a_m=m$. 
Moreover, to verify that $w_k\leq k$, note that for $1\leq k\leq i-1$, $w_k=a_k\leq k$;
for $i\leq k\leq i+n-2$, $w_k=b_{k-i+1}\leq k-i+1\leq k$; finally, for $k\geq i+n-1$, ${w_k=a_{k-n+1}'\leq a_{k-n+1}+n-1\leq (k-n+1)+n-1=k}$.

Now, let $h(a)=(u_1,u_2,\dots,u_m)$, $h(b)=(v_1,v_2,\dots,v_n)$. 
We seek to show that
\begin{align}
\label{hComp}
{h(a)\circ_i h(b)}:={(u_1,\dots,u_{i-1},v_1,\dots,v_{n-1},u_i',\dots,u_{m+n-1}')}
\end{align}
is equal to $h(a\circ_i b)$.
To this end, let $w=a\circ_i b$ be as above.
By the inductive construction of $h$, the $k$-th entry of $h(w)$ depends only on $w_k$ and the $w_k-1$ preceding entries of $h(w)$ with the indices $k-w_k+1,\dots,k-1$.
In particular, evaluating $h$ on the first $i-1$ entries of \eqref{WPartComp2} is the same as evaluating $h$ on the first $i-1$ entries of $a$ yielding $h(w)_k=u_k$ for all $k\leq i-1$.

Similarly, since $b_j\leq j$ for all $1\leq j\leq n$, then for $k$ in the range $i\leq k\leq i+n-2$, the $k$-th entry of $h(w)$ depends only on the entries of $w$ with the indices in the range $i,i+1,\dots,k-1$ at most. These are equal to 
$b_1,b_2,\dots, b_{k-i+1}$ respectively. Hence, evaluating $h(w)$ in this range is the same as evaluating $h$ on the first $k-i+1$ entries of $b$ yielding 
$h(w)_k=v_{k-i+1}$ for all $i\leq k\leq i+n-2$.

To handle the remaining case $k\geq i+n-1$, we proceed by induction.
For the base case $k=i+n-1$, we have ${w_k=a_i'=a_i+n-1}$, since $a_i\geq 1$. 
As established above, the $w_k-1$ terms preceding the $k$-th entry of $h(w)$ are 
$u_{i-a_i+1},\dots, u_{i-1},v_1,\dots,v_{n-1}$. 
By construction of $h$, the $k$-th entry of $h(w)$ is the greatest of
\[
u_{i-a_i+1}+k-(i-a_i+1), \dots, u_{i-1}+k-(i-1),v_1+n-1,\dots,v_{n-1}+1
\]
and $a_i+n-1$. Since $v_j+(n-j)\leq n$ for all $j$, and the greatest value is at least $a_i+n-1\geq n$, the last $n-1$ entries in the list above can be dropped. Thus, the greatest value is to be found among
\[
u_{i-a_i+1}+(a_i-1)+(n-1),\dots, u_{i-1}+1+(n-1)
\]
and $a_i+n-1$. By definition of $h(a)$, this value is equal to $u_i+(n-1)$, or equivalently, $u_i'$, since $u_i\geq 1$. That is, the $k$-th entry of $h(w)$ matches the $k$-th entry of $h(a)\circ_i h(b)$, as desired.

We consider now $k>i+n-1$ under the inductive assumption that the first $k-1$ entries of $h(w)$ match coordinate-wise the first $k-1$ entries of \eqref{hComp}.
Let $j=k-n+1$. Then $w_k=a_j'$, and there are two cases to consider: $a_j<j-i+1$ and $a_j\geq j-i+1$. 

\textit{Case 1.} If $a_j<j-i+1$, then $a_j'=a_j$. Hence, invoking the inductive assumption, the $a_j-1$ entries preceding the $k$-th entry of $h(w)$ are
\[
u_{j-a_j+1}',u_{j-a_j+2}',\dots, u_{j-1}'.
\]
Then $h(w)_k$ is equal to the greatest value among
\begin{align}
\label{u_primes}
u_{j-a_j+1}'+a_j-1,u_{j-a_j+2}'+a_j-2,\dots, u_{j-1}'+1
\end{align}
and $a_j$. If for all $j-a_j+1\leq t \leq j-1$, $u_{t}<t-i+1$, then $u_t'=u_t$ for all such $t$, and the $k$-th entry of $h(w)$ is the greatest value among
\[
u_{j-a_j+1}+a_j-1,u_{j-a_j+2}+a_j-2,\dots, u_{j-1}+1
\]
and $a_j$. By definition of $h(a)=(u_1,u_2,\dots,u_{m})$, this is the same as $u_j$. Due to 
$u_{t}+(j-t)<t-i+1+(j-t)=j-i+1$, we know that $u_j<j-i+1$. Hence $u_j=u_j'$ establishing that $h(w)_k=u_j'$, as needed.

Now, if $u_{t}\geq t-i+1$ for some $j-a_j+1\leq t \leq j-1$, then $u_t'=u_t+(n-1)$, and the greatest entry of \eqref{u_primes} is guaranteed to be of the form $u_t'+(j-t)=(u_t+j-t)+(n-1)$.
By definition of $h(a)$, this is $u_j+(n-1)$, which in its turn is equal to $u_j'$ yielding $h(w)_k=u_j'$.

\textit{Case 2.} Assume now that $a_j\geq j-i+1$. By the inductive assumption, the $w_k-1$ entries of $h(w)$ preceding the $k$-th one are
\[
 u_{n-a_j+j},\dots, u_{i-1},v_1,\dots,v_{n-1},u_i',\dots u_{j-1}',
\]
since $w_k=a_j'=a_j+n-1$.
Then, by the same argument as in the base case, the $v$-entries can be dropped off, and $h(w)_k$ is equal to the greatest value among
\begin{align}
\label{u_collected}
u_{n-a_j+j}+(a_j-n)+(n-1),\dots, u_{i-1}+(j-i+1)+(n-1), u_{i}'+(j-i),\dots u_{j-2}'+2,\dots u_{j-1}'+1
\end{align}
and $a_j+n-1$.
If there is any term $u_t'$ in this list such that $u_t'=u_t$, that is, if $u_t<t-i+1$, then 
$$u_t'+(j-t)=u_t+(j-t)<(t-i+1)+(j-t)=j-i+1\leq a_j.$$ 
In particular, since this implies $u_t+(n-1)+(j-t)< a_j+n-1$, we can replace $u_t'$ by $u_t+(n-1)$ in \eqref{u_collected} without any loss of generality, because the greatest element we are after is at least $a_j+n-1$. Therefore, the greatest entry among \eqref{u_collected} and $a_j+n-1$ is the one among
\begin{align}
u_{n-a_j+j}+(a_j-n)+(n-1),\dots, u_{j-1}+1+(n-1)
\end{align}
and $a_j+(n-1)$. By definition of $h(a)=(u_1,\dots,u_m)$, this equals $u_j+(n-1)$ or, equivalently, $u_j'$, since $u_j\geq a_j\geq j-i+1$.

\end{proof}
   
The join of weight sequences $u,v\in W(n)$
is defined by applying $h$ to the join of $u$ and $v$ in the ambient lattice $L(n)$, that is, $u\vee v:=h((\max(u_1,v_1),\max(u_2,v_2),\dots,\max(u_n,v_n)))$.  
The fact that it indeed defines a join on $W(n)$ follows from the following
\begin{lemma}
Let $(L, \vee_L)$ be a (join-)semilattice and $K \subset L$ be a subposet of $L$. If $h: L\to L$ is an order-preserving map such that $h|_{K}=id_K$ and $h(L)=K$, then $x \vee_K y:=h(x \vee_L y)$, for $x, y\in K$, defines a join-semilattice structure on $K$.
\end{lemma}
\begin{proof}
Let $x,y\in K$. First, we show that $x \vee_K y=h(x \vee_L y)$ is an upper bound for $x$ and $y$ in $K$. Indeed, since $x \vee_L y \geq x$ and $x \vee_L y \geq y$ in $L$, then by monotonicity of $h$, $x\vee_K y=h(x \vee_L y) \geq h(x)=x,h(y)=y$.    
Now, let $z\in K$ be such that $z \geq x, y$ in $K$. 
Since $K$ is a subposet of $L$, then $z\geq x\vee_L y$ in $L$. 
Thus, $z=h(z)\geq h(x\vee_L y)=x\vee_K y$, establishing that $x\vee_K y$ is indeed the smallest common upper bound of $x,y\in K$.
\end{proof}  
In our case, we take $L$ to be the lattice $L(n)$ with with the coordinate-wise maximum as the join, ${K=W(n)}$ and the closure map $h$ is as inductively defined above.

\begin{theorem}[\cite{pallo}]
The set $W(n)$ with the meet and join defined as above is isomorphic to the Tamari lattice $T_{n-1}$.
\end{theorem}
 An example below is a rendition of a Tamari lattice in terms of weight sequences.
 \begin{center}
 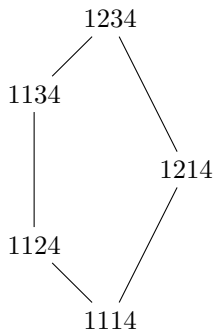
\begin{figure}[H]
  {
   \begin{tikzpicture}
    \node (T0) at (0,0) {$1234$};
    \node (T1) at (-1,-1) {$1134$};
    \node (T2) at (1,-2) {$1214$};
    \node (T3) at (-1,-3) {$1124$};
    \node (T5) at (0,-4) {$1114$};
    \draw (T0) -- (T1) -- (T3) -- (T5) -- (T2);
    \draw (T0) -- (T2);
    \end{tikzpicture}
  }   
  \caption{The Tamari lattice $T_3\simeq W(4)$.}
  \end{figure}
 \end{center} 
   We show that the operad structure on $W$ is compatible with the Tamari order. 
  \begin{theorem}
  \label{WPropMain}
   The partial compositions for the weight sequences as defined in Proposition \ref{WScomp} respect the lattice structures on $W(n)$'s. That is, $PBT\simeq W$ is a lattice operad with respect to the Tamari order.
  \end{theorem}
  \begin{proof}
  First, we check that the partial compositions $(-\circ_i-)$ respect the meet in both arguments.
  Indeed, let $a\in W(m)$, $b,c\in W(n)$ and $1\leq i\leq m$ for $m,n\geq 1$.
  Then, by \eqref{WPartComp2},
  \begin{align*}
  a\circ_i (b\wedge c)=(a_1,\dots, a_{i-1}, {\min(b_1,c_1),\dots, \min(b_{n-1}, c_{n-1})}, a_i',a_{i+1}',\dots, a_{m}'),
  \end{align*}
  which is equal to the coordinate-wise minimum of $a\circ_i b$ and 
  $a\circ_i c$ yielding $a\circ_i (b\wedge c)=(a\circ_i b)\wedge (a\circ_i c)$.
  Now, for $1\leq i\leq n$,
  \begin{align}
  \label{JoinCirci}
  (b\wedge c)\circ_i a=(\min(b_1,c_1),\dots, \min(b_{i-1},c_{i-1}), a_1,\dots,a_{m-1},
  \min(b_{i}, c_{i})', \min(b_{n-1}, c_{n-1})'),
  \end{align}
  where $\min(b_j,c_j)'=\min(b_j,c_j)+m-1$ if $\min(b_j,c_j)\geq j-i+1$ and it is $\min(b_j,c_j)$ otherwise. 
  We observe that the equality ${\min(b_j,c_j)'=\min(b_j',c_j')}$ holds in all six possible cases. These are ${j-i+1\leq b_j\leq c_j}$,   ${b_j<j-i+1\leq c_j}$, ${b_j\leq c_j<j-i+1}$ and the same three double inequalities with $b_j$ and $c_j$ switched. 
    Thus, \eqref{JoinCirci} is equal to the coordinate-wise minimum of 
  \[{b\circ_i a=(b_1,\dots,b_{i-1},a_1,\dots,a_m,b_{i}',\dots b_{n}')}, \quad c\circ_i a=(c_1,\dots,c_{i-1},a_1,\dots,a_m,c_{i}',\dots c_{n}'),\]
  that is, to $(b\circ_i a)\wedge (c\circ_i a)$.
  
  We proceed to showing that the partial compositions $(-\circ_i-)$ respect the join. Let $b\vee_L c$ denote the coordinate-wise maximum of $b$ and $c$. It is the join of $b$ and $c$ taken in the lattice $L=L(n)$ rather than in $W(n)$. Then
  \begin{align*}
  (a\circ_i b)\vee_L (a\circ_i c)&=
  (a_1,\dots,a_{i-1},b_1,\dots,b_{n-1},a_i',\dots,a_{m}')\vee_L 
  (a_1,\dots,a_{i-1},c_1,\dots,c_{n-1},a_i',\dots,a_{m}')\\
  &=
  (a_1,\dots,a_{i-1},\max(b_1,c_1),\dots,\max(b_{n-1},c_{n-1}),a_i',\dots,a_{m}')=a\circ_i (b\vee_L c)\in L(m+n-1).
  \end{align*}
  Applying $h$ to both sides, we get
  $h((a\circ_i b)\vee_L (a\circ_i c))=h(a\circ_i (b\vee_L c))$.
  The expression on the left is $(a\circ_i b)\vee (a\circ_i c)$ by definition of join in $W(m+n-1)$. 
  To handle the right-hand side, note that since $a\in W(m)$, then $h(a)=a$. Hence, by Proposition~\ref{WPropMain}, we have
  $h(a\circ_i (b\vee_L c))=h(a)\circ_i h(b\vee_L c)=a\circ_i (b\vee c)$, establishing that the partial compositions respect the join in the second argument.

  Now, let $1\leq i \leq n$. Then
  \begin{align*}
  (b\circ_i a)\vee_L (c\circ_i a)&=
  (b_1,\dots,b_{i-1},a_1,\dots,a_{m-1},b_i',\dots,b_{m}')\vee_L 
  (c_1,\dots,c_{i-1},a_1,\dots,a_{m-1},c_i',\dots,c_{m}')\\
  &=
  (\max(b_1,c_1),\dots,\max(b_{i-1},c_{i-1}),a_1,\dots,a_{m-1},
  \max(b_i',c_i'),\dots,\max(b_{n}',c_{n}'))\\
  &=
  (\max(b_1,c_1),\dots,\max(b_{i-1},c_{i-1}),a_1,\dots,a_{m-1},
  \max(b_i,c_i)',\dots,\max(b_{n},c_{n})')\\
  &=(b\vee_L c)\circ_i a
  \in L(m+n-1),
  \end{align*}
  where, as in the min-case, $\max(b_i,c_i)'=\max(b_i',c_i')$ follows upon direct verification for $j-i+1\leq b_j\leq c_j$,   $b_j<j-i+1\leq c_j$,
  $b_j\leq c_j<j-i+1$.
  Applying $h$ to both sides, we get $(b\circ_i a)\vee (c\circ_i a)=(b\vee c)\circ_i a$.
  \end{proof}
  \end{example}

\begin{example}
\label{YoungLatEx}
Generalizing example~\ref{LatOpExMZ}, we note that if $A$ is a totally ordered semigroup or, more generally, a partially ordered semigroup, such that $A$ is a lattice, then the word operad $\W A$ has a natural lattice operad structure with respect to the product order. Similarly, given a lattice $M$, the $M$-associative operad $\A M$ is a (non-$\Sigma$) lattice operad as well. We use this observation as follows.

Let $M=\mathbb{N}_0$ with its usual total order and let $Part$ be the lattice suboperad of the $M$-associative operad~$\A \mathbb{N}_0$ such that for all $k\geq 2$, $Part(k)$ consists of all the tuples ${(d_1,\dots,d_{k-1})\in \A \mathbb{N}_0(k)}$, where~${d_{k-1}>0}$. 
For each $k\geq 2$, $Part(k)$ can be identified with the set of all integer partitions of length $k-1$. Indeed, to a tuple~$(d_1,d_2,\dots,d_{k-1})$ 
we associate the partition $\lambda_1\geq \lambda_2\geq \dots \geq  \lambda_{k-1}$
of $${n=d_1+2d_2+\dots+(k-1)d_{k-1}},$$
where
\begin{equation}
\label{TupToPart}
\begin{cases}
\lambda_1 &= d_1 + d_2 + \dots + d_{k-2} + d_{k-1}\\
\lambda_2 &= d_2 + d_3 + \dots + d_{k-1}\\
&\dots\\
\lambda_{k-1} &= d_{k-1},
\end{cases} 
\end{equation}
and vice versa.
In terms of this bijective correspondence, the operadic structure of $\A \mathbb{N}_0$ translates into the composition of partitions in $Part$ that reads
\begin{align}
\label{PartComp}
(\lambda_1,\dots,\lambda_{k-1})\circ_i (\mu_1,\dots,\mu_{l-1})=
(\lambda_1+\mu_1, \lambda_2+\mu_1,\dots, 
\underbracket[0.1ex]{\lambda_{i}+\mu_1,\lambda_{i}+\mu_2,\dots, \lambda_{i}+\mu_{l-1}}_{}, \lambda_i,\dots, \lambda_{k-1})
\end{align}
for $1\leq i \leq k-1$ and 
\[
(\lambda_1,\dots,\lambda_{k-1})\circ_{k} (\mu_1,\dots,\mu_{l-1})=
(\lambda_1+\mu_1, \lambda_2+\mu_1,\dots, \lambda_{k-1}+\mu_1,
 \underbracket[0.1ex]{\mu_1, \mu_2,\dots, \mu_{l-1}}_{})
\]
for $i=k$.
 Furthermore, for each $k\geq 2$, the product order on $\A \mathbb{N}_0(k)=\mathbb{N}_0^{(k-1)}$ translates into the partial order on $Part(k)$ by containment of the corresponding Young diagrams. 
 In particular, this endows Young's lattice\footnote{We discard the empty partition.}~$\bigcup\limits_{k\geq 2}Part(k)$ with a lattice operad structure.
 For each $k\geq 2$, $Part(k)$ is a sublattice thereof.
  \begin{center}
  \begin{figure}[H]
  \subcaptionbox{$(3,2,1,1)\circ_2(2,2)=(5,4,4,2,1,1)$}
  {
   %\resizebox{0.3\textwidth}{!} {\input{partitions1}}
   $\begin{tikzcd}
    \begin{ytableau}
    {} & {}  & *(red!10) {} \\ 
    *(red!10) {} & *(red!10) {}  \\
    *(red!10) {}  \\
    *(red!10) {}
    \end{ytableau}\circ_2     
    \begin{ytableau}
    *(cyan!20) {}  & *(cyan!20) {} \\ 
    *(cyan!20) {}  & *(cyan!20) {} 
    \end{ytableau}=
    \begin{ytableau}
    {} & {} & {}  & {} & *(red!10) {} \\ 
    {} & {} & *(cyan!20) {} & *(cyan!20) {}  \\
    {} & {} & *(cyan!20) {} & *(cyan!20) {}  \\
    *(red!10) {} & *(red!10) {}  \\
    *(red!10) {}  \\
    *(red!10) {}
    \end{ytableau}
 \end{tikzcd}
 $
  }
  \subcaptionbox{$(3,2)\circ_3(2,1,1)=(5,4,2,1,1)$}
  {
  $\begin{tikzcd}
    \begin{ytableau}
    *(red!10) {} & *(red!10) {}  & *(red!10) {} \\ 
    *(red!10) {} & *(red!10) {}     
    \end{ytableau}\circ_3     
    \begin{ytableau}
    *(cyan!20) {}  & *(cyan!20) {} \\ 
    *(cyan!20) {}  \\
    *(cyan!20) {}   
    \end{ytableau}=
    \begin{ytableau}
    {} & {} & *(red!10) {} & *(red!10) {}  & *(red!10) {} \\ 
    {} & {} & *(red!10) {} & *(red!10) {} \\
    *(cyan!20) {} & *(cyan!20) {}  \\
    *(cyan!20) {}  \\
    *(cyan!20) {}
    \end{ytableau}
 \end{tikzcd}
 $
  }    
  \caption{Evaluating compositions in $Part$.}
  \label{YoungLatEval}
  \end{figure}
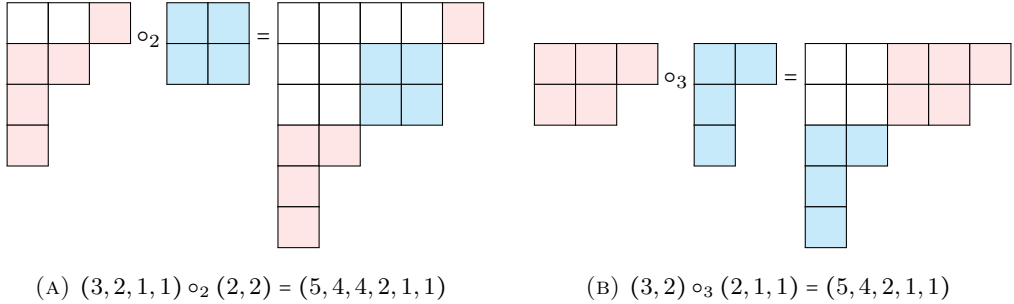
  \end{center}
Consider now the bijection that maps a tuple $(d_1,\dots, d_{k-1})$ to the partition conjugate the one defined in~\eqref{TupToPart}. That is, to such a tuple we associate the partition denoted, as usual, by ${1^{d_1}2^{d_2}\dots (k-1)^{d_{k-1}}}$ that consists of $d_i$ parts equal to~$i$ for~$1\leq i\leq k - 1$. 
Here, as before, ${d_{k-1}>0}$. In terms of this correspondence, the partial compositions within~$\A \mathbb{N}_0$ now look as follows:
\[
1^{d_1}2^{d_2}\dots (k-1)^{d_{k-1}}\circ_i 1^{e_1}2^{e_2}\dots (l-1)^{e_{l-1}}
= 1^{d_1}2^{d_2}\dots (i-1)^{d_{i-1}} 
\underbracket[0.1ex]{i^{e_1}(i+1)^{e_2}\dots (i+l-1)^{e_{l-1}}}_{}
(i+l)^{d_i}\dots (k+l-1)^{d_{k-1}}.
\]
for $1\leq i\leq k-1$ and
\[
1^{d_1}2^{d_2}\dots (k-1)^{d_{k-1}}\circ_k 1^{e_1}2^{e_2}\dots (l-1)^{e_{l-1}}
= 1^{d_1}2^{d_2}\dots (k-1)^{d_{k-1}} \underbracket[0.1ex]{k^{e_1}(k+1)^{e_2}\dots (k+l-1)^{e_{l-1}}}_{}
\]
for $i=k$. 
The non-$\Sigma$ operad $Part'$ with these compositions is isomorphic, as a lattice operad, to $Part$
by means of the partition conjugation.
Note that the conjugation, while being an automorphism of the Young's lattice, is not an automorphism of the operad $Part$. Here, $Part'(k)$ is to be interpreted as the set of all integer partitions with the largest part(s) equal to $k-1$.
As an example, the calculation presented in figure~\ref{YoungLatEval}~(A) above takes the following form in~$Part'$:
 \begin{center}
  \begin{figure}[H]
  {
   $\begin{tikzcd}
    \ydiagram
    [*(red!10)]{1+3,1+1,1}
    *{4,2,1}\circ_2  
    \ydiagram
    [*(cyan!20)]{2,2}
    *{2,2}
    =
     \ydiagram
    [*(red!10)]{3+3,3+1,0,0,1}
    *[*(cyan!20)]{0,0,1+2,1+2,0}
    *{6,4,3,3,1}
 \end{tikzcd}
 $ 
  }
  \caption{
  Evaluating
  $1^1 2^1 4^1\circ_2 
  2^2=1^1 3^2 4^1 6^1$ in $Part'$.
  }
  \end{figure}
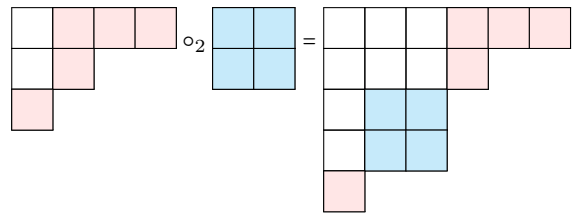
  \end{center}
By means of the correspondence \eqref{TupToPart}, the $M$-associative operad $\A\mathbb{N}$ can be regarded as the operad $Part_d$ of all integer partitions with distinct parts. It is a lattice suboperad of $Part$ and, as any $M$-associative operad, it comes with an involution induced by the tuple reversal $(d_1,d_2,\dots,d_{k-1}) \mapsto (d_{k-1},d_{k-2},\dots,d_{1})$ for all $k\geq 2$. Note that such an involution does not exist in $Part$ due to the defining condtition $d_{k-1}>0$ for~$(d_1,d_2,\dots,d_{k-1})\in Part(k)\subseteq \A \mathbb{N}_0(k)$. 

Explicitly, given a partition $\lambda\in Part_d(k)$ of length $k-1$, the involution $\#:Part_d(k) \to Part_d(k)$ returns its box complement taken in the $k$-by-$l$ box, where $l$ is the largest part of $\lambda$. It is a lattice anti-automorphism and satisfies $(\lambda\circ_i\mu)^\#=\lambda^\#\circ_{k-i+1}\mu^\#$ for all $k,l\geq 2$, $\lambda\in Part_d(k)$, $\mu\in Part_d(l)$ and $1\leq i\leq k$.
\begin{center}
  \begin{figure}[H]
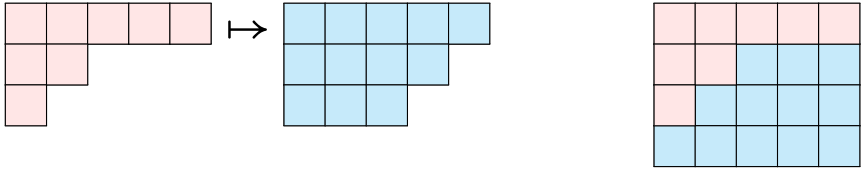

  {
    \ydiagram
    [*(red!10)]{5,2,1}
    *{5,2,1}\,
    {\Huge{$\mapsto$}}\,
    \ydiagram
    [*(cyan!20)]{5,4,3}
    *{5,4,3}
    \hspace{0.8in}
    \ydiagram
    [*(red!10)]{5,2,1,0}
    *[*(cyan!20)]{0,2+3,1+4,5}
    *{5,5,5,5}
  }
  \caption{
  Evaluating
  $(5,2,1)^\#=(5,4,3)$ in $Part_d(4)$.
  }
  \end{figure}
  \end{center}
\end{example}

\begin{example}
For $n\geq 1$, let $Perm(n):=\S_n$. The operad of permutations $Perm$ is defined by partial compositions
\begin{align}
\label{PermOp}
a_1a_2\dots a_m \circ_i b_1 b_2\dots b_n:=
a_1'a_2'\dots {a_{i-1}'}\underbracket[0.1ex]{(b_1+a_i-1)\,(b_2+a_i-1)\,\dots (b_n+a_i-1)} a_{i+1}'\dots a_m'
\end{align}
for all $m,n\geq 1$, $1\leq i\leq m$. Here, we employ the one-line notation for the permutations  $a=a_1a_2\dots a_m\in Perm(m)$, $b=b_1b_2\dots b_n\in Perm(n)$, and set $a_k'=a_k+n-1$ if $a_k>a_i$ while letting $a_k'=a_k$ otherwise. While there are natural actions of $\S_n$ on itself, in the context of this example we will regard $Perm$ as a non-$\Sigma$ operad. It can be regarded as the desymmetrization of the symmetric associative operad $\A s$ in $\Sets$.

The operad structure on $Perm$ can be conveniently visualized by means of the corresponding bipartite matching diagrams, where the partial composition $a \circ_i b$ amounts to substituting the entire diagram $b$ for the $i$-th string of $a$. 
\begin{figure}[H]
\begin{tikzpicture}

% First permutation diagram: (3,1,4,2,5)
\begin{scope}
    % Coordinates for the bottom row (domain of the permutation)
    \foreach \i in {1,...,5} {
        \node[below] at (\i, -1) {\i};   % Labels 1 to 5
        \fill (\i,-1) circle (2pt);      % Nodes in the bottom row
    }

    % Coordinates for the top row (image of the permutation)
    \foreach \i/\j in {1/3, 2/1, 3/4, 4/2, 5/5} {
        \node[above] at (\i, 1) {};  % Labels of the permutation (3,1,4,2,5)
        \fill (\i,1) circle (2pt);     % Nodes in the top row
    }

    % Matching lines between nodes
    \draw (1,-1) -- (3,1);  % 1 maps to 3
    \draw (2,-1) -- (1,1);  % 2 maps to 1
    \draw[red] (3,-1) -- (4,1);  % 3 maps to 4
    \draw (4,-1) -- (2,1);  % 4 maps to 2
    \draw (5,-1) -- (5,1);  % 5 maps to 5
\end{scope}

% Symbol \circ_3 in between
\node at (6, 0) {\Large $\circ_3$};

% Second permutation diagram: (2,3,1)
\begin{scope}[xshift=6cm]
    % Coordinates for the bottom row (domain of the permutation)
    \foreach \i in {1,...,3} {
        \node[below] at (\i, -1) {\i};   % Labels 1 to 3
        \fill (\i,-1) circle (2pt);      % Nodes in the bottom row
    }

    % Coordinates for the top row (image of the permutation)
    \foreach \i/\j in {1/2, 2/3, 3/1} {
        \node[above] at (\i, 1) {};  % Labels of the permutation (2,3,1)
        \fill (\i,1) circle (2pt);     % Nodes in the top row
    }

    % Matching lines between nodes
    \draw (1,-1) -- (2,1);  % 1 maps to 2
    \draw (2,-1) -- (3,1);  % 2 maps to 3
    \draw (3,-1) -- (1,1);  % 3 maps to 1
\end{scope}

\node at (10, 0) {\Large $=$};

\begin{scope}[xshift=10cm]
    % Coordinates for the bottom row (domain of the permutation)
    \foreach \i in {1,...,7} {
        \node[below] at (\i, -1) {\i};   
        \fill (\i,-1) circle (2pt);      % Nodes in the bottom row
    }

    % Coordinates for the top row (image of the permutation)
    \foreach \i/\j in {1/3, 2/1, 6/2, 7/7} {
        \node[above] at (\i, 1) {};  
        \fill (\i,1) circle (2pt);     % Nodes in the top row
        \draw (\i, -1) -- (\j,1);
    }

     % Coordinates for the top row (image of the permutation)
    \foreach \i/\j in {3/5, 4/6, 5/4} {
        \node[above] at (\i, 1) {};  
        \fill (\i,1) circle (2pt);     % Nodes in the top row
        \draw[red] (\i, -1) -- (\j,1);
    }

\end{scope}

\end{tikzpicture}
\caption{Evaluating $31425\circ_3 231 = 3156427$ in $Perm$.}
\end{figure}
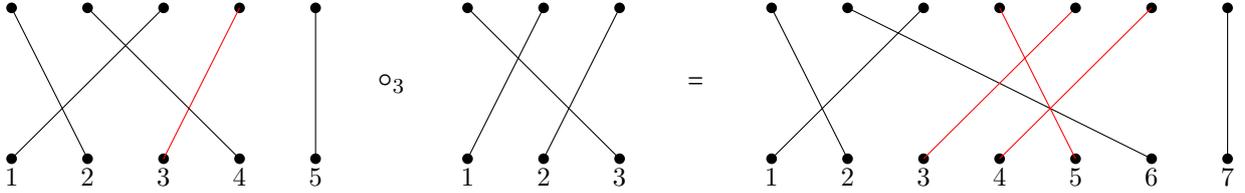
Each component $Perm(n)$ is a lattice with respect to the weak Bruhat order on the symmetric group. Recall that, as a partial order, the (right) weak Bruhat order is characterized by containment of inversion sets: ${a \leq b \Leftrightarrow Inv(a) \subseteq Inv(b)}$, where $Inv(u):=\{(u_i, u_j)|\,i<j\text{ and }u_i>u_j\}$ for a permutation $u$. 
A permutation is uniquely determined by its inversion set. In graphical terms, the inversion set of $u$ can be identified with the set of all pairwise string crossings  in the corresponding matching diagram, where a pair~${(p, q)\in Inv(u)}$ indicates a crossing between the $u^{-1}(p)$-th and the $u^{-1}(q)$-th string.
\begin{lemma}
\label{InvIm}
Let $m,n\geq 1$, $a\in Perm(m)$, $b\in Perm(n)$ and $1\leq i\leq m$. Then $Inv(a\circ_i b)$ is the disjoint union of the sets 
\begin{align*}
A_{\neq i}&=\{(p',q')|\,(p,q)\in Inv(a), p\neq a_i, q\neq a_i\}\\
A_{>}&=\{(p+n-1,k+a_i-1)|\,(p,a_i)\in Inv(a), k=1,2\dots n\}\\
A_{<}&=\{(k+a_i-1, q)|\,(a_i, q)\in Inv(a), k=1,2\dots n\}\\
B&=\{(p+a_i-1, q+a_i-1)|\,(p, q)\in Inv(b)\}.
\end{align*}
\end{lemma}
\begin{proof}
Interpreted in terms of matching diagrams, the lemma amounts to noting that upon substituting the matching diagram of $b$ for the $i$-th string of $a$, any string that used to cross the $i$-th string of $a$ will cross the entire batch of strings $b$ in $a\circ_i b$. Furthermore, all other crossings that were present in $a$ will remain in $a\circ_i b$, and all the crossings of $b$ will transfer into $a\circ_i b$ accordingly.

More precisely, note that any inversion $(p,q)\in Inv(a)$ such that $p\neq a_i$ and $q\neq a_i$ gives rise to a unique inversion $(p',q')\in Inv(a\circ_i b)$, where, as before, $t'=t+n-1$ if $t>a_i$ and $t'=t$ otherwise.
Conversely, by definition of the partial compositions in $Perm$, any inversion $(u,v)\in Inv(a\circ_i b)$, where neither $u$ nor $v$ is in the range $[a_i,a_i+n-1]$ comes from a unique inversion $(p,q)\in Inv(a)$ such that $u=p'$, $v=q'$. Next, any inversion from $Inv(a)$ involving $a_i$, would it appear on the right $(p,a_i)$ or on the left $(a_i, q)$, gives rise to $n$ inversions in $Inv(a\circ_i b)$ induced by the batch of $n$ underscored elements in \eqref{PermOp}. All such inversions, and only they, are of the form $(u,v)$, where either $a_i\leq u\leq a_i+n-1$ or $a_i\leq v\leq a_i+n-1$. Finally, any inversion $(u,v)\in Inv(a\circ_i b)$ where $a_i \leq u, v \leq a_i+n-1$ is induced by the inversion $(u-a_i+1, v-a_i+1)$ from $Inv(b)$ and this correspondence is bijective.
\end{proof}
In what follows, the union $A_{\neq i}\cup A_{>}\cup A_{<}$ for $a\circ_i b$ will be denoted by $Inv_A(a\circ_i b)$, while $Inv_B(a\circ_i b)$ will stand for the set $B$ in the sense of the notation used above.
\begin{corollary}
\label{PermMonoCor}
For any two permutations $b',b''\in Perm(m)$, $Inv_A(a\circ_i b')=Inv_A(a\circ_i b'')$. 
Moreover, for any~$b\leq c$ in $Perm(n)$, we have $a\circ_i b\leq a \circ_i c$. The latter follows from observing that $Inv_A(a\circ_i b)=Inv_A(a\circ_i c)$, and $Inv_B(a\circ_i b)\subseteq Inv_B(a\circ_i c)$.
\end{corollary}

In general, the partial compositions $(-\circ_i-)$ in $Perm$ are not order-preserving in the first argument. Thus, in particular, $Perm$ is not a lattice operad in the sense of our definition. 
Indeed, note that in $Perm(3)$ we have $132\leq 312$, since $Inv(132)=\{(3,2)\}$ and ${Inv(312)=\{(3,1), (3,2)\}}$. At the same time,
$132\circ_1 12=1243$, $312\circ_1 12=3412$. Then $Inv(1243)=\{(4,3)\}$ and $Inv(3412)=\{(3,1),(3,2),(4,1),(4,2)\}$ render $1243$ and $3412$ incomparable.
Nevertheless, the partial compositions do fully respect the lattice structure in the second argument:
\begin{align*}       
a\circ_i (b\wedge c) &= (a\circ_i b)\wedge (a\circ_i c)\\
a\circ_i (b\vee c) &= (a\circ_i b)\vee (a\circ_i c).    
\end{align*}
At this point we recall \cite[Exercise 3.185]{stanley} that the join of $b,c\in Perm(n)$ with respect to the weak Bruhat order can be characterized as the unique permutation  $b\vee c$ in $Perm(n)$ such that $Inv(b\vee c)=\overline{Inv(b)\cup Inv(c)}$, where the overline denotes the transitive closure of the corresponding set regarded as a binary relation on~$[n]$. 
Importantly, any inversion set $Inv(u)$ is transitively closed itself.
The meet $b\wedge c$ can be obtained as $(b^\# \vee c^\#)^\#$, where~${\#:Perm(n)\to Perm(n)}$ is the involution $a_1a_2\dots a_n\mapsto a_n a_{n-1}\dots a_1$, which is a lattice anti-automorphism.
\begin{proposition}
  For any $m,n\geq 1$, $a\in Perm(m)$ and $1\leq i\leq m$, the map~${a\circ_i(-): Perm(n)\to Perm(m+n-1)}$ is an injective lattice homomorphism.
\end{proposition}
\begin{proof}
We begin by showing that $a\circ_i(-)$ preserves the join.
Let $b, c\in Perm(n)$. 
It suffices to show that 
${a\circ_i (b\vee c)\leq (a\circ_i b) \vee (a\circ_i c)}$.
The reverse inequality follows from $(-\circ_i-)$ being order-preserving in the second argument, cf. corollary \ref{PermMonoCor}.
To this end, let $(u,v)$ be an inversion in $a\circ_i(b\vee c)$. If ${(u,v)\in Inv_A(a\circ_i(b\vee c))}$, then by corollary \ref{PermMonoCor}, $Inv_A(a\circ_i(b\vee c))=Inv_A(a\circ_i b)$. Thus, $$(u,v)\in Inv(a\circ_i b)\subseteq Inv(a\circ_i b)\cup Inv(a\circ_i c) \subseteq Inv((a\circ_i b)\vee (a\circ_i c))$$

Otherwise, if $(u,v)\in Inv_B(a\circ_i (b\vee c))$, then it is of the form $(p+a_i-1,q+a_i-1)$ for some $(p,q)\in Inv(b\vee c)$.
Since $Inv(b\vee c)=\overline{Inv(b)\cup Inv(c)}$, then by definition of transitive closure, there is a chain of inversions
\begin{align}
\label{InvChain}
 (p,t_1),(t_1, t_2),\dots, (t_n,q),
\end{align}
from $Inv(b)$, $Inv(c)$, where $p>t_1>\dots>t_n>q$.
We proceed by induction on the smallest value of $n$ for such a chain. If $n=1$, then $(p,q)\in Inv(b)$ or $(p,q)\in Inv(c)$. Hence, $(u,v)=(p+a_i-1,q+a_i-1)$ is in $Inv(a\circ_i b)$ or in $Inv(a\circ_i c)$ leading to
$(u,v)\in\overline{Inv(a\circ_i b) \cup Inv(a\circ_i c)}= Inv((a\circ_i b)\vee (a\circ_i c))$.

Now, let $(p,q)$ be such that it has \eqref{InvChain} as its shortest transitive chain (possibly, non-unique). Dropping off temporarily the last term $(t_n,q)$, by the inductive assumption, we get ${(p+a_{i}-1,t_{n}+a_{i}-1)\in Inv((a\circ_i b)\vee (a\circ_i c))}$. Since $(t_n,q)$ is in $Inv(b)$ or in $Inv(c)$, then 
$$(t_n+a_i-1,q+a_i-1)\in Inv(a\circ_i b)\cup Inv(a\circ_i c)\subseteq Inv((a\circ_i b)\vee (a\circ_i c)).$$ 
The latter set is transitively closed and contains elements $(p+a_{i}-1,t_{n}+a_{i}-1)$, $(t_n+a_i-1,q+a_i-1)$. Hence, it contains $(p+a_i-1,q+a_i-1)=(u, v)$ as well, thus giving the desired inclusion.

To see that $a\circ_i(-)$ preserves the meet, observe that by the very definition $\eqref{PermOp}$ of the partial compositions in $Perm$, we have ${(a\circ_i d)^\#=a^\#\circ_{m-i+1} d^\#}$ for any permutation $d$.
Then
\begin{align*}
a\circ_i(b\wedge c)&=
a\circ_i(b^\#\vee c^\#)^\#=
a^{\#\#}\circ_i(b^\#\vee c^\#)^\#=
(a^{\#}\circ_{m-i+1}(b^\#\vee c^\#))^\#\\
&=((a^{\#}\circ_{m-i+1} b^\#)\vee (a^{\#}\circ_{m-i+1}c^\#))^\#
=((a\circ_{i} b)^\#\vee (a\circ_{i}c)^\#)^\#
=(a\circ_i b)\wedge (a\circ_i c).
\end{align*}
\end{proof}
For $k\geq l$, this construction yields $(k-l+1)\cdot(k-l+1)!$ lattice embeddings of $Perm(l)$ into $Perm(k)$.
  \begin{center}
  \begin{figure}[H]
    \resizebox{0.43\textwidth}{!} { \begin{tikzpicture}[scale=1, transform shape, node distance=1.5cm, every node/.style={inner sep=1pt}]

% % Define nodes with positions for a vertical layout
\node (1234) at (0, 0) {1234};

\node (1243) at (-1.3, 1) {1243};
\node (1324) at (0, 1) {1324};
\node (2134) at (1.3, 1) {2134};

\node (1423) at (-2.6, 2) {1423};
\node (1342) at (-1.3, 2) {1342};
\node (2143) at (0, 2) {2143};
\node (3124) at (1.3, 2) {3124};
\node (2314) at (2.6, 2) {2314};

\node (1432) at (-3.5, 3) {1432};
\node (4123) at (-2.1, 3) {4123};
\node (2413) at (-0.7, 3) {2413};
\node (3142) at (0.7, 3) {3142};
\node (3214) at (2.1, 3) {3214};
\node (2341) at (3.5, 3) {2341};

\node (4132) at (-2.6, 4) {4132};
\node (4213) at (-1.3, 4) {4213};
\node (3412) at (0, 4) {3412};
\node (2431) at (1.3, 4) {2431};
\node (3241) at (2.6, 4) {3241};

\node (4312) at (-1.3, 5) {4312};
\node (4231) at (0, 5) {4231};
\node (3421) at (1.3, 5) {3421};

\node (4321) at (0, 6) {4321};

%Sublattices

\draw[line width=2.5, green!60] (2341) -- (2431);
\draw[line width=2.5, green!60] (2341) -- (3241);
\draw[line width=2.5, green!60] (2431) -- (4231);
\draw[line width=2.5, green!60] (3241) -- (3421);
\draw[line width=2.5, green!60] (4231) -- (4321);
\draw[line width=2.5, green!60] (3421) -- (4321);

\draw[line width=2.5, red!60] (1234) -- (1324);
\draw[line width=2.5, red!60] (1234) -- (1243);
\draw[line width=2.5, red!60] (1243) -- (1423);
\draw[line width=2.5, red!60] (1324) -- (1342);
\draw[line width=2.5, red!60] (1423) -- (1432);
\draw[line width=2.5, red!60] (1342) -- (1432);

\draw[line width=2.5, blue!60] (1234) -- (1324);
\draw[line width=2.5, blue!60] (1234) -- (2134);
\draw[line width=2.5, blue!60] (1324) -- (3124);
\draw[line width=2.5, blue!60] (2134) -- (2314);
\draw[line width=2.5, blue!60] (2314) -- (3214);
\draw[line width=2.5, blue!60] (3124) -- (3214);

\draw[line width=2.5, cyan!60] (4123) -- (4132);
\draw[line width=2.5, cyan!60] (4123) -- (4213);
\draw[line width=2.5, cyan!60] (4213) -- (4231);
\draw[line width=2.5, cyan!60] (4132) -- (4312);
\draw[line width=2.5, cyan!60] (4231) -- (4321);
\draw[line width=2.5, cyan!60] (4312) -- (4321);

% Edges
\draw (1234) -- (2134);
\draw (1234) -- (1324);
\draw (1234) -- (1243);

\draw (1243) -- (2143);
\draw (1243) -- (1423);
\draw (2134) -- (2314);
\draw (2134) -- (2143);
\draw (1324) -- (3124);
\draw (1324) -- (1342);

\draw (1423) -- (1432);
\draw (1423) -- (4123);
\draw (1342) -- (1432);
\draw (1342) -- (3142);
\draw (2143) -- (2413);
\draw (3124) -- (3214);
\draw (3124) -- (3142);
\draw (2314) -- (3214);
\draw (2314) -- (2341);

\draw (1432) -- (4132);
\draw (4123) -- (4213);
\draw (4123) -- (4132);
\draw (2413) -- (4213);
\draw (2413) -- (2431);
\draw (3142) -- (3412);
\draw (3214) -- (3241);
\draw (2341) -- (2431);
\draw (2341) -- (3241);

\draw (4132) -- (4312);
\draw (4213) -- (4231);
\draw (3412) -- (4312);
\draw (3412) -- (3421);
\draw (2431) -- (4231);
\draw (3241) -- (3421);

\draw (4321) -- (4312);
\draw (4321) -- (4231);
\draw (4321) -- (3421);

\end{tikzpicture}}  
    \caption{Four embeddings of $Perm(3)$ into $Perm(4)$ under $a\circ_i(-)$ for $a=12,21$ and $i=1,2$.}
  \end{figure}
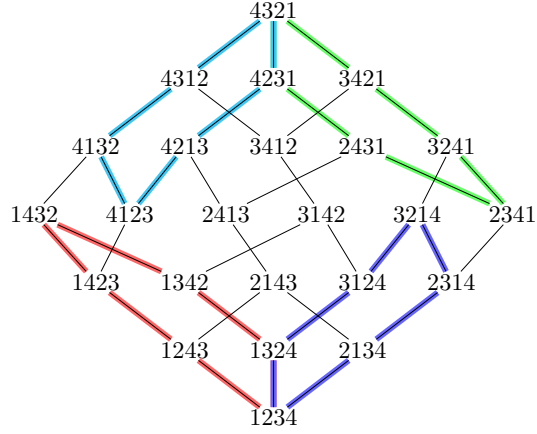  
  \end{center}
\end{example}

\section{More examples: lattice operads from polytopes}
\label{LatOpsPolys}
We turn to examples of lattice operads induced by operads of polytopes in $(\Poly, \times)$, $(\Poly, *)$ and $(\Poly_0, \oplus)$ via the following construction. Given a convex polytope $P$, let $\L(P)$ be its face lattice. Given an affine map of polytopes ${f: P \to Q}$ and a face $F$ of $P$, the intersection of all faces of $Q$ that contain $f(F)$ is the unique face $F'\subset Q$ of the smallest dimension that contains $f(F)$. Furthermore, this preserves the inclusion relation $F\subseteq G\Rightarrow F'\subseteq G'$ and is compatible with compositions of affine maps. That turns $\L$ into a poset-valued functor on $\Poly$. Note that, in general, an order-perserving map $\L(f): \L(P) \to \L(Q)$ constructed in this way fails to be a lattice homomorphism. 
  Indeed, 
  as a simple example, consider the projection $p$ of the one-dimensional polytope $P=[0,1]$ onto a single point $Q=\{pt\}$. Then we have ${\L(f)(\{0\}\wedge \{1\})=\L(f)(\varnothing)=\varnothing}$, while ${\L(f)(\{0\})\wedge \L(f)(\{1\})=Q\wedge Q = Q}$.

Nevertheless, in some cases, such as for polytope embeddings $P \to Q$, or as in the setting of Example~\ref{SubsetLatOp}, the functor $\L$ does respect the lattice structure. Combined with the crucial properties, cf.~\cite{kalai}, that 
\begin{equation}
\begin{aligned}
\L(P\times Q)&=\L(P)\dtimes \L(Q)\\
\L(P*Q)&=\L(P)\times \L(Q)\\
\L(P\oplus Q)&=\L(P)\utimes \L(Q)
\end{aligned}
\label{MonPolyToLat}
\end{equation} 
this allows one to construct operads in $\Lat$ out of operads in $\Poly$ with the corresponding symmetric monoidal structures.
\begin{example}
\label{SubsetLatOp}
  For each $n\geq 1$, let $S(n)$ be the lattice of all subsets of $[n]:=\{1,2,\dots,n\}$. Given $P\in \oS(n)$ and~$1\leq i\leq n$, we denote $P_{<i}:=\{x\in P| x< i\}$ and $P_{>i}:=\{x\in P| x > i\}$.
  Now, for $P\in S(n)$, $Q\in S(m)$ and $1\leq i\leq n$, we define
  \begin{align}
   \label{SubsetLatOpComp}
    P\circ_i Q:=\begin{cases}
                    P_{<i} \cup \{y + i - 1| y\in Q\} \cup \{x + m - 1| x \in P_{>i}\},\quad & i\in P\\
                    P_{<i} \cup \{x + m - 1| x \in P_{>i}\},\quad & i\notin P\\
                    %\varnothing,\quad &P=\varnothing\text{ or }Q=\varnothing
                  \end{cases}.
  \end{align}  
  Upon providing each lattice $S(n)$ with the $\S_n$-action inherited from the natural action of $\S_n$ on $[n]$, the partial compositions defined above turn ~${S=\{S(n)|n\geq 1\}}$ into a symmetric operad in $(\Lat, \times)$. 
  By imposing an additional relation $P\circ_i \varnothing = \varnothing$ for all $n\geq 1$, $P\in S(n)$ and $1\leq i\leq n$, we get a lattice operad in $(\Lat, \dtimes)$. While the operadic axioms for the partial compositions defined above can be verified directly, a more conceptual way to get this operad structure is as follows.

  Consider the the operad  ${\it\oP r}$ of \emph{finite probability distributions} \cite{leinster, baezfritz}. 
  As a symmetric operad in $(\Poly, \times)$, the operad ${\it \oP r}$ is comprised of the standard simplicies
  \[
   {\it \oP r}(n):=\Delta^{n-1}=\{(p_1,p_2,\dots, p_n)\in\mathbb{R}^{n}| \sum\limits_{i=1}^{n}p_i=1, p_i\geq 0\},
  \]
  for all $n\geq 1$ with the coordinate-permuting $\S_n$-action. The partial compositions  are given by
  %${\circ_i: \Delta^{n-1}\times \Delta^{m-1}\to\Delta^{n+m-2}}$
  \begin{align*}  
   (p_1,p_2,\dots,p_n)\circ_i(q_1,q_2,\dots,q_m)=(p_1,p_2,\dots,p_{i-1},p_iq_1,\dots,p_iq_m,p_{i+1},\dots, p_n).
  \end{align*}
  That is, one may regard ${\it\oP r}$ as the suboperad of the word operad $\W I$, where $I:=[0,1]$ is considered as a  multiplicative monoid, cut out by the affine equation $\sum\limits_{i=1}^{n}p_i=1$ in each arity $n \geq 1$.

  Passing to the face lattices we have $\L(\Delta^{n-1})\simeq S(n)$ for all $n\geq 1$. Specifically, a face $F$ of the standard simplex~$\Delta^{n-1}$ can be identified with the set of all indicies $i\in [n]$ such that $p_i\neq 0$ in the relative interior of $F$. Equivalently, it is the unique face of $\Delta^{n-1}$ spanned by the vertices $x_i=(0\dots, 1,\dots, 0)$, with $1$ at the $i$-th position. The partial compositions~ $\circ_i$ transfer accordingly, yielding \eqref{SubsetLatOpComp} as we illustrate below.
  \begin{center}
  \begin{figure}[H]
  \subcaptionbox{$\{2,3\}\circ_1\{1,2\}=\{3,4\}$}
  {
  \resizebox{0.48\textwidth}{!} {%\documentclass[tikz,border=3mm]{standalone}

%\begin{document}

\begin{tikzpicture}[line join = round, line cap = round, z={(-0.2, -0.2)}]
\coordinate [label=right:4] (A) at (5,0,-4);
\coordinate [label=left:1] (B) at (1.8,0,-2);
\coordinate [label=above:3] (C) at (4,3,0);
\coordinate [label=below:2] (D) at (5,0,2);

\coordinate [label=below:1] (TA) at (-1.4-3,0.4,0);
\coordinate [label=below:2] (TB) at (1.4-3,0.4,0);
\coordinate [label=above:3] (TC) at (-3,1.4*1.73+0.4,0);

\coordinate [label=below:1] (LA) at (0,0.4,0);
\coordinate [label=above:2] (LB) at (0,1.4*1.73+0.4,0);

%\foreach \i in {A,B,C,D}
%    \draw[dashed] (0,0)--(\i);
\draw[-, fill=green!30, opacity=.5] (A)--(D)--(B)--cycle;
\draw[-, fill=green!30, opacity=.5] (A) --(D)--(C)--cycle;
\draw[-, fill=green!30, opacity=.5] (B)--(D)--(C)--cycle;

\fill[opacity=0.3,green] (TA)--(TB)--(TC)--cycle;
\draw[-] (TC)--(TA)--(TB);
\draw[-, line width=2pt, red] (TC)--(TB);
\draw[-, line width=2pt, red] (A)--(C);

\draw[-, line width=2pt, red] (LA)--(LB);
\draw[->, thick] (1,1.6)-- (2,1.6) node[midway, above] {$\circ_1$};

\node at (-1,1.6) {\huge $\times$};

\end{tikzpicture}

%\end{document}
}
  }
  \subcaptionbox{$\{2,3\}\circ_2 \{1,2\}=\{2,3,4\}$}
  {
  \resizebox{0.48\textwidth}{!} {%\documentclass[tikz,border=3mm]{standalone}

%\begin{document}

\begin{tikzpicture}[line join = round, line cap = round, z={(-0.2, -0.2)}]
\coordinate [label=right:4] (A) at (5,0,-4);
\coordinate [label=left:1] (B) at (1.8,0,-2);
\coordinate [label=above:3] (C) at (4,3,0);
\coordinate [label=below:2] (D) at (5,0,2);

\coordinate [label=below:1] (TA) at (-1.4-3,0.4,0);
\coordinate [label=below:2] (TB) at (1.4-3,0.4,0);
\coordinate [label=above:3] (TC) at (-3,1.4*1.73+0.4,0);

\coordinate [label=below:1] (LA) at (0,0.4,0);
\coordinate [label=above:2] (LB) at (0,1.4*1.73+0.4,0);

%\foreach \i in {A,B,C,D}
%    \draw[dashed] (0,0)--(\i);
\draw[-, fill=green!30, opacity=.5] (A)--(D)--(B)--cycle;
\draw[-, fill=red!30, opacity=.7] (A) --(D)--(C)--cycle;
\draw[-, fill=green!30, opacity=.7] (B)--(D)--(C)--cycle;

\fill[opacity=0.3,green] (TA)--(TB)--(TC)--cycle;
\draw[-] (TC)--(TA)--(TB);
\draw[-, line width=2pt, red] (TC)--(TB);

\draw[-, line width=2pt, red] (LA)--(LB);
\draw[->, thick] (1,1.6)-- (2,1.6) node[midway, above] {$\circ_2$};

\node at (-1,1.6) {\huge $\times$};

\end{tikzpicture}

%\end{document}
}
  }  
  \end{figure}
  \end{center}
  The suboperad $S'$ of $S$ consisting of non-empty sets is unital quadratic with $S'(1)=\{\{1\}\}$, generated by the binary operations ${\dashv}=\{1\}, {\vdash}=\{2\}, {\bot}=\{1,2\}\in S'(2)$.
  Modulo the existence of a unit and a symmetric structure, the linearization of $S'$ can be identified with the operad of \textit{triassociative algebras} \cite{loday-ronco}, \cite[Section 3.2.2]{giraudo}.   
  The 18 possible pairwise composites of the generators take values within all 7 non-empty subsets of $\{1,2,3\}$ corresponding to the non-empty faces of a 2-simplex, thus giving rise to 11 independent quadratic relations. The latter can in turn be indexed by the $11$ planar rooted trees with $4$ leaves. 
  \begin{center}
  \begin{figure}[H]
    \resizebox{0.5\textwidth}{!} {\begin{tikzpicture}

\coordinate [label=left:1]  (TA) at (-1.4,0);
\coordinate [label=right:3] (TB) at (1.4,0);
\coordinate [label=above:2] (TC) at (0,1.4*1.73);

\fill[opacity=0.3, green] (TA)--(TB)--(TC)--cycle;
\draw[-] (TC)--(TA)--(TB)--cycle;

\node [align=left, fill=red!10, rectangle, draw=gray] (A) at (-3.5, -0.3) 
{$\dashv \circ_1 \dashv$ 
 $\dashv \circ_2 \dashv$ \\
 $\dashv \circ_2 \vdash$
 $\dashv \circ_2 \bot$};

\path [->, draw=gray] (A) edge (TA);

\node [align=left, fill=red!10, rectangle, draw=gray] (B) at (3.5, -0.3)
{$\vdash \circ_1 \vdash$ 
 $\vdash \circ_2 \vdash$ \\ 
 $\vdash \circ_1 \dashv$
 $\vdash \circ_1 \bot$
 };

\path [->, draw=gray] (B) edge (TB);

\node [align=left, fill=red!10, rectangle, draw=gray] (C) at (0, 3.4) 
{$\dashv \circ_1 \vdash$
 $\vdash \circ_2 \dashv$
 };

 \path [->, draw=gray] (C) edge (TC);

 \node [align=left, fill=blue!10, rectangle, draw=gray] (E12) at (-2, 2) 
{$\dashv \circ_1 \bot$
 $\bot \circ_2 \dashv$
 };

\path [->, draw=gray] (E12) edge (-0.7, 0.7 * 1.73);

 \node [align=left, fill=blue!10, rectangle, draw=gray] (E23) at (2, 2) 
{$\vdash \circ_2 \bot$
 $\bot \circ_1 \vdash$
 };

\path [->, draw=gray] (E23) edge (0.7, 0.7 * 1.73);

\node [align=left, fill=blue!10, rectangle, draw=gray] (E13) at (0, -1.3) 
{$\bot \circ_1 \dashv$
 $\bot \circ_2 \vdash$
 };

\path [->, draw=gray] (E13) edge (0, 0);

\node [align=center]  at (0, 0.4) 
{$\bot \circ_1 \bot$
 $\bot \circ_2 \bot$
 };

\end{tikzpicture}}  
    \caption{Relations in $S'(3)$. The partial compositions within each box are pairwise equal.}
  \end{figure}
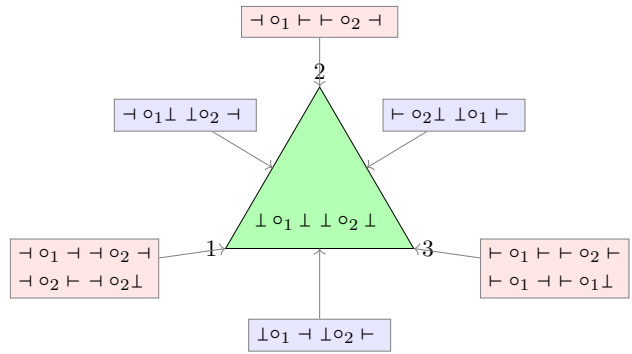  
  \end{center}  
\end{example}

\begin{example}
\label{CompAEx}
Let $M=\{pt\}$ be the $0$-simplex considered as an object of $(\Poly, *)$ and $\A \{pt\}$ be the corresponding $M$-associative operad.
Since $\Delta^p*\Delta^q=\Delta^{p+q+1}$ for all $p ,q\geq 0$, then, inductively, ${\A \{pt\}(n)=\Delta^{n-2}}$.
Upon applying the face lattice functor, we obtain a non-$\Sigma$, non-unital operad $Comp$ in $(\Lat, \times)$. Explicitly, for $n\geq 2$, $Comp(n)=\{0,1\}^{(n-1)}$ is the boolean lattice of rank $n-1$.
%$Comp$ can be identified with the $M$-associative operad in $(\Lat, \times)$ for %$M=\mathcal{B}_1$, the two-element lattice. That is, $Comp(n)=\{0,1\}^{{(n-1)}}$ %is the boolean lattice of order $n-1$. 
A binary word~${b_1\dots b_{n-1}\in Comp(n)}$ encodes the unique face of $\Delta^{n-2}$ whose vertices are $(0,\dots, 1,\dots, 0)$ with $1$ at the $i$-th position iff $b_i=1$. In particular, the empty face is encoded by a string of zeros.
\begin{center}
  \begin{figure}[H]
  \subcaptionbox{$10\circ_1 11=1110$}
  {
  \resizebox{0.41\textwidth}{!} {%\documentclass[tikz,border=3mm]{standalone}

%\begin{document}

\begin{tikzpicture}[line join = round, line cap = round, z={(-0.2, -0.2)}]
\coordinate [label=right:4] (A) at (5,0,-4);
\coordinate [label=left:1] (B) at (1.8,0,-2);
\coordinate [label=above:3] (C) at (4,3,0);
\coordinate [label=below:2] (D) at (5,0,2);

\coordinate [label=below:1] (LA) at (0,0.4,0);
\coordinate [label=above:2] (LB) at (0,1.4*1.73+0.4,0);

\coordinate [label=below:1] (MA) at (-2,0.4,0);
\coordinate [label=above:2] (MB) at (-2,1.4*1.73+0.4,0);

%\foreach \i in {A,B,C,D}
%    \draw[dashed] (0,0)--(\i);
\draw[-, fill=green!30, opacity=.5] (A)--(D)--(B)--cycle;
\draw[-, fill=green!30, opacity=.5] (A) --(D)--(C)--cycle;
\draw[-, fill=red!30, opacity=.5] (B)--(D)--(C)--cycle;

\draw[-, line width=2pt, red] (LA)--(LB);
\draw[-, line width=2pt, black!30!green!50] (MA)--(MB);
\node[fill, circle, red, minimum size=6pt, inner sep=0] at (MA) {};
\draw[->, thick] (1,1.6)-- (2,1.6) node[midway, above] {$\circ_1$};

\node at (-1,1.6) {\huge $*$};

\end{tikzpicture}

%\end{document}
}
  }
  \subcaptionbox{$011\circ_2 0=0011$}
  {
  \resizebox{0.5\textwidth}{!} {%\documentclass[tikz,border=3mm]{standalone}

%\begin{document}

\begin{tikzpicture}[line join = round, line cap = round, z={(-0.2, -0.2)}]
\coordinate [label=right:4] (A) at (5,0,-4);
\coordinate [label=left:1] (B) at (1.8,0,-2);
\coordinate [label=above:3] (C) at (4,3,0);
\coordinate [label=below:2] (D) at (5,0,2);

\coordinate [label=below:1] (TA) at (-1.4-3,0.4,0);
\coordinate [label=below:2] (TB) at (1.4-3,0.4,0);
\coordinate [label=above:3] (TC) at (-3,1.4*1.73+0.4,0);

%\foreach \i in {A,B,C,D}
%    \draw[dashed] (0,0)--(\i);
\draw[-, fill=green!30, opacity=.5] (A)--(D)--(B)--cycle;
\draw[-, fill=green!30, opacity=.7] (A) --(D)--(C)--cycle;
\draw[-, fill=green!30, opacity=.7] (B)--(D)--(C)--cycle;
\draw[-, line width=2pt, red] (A)--(C);

\fill[opacity=0.3,green] (TA)--(TB)--(TC)--cycle;
\draw[-] (TC)--(TA)--(TB);
\draw[-, line width=2pt, red] (TC)--(TB);

\draw[->, thick] (1,1.6)-- (2,1.6) node[midway, above] {$\circ_2$};
\node[fill, circle, minimum size=6pt, inner sep=0] at (0, 1.6) {};

\node at (-1,1.6) {\huge $*$};

\end{tikzpicture}

%\end{document}
}
  }  
  \end{figure}
  \end{center}
Combinatorially, this operad can be interpreted as an operad of integer compositions.
Namely, given a binary word ${b=b_1\dots b_{n-1}\in Comp(n)}$, let $i_1<i_2<\dots < i_k$ be all the indicies $i$ such that $b_i=1$.  Then  the equality~$n=i_1+(i_2-i_1)+\dots + (i_k-i_{k-1}) + (n-i_k)$ yields a composition of $n$ consisting of $k+1$ summands given in the respective order. If $b$ is a string of zeros, the corresponding composition is the trivial one consisting of a single summand $n$. The partial order structure on~$Comp(n)$ is that of a boolean lattice. 

Integer compositions can be encoded graphically by means of tiling diagrams, where a composition ${n=n_1+\dots+n_k}$ is represented by a sequence of adjacent rectangular tiles of lengths $n_1, \dots, n_k$ placed on a grid and aligned in the given order. The corresponding binary word is obtained by marking separators between adjacent tiles by~$1$ and marking all the auxiliary grid lines by $0$. 
\begin{center}
  \begin{figure}[H]
    \resizebox{0.23\textwidth}{!} {\begin{tikzpicture}
\draw[step=1cm,gray,very thin, fill=green!15] (0,0) grid (6,1) rectangle (0,0);
\draw[red, line width=0.8mm] (1,1) -- (1,0);
\draw[red, line width=0.8mm] (4,1) -- (4,0);
\draw[red, line width=0.8mm] (5,1) -- (5,0);

\node at (1,-0.6) {\Large 1};
\node at (2,-0.6) {\Large 0};
\node at (3,-0.6) {\Large 0};
\node at (4,-0.6) {\Large 1};
\node at (5,-0.6) {\Large 1};

\end{tikzpicture}}
  \caption
  {The tiling diagram and the binary word for the composition $6=1+3+1+1$.}
  \end{figure}
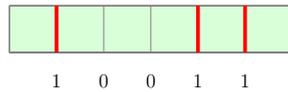
  \end{center}
  
Under this presentation, a partial composition on $Comp$ amounts to substituting a tiling diagram $b$ into the $i$-th cell of a tiling diagram $a$. For instance, the calculation presented in figure (B) above is rendered as follows:
\begin{center}
  \begin{figure}[H]
    \resizebox{7.5cm}{!} {\begin{tikzpicture}
\draw[step=1cm,gray,very thin, fill=green!15] (0,0) grid (4,1) rectangle (0,0);
\draw[red, line width=0.8mm] (2,1) -- (2,0);
\draw[red, line width=0.8mm] (3,1) -- (3,0);

\node at (4.5, 0.5) {\Large $\circ_2$};

\draw[step=1cm,gray,very thin, fill=green!15] (5,0) grid (7,1) rectangle (5,0);

\node at (7.5, 0.5) {\Large $=$};

\draw[step=1cm,gray,very thin, fill=green!15] (8,0) grid (13,1) rectangle (8,0);
\draw[red, line width=0.8mm] (11,1) -- (11,0);
\draw[red, line width=0.8mm] (12,1) -- (12,0);

\end{tikzpicture}} 
    \caption
  {Evaluating $(2+1+1)\circ_2(2)=(3+1+1)$ in $Comp$.}
  \end{figure}
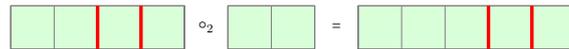
  \end{center}
The boolean structure on $Comp(n)$ can be interpreted as the partial order by reversed refinement of tiling diagrams.
\begin{center}
  \begin{figure}[H]
    \resizebox{0.4\textwidth}{!} {% \begin{tikzpicture}

% \newcommand{\tile}[1]
% {
%  \fill[fill=green!15,draw=gray, thick] (0,0) rectangle ++(3,1);   
%  \draw[step=1cm,gray!70] (0,0) grid (3,1) rectangle (0,0); 
%  \foreach \i in {#1} \draw[red, line width=0.8mm] (\i,1) -- (\i,0); 
% }

% \node (top) at (0,1.7) {\tikz{\tile{1,2}}};

% \matrix (T) [matrix of nodes, column sep=0.7cm,row sep=0.3cm]
% {
%    \tikz{\tile{1}} & \tikz{\tile{2}} \\
% };

% \node (bottom) at (0,-1.7) {\tikz{\tile{}}};

% %%%

% \foreach \j in {1,...,2} 
% {
%  \draw [shorten <=-2pt, shorten >=-2pt] (top) to (T-1-\j);
%  \draw [shorten <=-2pt, shorten >=-2pt] (bottom) to (T-1-\j);
% }

% \end{tikzpicture}

\begin{tikzpicture}

\newcommand{\tile}[2]
{
 \fill[fill=green!15,draw=gray, thick] (0,0) rectangle ++({#1},1);   
 \draw[step=1cm,gray!70] (0,0) grid ({#1},1) rectangle (0,0); 
 \foreach \i in {#2} \draw[red, line width=0.8mm] (\i,1) -- (\i,0); 
}

\matrix (T) [matrix of nodes, column sep=0.7cm,row sep=0.8cm]
{  
    \tikz{\tile{4}{1,2}} & \tikz{\tile{4}{1,3}} & \tikz{\tile{4}{2,3}} \\   
   \tikz{\tile{4}{1}} & \tikz{\tile{4}{2}} & \tikz{\tile{4}{3}} \\   
};

\node (bottom) at (0,-3) {\tikz{\tile{4}{}}};
\node (top) at (0,3.2) {\tikz{\tile{4}{1,2,3}}};

\foreach \j in {1,...,3}
{
    \draw [shorten <=-2pt, shorten >=-2pt] (T-2-\j) -- (bottom);
    \draw [shorten <=-2pt, shorten >=-2pt] (T-1-\j) -- (top);
}

\draw [shorten <=-2pt, shorten >=-2pt] (T-2-1) to (T-1-1);
\draw [shorten <=-2pt, shorten >=-2pt] (T-2-1) to (T-1-2);

\draw [shorten <=-2pt, shorten >=-2pt] (T-2-2) to (T-1-1);
\draw [shorten <=-2pt, shorten >=-2pt] (T-2-2) to (T-1-3);

\draw [shorten <=-2pt, shorten >=-2pt] (T-2-3) to (T-1-2);
\draw [shorten <=-2pt, shorten >=-2pt] (T-2-3) to (T-1-3);

% \node at (0,3) {\Huge$\dots$};

% \draw [shorten <=-2pt, shorten >=-2pt] (bottom) to (T-2-2);
% \draw [shorten <=-2pt, shorten >=-2pt] (bottom) to (T-2-3);

% \draw [shorten <=-2pt, shorten >=-2pt] (T-2-2) to (T-1-1);
% \draw [shorten <=-2pt, shorten >=-2pt] (T-2-2) to (T-1-2);
% \draw [shorten <=-2pt, shorten >=-2pt] (T-2-2) to (T-1-3);

% \draw [shorten <=-2pt, shorten >=-2pt] (T-2-3) to (T-1-2);
% \draw [shorten <=-2pt, shorten >=-2pt] (T-2-3) to (T-1-3);
% \draw [shorten <=-2pt, shorten >=-2pt] (T-2-3) to (T-1-4);

% \foreach \j in {1,...,4}
% {
%     \draw [shorten <=-2pt, shorten >=-2pt] (T-1-\j) -- ++(1,1);
%     \draw [shorten <=-2pt, shorten >=-2pt] (T-1-\j) -- ++(0.3,1);
%     \draw [shorten <=-2pt, shorten >=-2pt] (T-1-\j) -- ++(-1,1);
%     \draw [shorten <=-2pt, shorten >=-2pt] (T-1-\j) -- ++(-0.3,1);
% }

\end{tikzpicture}}
  \caption
  {The lattice $Comp(4)$.}
  \end{figure}
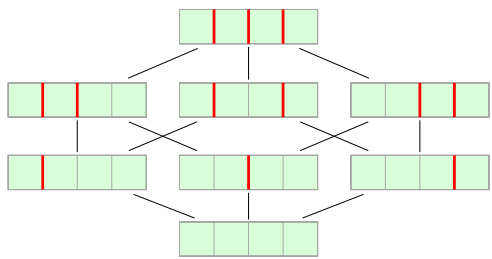
  \end{center}
  
The operad $Comp$ comes equipped with two natural commuting $\mathbb{Z}_2$-actions. 
The first one is an involution~$*: {Comp\to Comp}$ that, in terms of binary words, flips each letter of $b\in Comp(n)$ from $0$ to $1$ and vice versa for all $n\geq 2$. In each arity, it is a lattice anti-automophism that satisfies $(v\circ_i w)^*=v^*\circ_i w^*$ for all $v\in Comp(n)$, $w\in Comp(m)$ and $1\leq i\leq n$. The second involution $\#: {Comp\to Comp}$ acts on a word $b\in Comp(n)$ by reversing it. It is a lattice automorphism of~$Comp(n)$ for all $n\geq 2$ that satisfies \eqref{AMInvol}.
  
The linearization of the operad $Comp$ has appeared before as the operad of \textit{nonsymmetric Poisson} algebras \cite[Example 3.1]{markl-distr} or $As(2)$-algebras \cite{zinbiel}.

\begin{remark}
Naturally, by means of bijections between subsets of a given finite set, binary words and integer compositions, the operad of Example~\ref{SubsetLatOp} can be interpreted in terms of integer compositions as well; cf. \cite[Section 3.1.6]{giraudo}.

Another natural operad structure on integer compositions can be introduced in the spirit of example \ref{YoungLatEx}. Namely, for $k\geq 1$, an element $(a_1,a_2,\dots,a_k)\in \A \mathbb{N}(k+1)$ of the $M$-associative operad $A \mathbb{N}$ can be regarded as a composition $a_1+a_2+\dots+a_k$ of length $k$. In terms of the tiling diagrams, the partial composition $a\circ_i b$ amounts to inserting the diagram $b$ in between the $(i-1)$-th and the $i$-th tiles of the diagram $a$ if $1<i<k$, putting it in front of $a$ if $i=1$ and appending it at the end if $i=k$.
\begin{center}
  \begin{figure}[H]
    \resizebox{7.5cm}{!} {\begin{tikzpicture}
\draw[step=1cm,gray,very thin, fill=green!15] (0,0) grid (4,1) rectangle (0,0);
\draw[red, line width=0.8mm] (2,1) -- (2,0);
\draw[red, line width=0.8mm] (3,1) -- (3,0);
% \draw[red, line width=0.8mm] (0,1) -- (0,0);
% \draw[red, line width=0.8mm] (4,1) -- (4,0);

\node at (4.5, 0.5) {\Large $\circ_2$};

\draw[step=1cm,gray,very thin, fill=green!15] (5,0) grid (7,1) rectangle (5,0);
% \draw[red, line width=0.8mm] (5,1) -- (5,0);
% \draw[red, line width=0.8mm] (7,1) -- (7,0);

\node at (7.5, 0.5) {\Large $=$};

\draw[step=1cm,gray,very thin, fill=green!15] (8,0) grid (14,1) rectangle (8,0);

\draw[red, line width=0.8mm] (10,1) -- (10,0);
\draw[red, line width=0.8mm] (12,1) -- (12,0);
\draw[red, line width=0.8mm] (13,1) -- (13,0);

\end{tikzpicture}} 
    \caption
  {Evaluating $(2+1+1)\circ_2(2)=(2+2+1+1)$.}
  \end{figure}
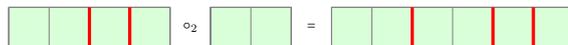
  \end{center}
  This operad structure is compatible with the partial order on integer compositions, analogous to that of the Young lattice on partitions, as introduced in \cite{bergeronetal} and its refinement as in \cite{bjorner-stanley}. In the former case the partial order is induced by the cover relation $a\prec a'$ iff a composition $a'$ is obtained from $a$ by increasing one of the parts $a_j$ by $1$ or by inserting a new part equal to $1$ into an arbitrary position in $a$. In the latter case, $a\prec a'$ iff a composition $a'$ is obtained from $a$ by increasing one of its parts $a_j$ by $1$ or by increasing $a_j$ by $1$ and then splitting it into two adjacent parts $h, a_j+1-h$ for some $h>0$. Neither of these posets is a lattice. Both of them are naturally graded by $n=a_1+\dots+a_k$. In particular, for each $n\geq 1$, the boolean lattice of compositions $Comp(n)$ considered above is an anti-chain in the set of all compositions $Comp=\bigcup_{n\geq 1}Comp(n)$ with respect to either one of these partial orders. 
 \begin{center}
  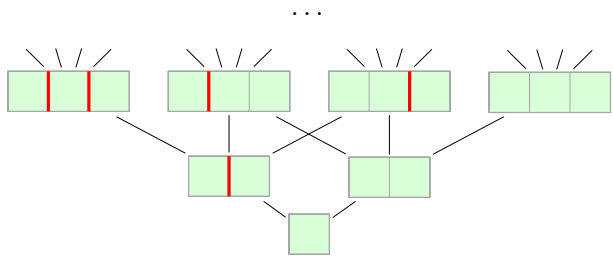
\begin{figure}[H]
    \resizebox{0.5\textwidth}{!} {\begin{tikzpicture}

\newcommand{\tile}[2]
{
 \fill[fill=green!15,draw=gray, thick] (0,0) rectangle ++({#1},1);   
 \draw[step=1cm,gray!70] (0,0) grid ({#1},1) rectangle (0,0); 
 \foreach \i in {#2} \draw[red, line width=0.8mm] (\i,1) -- (\i,0); 
}

\matrix (T) [matrix of nodes, column sep=0.7cm,row sep=0.8cm]
{   
   \tikz{\tile{3}{1,2}} & \tikz{\tile{3}{1}} & \tikz{\tile{3}{2}} & \tikz{\tile{3}{}} \\
   & \tikz{\tile{2}{1}} & \tikz{\tile{2}{}} & \\
};

\node (bottom) at (0,-2.5) {\tikz{\tile{1}{}}};
\node at (0,3) {\Huge$\dots$};

\draw [shorten <=-2pt, shorten >=-2pt] (bottom) to (T-2-2);
\draw [shorten <=-2pt, shorten >=-2pt] (bottom) to (T-2-3);

\draw [shorten <=-2pt, shorten >=-2pt] (T-2-2) to (T-1-1);
\draw [shorten <=-2pt, shorten >=-2pt] (T-2-2) to (T-1-2);
\draw [shorten <=-2pt, shorten >=-2pt] (T-2-2) to (T-1-3);

\draw [shorten <=-2pt, shorten >=-2pt] (T-2-3) to (T-1-2);
\draw [shorten <=-2pt, shorten >=-2pt] (T-2-3) to (T-1-3);
\draw [shorten <=-2pt, shorten >=-2pt] (T-2-3) to (T-1-4);

\foreach \j in {1,...,4}
{
    \draw [shorten <=-2pt, shorten >=-2pt] (T-1-\j) -- ++(1,1);
    \draw [shorten <=-2pt, shorten >=-2pt] (T-1-\j) -- ++(0.3,1);
    \draw [shorten <=-2pt, shorten >=-2pt] (T-1-\j) -- ++(-1,1);
    \draw [shorten <=-2pt, shorten >=-2pt] (T-1-\j) -- ++(-0.3,1);
}

%%%

% \foreach \j in {2,...,3} 
% {
%   \draw [shorten <=-2pt, shorten >=-2pt] (bottom) to (T-2-\j);
% }

\end{tikzpicture}}
  \caption
  {The graded poset of compositions \cite{bergeronetal}.}
  \end{figure}
  \end{center}
  
\end{remark}

\end{example}

\begin{example}
\label{CubeWordEx}
Let $J:=[-1,1]$ be the closed interval considered as a multiplicative semigroup and let $\W J$ be the corresponding word operad in $(\Poly, \times)$. That is, for $n\geq 1$, $\W J(n)$ is the $n$-dimensional hypercube~$[-1,1]^{n}$. Upon taking the face lattices, we obtain the corresponding lattice operad $T$ in $(\Lat, \dtimes)$. Explicitly, $T(n)$ is the $n$-fold lower truncated product of the four-element lattice 
\begin{equation}
 \label{Rhombus}
 \begin{aligned}
 \begin{tikzpicture}
     \node[] (top) at (0,-1) {$0$};
     \node[below left of=top] (0) {$+$};
     \node[below right of=top] (1) {$-$};
     \node[below right of=0] (bottom) {$\varnothing$};
     \draw [thick] (top) -- (0);
     \draw [thick] (top) -- (1);
     \draw [thick] (0) -- (bottom);
     \draw [thick] (1) -- (bottom);
 \end{tikzpicture}
 \end{aligned}
 \end{equation}
with itself. Here, the symbols $"0"$, $"+"$ and $"-"$ denote $[-1,1]$, $\{1\}$ and $\{-1\}$ respectively as the non-empty faces of $J$. 
Thus we can identify the elements of $T(n)$ with ternary words of length $n$ along with a special word~$\varnothing$. A non-special word $t_1t_2\dots t_n$ represents the non-empty face of $\W J(n)=[-1,1]^n$ consisting of all the points~$(x_1,x_2,\dots, x_n)$ such that $x_i=1$ if $t_i=+$, $x_i=-1$ if $t_i=-$ and $-1 \leq x_i\leq 1$ if $t_i=0$. 
The lattice structure on~$T(n)$ is given by ordering non-special words 
 $w=t_1t_2\dots t_n$ and $w'=t'_1t'_2\dots t'_n$
by the relation~$w\leq w'$ iff~$t_i=t'_i$ or $t'_i=0$. The special word $\varnothing$ corresponds to the empty face and it is the smallest element of $T(n)$. 

As of the compositional structure, all non-special words of $T$ form a suboperad that can be identified with the word operad for the semigroup $\{0, +, -\}$, which is the multiplicative semigroup of integers $\{0,\pm 1\}$, where $"+"$ and~$"-"$ correspond to $1$ and $-1$ respectively. Any composition with $\varnothing$ results in $\varnothing$. The suboperad $T'$ consisting of non-special words  has $9$ binary operations with $135$ independent relations between them.
 \begin{center}
  \begin{figure}[H]
  \subcaptionbox{$0+\circ_2 0=00$}
  {
  \resizebox{0.43\textwidth}{!} {%\documentclass[tikz,border=3mm]{standalone}

%\begin{document}

\begin{tikzpicture}[line join = round, line cap = round]
\pgfmathsetmacro{\Depth}{2.4}
\pgfmathsetmacro{\Height}{2.4}
\pgfmathsetmacro{\Width}{2.4}
\pgfmathsetmacro{\xo}{4}
\pgfmathsetmacro{\yo}{1}
\pgfmathsetmacro{\zo}{0}

\coordinate  (SA) at (-1.2-3,0.4,0);
\coordinate  (SB) at (1.2-3,0.4,0);
\coordinate  (SC) at (1.2-3,\Height+0.4,0);
\coordinate  (SD) at (-1.2-3,\Height+0.4,0);

\coordinate  (TA) at (-1.2+4,0.4,0);
\coordinate  (TB) at (1.2+4,0.4,0);
\coordinate  (TC) at (1.2+4,\Height+0.4,0);
\coordinate  (TD) at (-1.2+4,\Height+0.4,0);

\coordinate (LA) at (0,0.4,0);
\coordinate (LB) at (0,\Height+0.4,0);

\draw[opacity=0.8,fill=green!30] (SA)--(SB)--(SC)--(SD)--cycle;
\draw[-, line width=2pt, red] (SC)--(SD);
\draw[draw=none] (SA)--(SB) node[midway, below] {$1$};
\draw[draw=none] (SA)--(SD) node[midway, left] {$2$};

\node at (-1,1.6) {$\Large \times$};

\draw[-, line width=2pt, red] (LA)--(LB) node[midway, left, black] {$1$};

\draw[->, thick] (1,1.6)-- (2,1.6) node[midway, above] {$\Large \circ_2$};

\draw[opacity=0.8,fill=red!30] (TA)--(TB)--(TC)--(TD)--cycle;
\draw[draw=none] (TA)--(TB) node[midway, below] {$1$};
\draw[draw=none] (TA)--(TD) node[midway, left] {$2$};
\end{tikzpicture}

%\end{document}
}
  }
  \hspace{0.5cm}
  \subcaptionbox{$0-\circ_2 -0=0+0$}
  {
  \resizebox{0.48\textwidth}{!} {%\documentclass[tikz,border=3mm]{standalone}

%\begin{document}

\begin{tikzpicture}[line join = round, line cap = round]
\pgfmathsetmacro{\Depth}{2.4}
\pgfmathsetmacro{\Height}{2.4}
\pgfmathsetmacro{\Width}{2.4}
\pgfmathsetmacro{\xo}{4}
\pgfmathsetmacro{\yo}{1}
\pgfmathsetmacro{\zo}{0}
\coordinate (O) at (\xo,\yo,\zo);
\coordinate (A) at (\xo,\yo+\Width,\zo);
\coordinate (B) at (\xo,\yo+\Width,\zo+\Height);
\coordinate (C) at (\xo,\yo,\zo+\Height);
\coordinate (D) at (\xo+\Depth,\yo,\zo);
\coordinate (E) at (\xo+\Depth,\yo+\Width,\zo);
\coordinate (F) at (\xo+\Depth,\yo+\Width,\zo+\Height);
\coordinate (G) at (\xo+\Depth,\yo,\zo+\Height);

\draw[fill=green!30] (O) -- (C) -- (G) -- (D) -- cycle;% Bottom Face
\draw[fill=green!30] (O) -- (A) -- (E) -- (D) -- cycle;% Back Face
\draw[fill=green!30] (O) -- (A) -- (B) -- (C) -- cycle;% Left Face
\draw[fill=green!30,opacity=0.8] (D) -- (E) -- (F) -- (G) -- cycle;% Right Face
\draw[fill=green!30,opacity=0.8] (C) -- (B) -- (F) -- (G) -- cycle;% Front Face
\draw[fill=red!30,opacity=0.8] (A) -- (B) -- (F) -- (E) -- cycle;% Top Face

\coordinate  (SA) at (-1.2,0.4,0);
\coordinate  (SB) at (1.2,0.4,0);
\coordinate  (SC) at (1.2,\Height+0.4,0);
\coordinate  (SD) at (-1.2,\Height+0.4,0);

\coordinate  (TA) at (-1.2-3.5,0.4,0);
\coordinate  (TB) at (1.2-3.5,0.4,0);
\coordinate  (TC) at (1.2-3.5,\Height+0.4,0);
\coordinate  (TD) at (-1.2-3.5,\Height+0.4,0);

\draw[opacity=0.8,fill=green!30] (SA)--(SB)--(SC)--(SD)--cycle;
\draw[opacity=0.8,fill=green!30] (TA)--(TB)--(TC)--(TD)--cycle;

\draw[-, line width=2pt, red] (SA)--(SD);
\draw[-, line width=2pt, red] (TB)--(TA);

\draw[draw=none] (SA)--(SB) node[midway, below] {$1$};
\draw[draw=none] (SA)--(SD) node[midway, left] {$2$};

\draw[draw=none] (TA)--(TB) node[midway, below] {$1$};
\draw[draw=none] (TA)--(TD) node[midway, left] {$2$};

\draw[draw=none] (C)--(G) node[midway, below] {$1$};
\draw[draw=none] (C)--(B) node[midway, left] {$2$};
\draw[draw=none] (G)--(D) node[midway, right] {$3$};

\draw[->, thick] (1.5,1.6)-- (2.5,1.6) node[midway, above] {\Large$\circ_2$};

\node at (-1.8,1.6) {\Large$\times$};

\end{tikzpicture}

%\end{document}
}
  }  
  \end{figure}
  \end{center}
 The operad $T$ carries an hyperoctahedral symmetry naturally inherited from the corresponding polytopes, and thus may be regarded as a type $B$ object.
 Namely, for $n\geq 1$, the component $T(n)$ is equipped with a (right) action of $\mathbb{B}_{n}:=\S_{n}\ltimes (\bbZ_2)^{n}$ as follows.
 An element~$(\sigma, g) \in \mathbb{B}_{n}$ acts on a non-special word $w=t_1t_2\dots t_{n}\in T(n)$ by permuting the letters via $\sigma^{-1}$, while each $\bbZ_2$-entry $g_i$ of $g$ acts on a letter $t$ by flipping its sign if $g_i=1$. The action of $(\sigma, g)$ on $\varnothing$ is trivial. For non-special $v\in T(n), w\in T(m)$, the equivariance of the partial compositions reads
 \begin{equation}
 \label{Bequiv}
 \begin{cases}
 v\circ_i (w\cdot (\sigma, g)) = (v\circ_i w)\cdot (\sigma', g'),\quad &(\sigma, g)\in \mathbb{B}_{m}\\
 (v\cdot (\sigma, g))\circ_i w = (v\circ_{{\sigma}(i)} w)\cdot (\sigma'', g''),\quad &(\sigma, g)\in \mathbb{B}_n
 \end{cases}.
 \end{equation}
 Here, $(\sigma', g') \in \mathbb{B}_{n+m-1}$ acts via $(\sigma, g)$ on the subword $t_{i}t_{i+1}\dots t_{i+m-2}$ of $v\circ_i w$ and acts identically outside of it.
 The element $(\sigma'', g'')$ swaps the subwords $s=t_{\sigma(i)}t_{\sigma(i)+1}\dots t_{\sigma(i)+m-2}$ and $t=t_it_{i+1}\dots t_{i+m-2}$ within $v\circ_i w$, and then proceeds to act on the word $(v\circ_i w)\setminus s$ via $(\sigma, g)$.  
 \begin{remark}
 The notion of a \emph{hyperoctahedral operad} that the above operad would be an example of can be formalized by considering the monad of rooted non-planar trees with leaves labeled by signed non-zero integers. Such a tree with $n$ leaves labelled by $\{\pm1, \pm 2,\dots, \pm n\}$ comes with a natural action of $\mathbbm{B}_n$. The operadic composition amounts to grafting, which, in addition, is required to be homogeneous with respect to the signs of the labels.
 \end{remark}
\end{example}

\begin{example}
\label{CompBEx}
Let $J=[-1,1]$ be as before, and consider the corresponding non-symmetric $J$-associative operad $\A J$ in $(\Poly, \times)$. That is, for $n\geq 2$, we have $\A J(n)=[-1,1]^{n-1}$ with gap insertions as operadic compositions. 
For the corresponding face lattice operad in $(\Lat, \dtimes)$, which we denote $Comp_B$, the component~ $Comp_B(n)$ is the $(n-1)$-fold lower truncated product of the lattice \eqref{Rhombus} with itself. As before, one can identify the elements of $Comp_B(n)$ with ternary words of length $n-1$ (note the shift by one) consisting of the symbols $0,+,-$, and the special word~$\varnothing$. The lattice structure and the interpretation of ternary words in terms of the hypercube faces are as described in the previous example. 
For non-special words, the partial composition~$w\circ_i w'$ amounts to inserting the entire word $w'$ before the $i$-th letter of $w$ if $i\leq n-1$ or appending it to the right of $w$ otherwise. Any partial composition with $\varnothing$ results in $\varnothing$. 
 \begin{center}
  \begin{figure}[H]
  \subcaptionbox{$+0\circ_3 0=+00$}
  {
  \resizebox{0.48\textwidth}{!} {%\documentclass[tikz,border=3mm]{standalone}

%\begin{document}

\begin{tikzpicture}[line join = round, line cap = round]
\pgfmathsetmacro{\Depth}{2.4}
\pgfmathsetmacro{\Height}{2.4}
\pgfmathsetmacro{\Width}{2.4}
\pgfmathsetmacro{\xo}{4}
\pgfmathsetmacro{\yo}{1}
\pgfmathsetmacro{\zo}{0}
\coordinate (O) at (\xo,\yo,\zo);
\coordinate (A) at (\xo,\yo+\Width,\zo);
\coordinate (B) at (\xo,\yo+\Width,\zo+\Height);
\coordinate (C) at (\xo,\yo,\zo+\Height);
\coordinate (D) at (\xo+\Depth,\yo,\zo);
\coordinate (E) at (\xo+\Depth,\yo+\Width,\zo);
\coordinate (F) at (\xo+\Depth,\yo+\Width,\zo+\Height);
\coordinate (G) at (\xo+\Depth,\yo,\zo+\Height);

\draw[fill=green!30] (O) -- (C) -- (G) -- (D) -- cycle;% Bottom Face
\draw[fill=green!30] (O) -- (A) -- (E) -- (D) -- cycle;% Back Face
\draw[fill=green!30] (O) -- (A) -- (B) -- (C) -- cycle;% Left Face
\draw[fill=red!30,opacity=0.8] (D) -- (E) -- (F) -- (G) -- cycle;% Right Face
\draw[fill=green!30,opacity=0.8] (C) -- (B) -- (F) -- (G) -- cycle;% Front Face
\draw[fill=green!30,opacity=0.8] (A) -- (B) -- (F) -- (E) -- cycle;% Top Face

\coordinate  (SA) at (-1.2-3,0.4,0);
\coordinate  (SB) at (1.2-3,0.4,0);
\coordinate  (SC) at (1.2-3,\Height+0.4,0);
\coordinate  (SD) at (-1.2-3,\Height+0.4,0);

\coordinate (LA) at (0,0.4,0);
\coordinate (LB) at (0,\Height+0.4,0);

\draw[opacity=0.8,fill=green!30] (SA)--(SB)--(SC)--(SD)--cycle;
\draw[-, line width=2pt, red] (SC)--(SB);

\draw[-, line width=2pt, red] (LA)--(LB) node[midway, left, black] {$1$};
\draw[draw=none] (SA)--(SB) node[midway, below] {$1$};
\draw[draw=none] (SA)--(SD) node[midway, left] {$2$};
\draw[draw=none] (C)--(G) node[midway, below] {$1$};
\draw[draw=none] (C)--(B) node[midway, left] {$2$};
\draw[draw=none] (G)--(D) node[midway, right] {$3$};

\draw[->, thick] (1,1.6)-- (2,1.6) node[midway, above] {$\circ_3$};

\node at (-1,1.6) {$\times$};

\end{tikzpicture}

%\end{document}
}
  }
  \hspace{0.5cm}
  \subcaptionbox{$-\circ_1 00=00-$}
  {
  \resizebox{0.44\textwidth}{!} {%\documentclass[tikz,border=3mm]{standalone}

%\begin{document}

\begin{tikzpicture}[line join = round, line cap = round]
\pgfmathsetmacro{\Depth}{2.4}
\pgfmathsetmacro{\Height}{2.4}
\pgfmathsetmacro{\Width}{2.4}
\pgfmathsetmacro{\xo}{4}
\pgfmathsetmacro{\yo}{1}
\pgfmathsetmacro{\zo}{0}
\coordinate (O) at (\xo,\yo,\zo);
\coordinate (A) at (\xo,\yo+\Width,\zo);
\coordinate (B) at (\xo,\yo+\Width,\zo+\Height);
\coordinate (C) at (\xo,\yo,\zo+\Height);
\coordinate (D) at (\xo+\Depth,\yo,\zo);
\coordinate (E) at (\xo+\Depth,\yo+\Width,\zo);
\coordinate (F) at (\xo+\Depth,\yo+\Width,\zo+\Height);
\coordinate (G) at (\xo+\Depth,\yo,\zo+\Height);

\draw[fill=green!30] (O) -- (C) -- (G) -- (D) -- cycle;% Bottom Face
\draw[fill=green!30] (O) -- (A) -- (E) -- (D) -- cycle;% Back Face
\draw[fill=green!30] (O) -- (A) -- (B) -- (C) -- cycle;% Left Face
\draw[fill=green!30,opacity=0.8] (D) -- (E) -- (F) -- (G) -- cycle;% Right Face
\draw[fill=red!30,opacity=0.8] (C) -- (B) -- (F) -- (G) -- cycle;% Front Face
\draw[fill=green!30,opacity=0.8] (A) -- (B) -- (F) -- (E) -- cycle;% Top Face

\coordinate  (SA) at (-1.2-0.5,0.4,0);
\coordinate  (SB) at (1.2-0.5,0.4,0);
\coordinate  (SC) at (1.2-0.5,\Height+0.4,0);
\coordinate  (SD) at (-1.2-0.5,\Height+0.4,0);

\coordinate (LA) at (-3,0.4,0);
\coordinate (LB) at (-3,\Height+0.4,0);

\draw[opacity=0.8,fill=red!30] (SA)--(SB)--(SC)--(SD)--cycle;
%\draw[-, line width=2pt, red] (SA)--(SB);

\draw[-, line width=2pt, black!30!green!50] (LA)--(LB) node[midway, left, black] {$1$};
\node[fill, circle, red, minimum size=6pt, inner sep=0] at (LA) {};

\draw[draw=none] (SA)--(SB) node[midway, below] {$1$};
\draw[draw=none] (SA)--(SD) node[midway, left] {$2$};
\draw[draw=none] (C)--(G) node[midway, below] {$1$};
\draw[draw=none] (C)--(B) node[midway, left] {$2$};
\draw[draw=none] (G)--(D) node[midway, right] {$3$};

\draw[->, thick] (1,1.6)-- (2,1.6) node[midway, above] {$\circ_1$};

\node at (-2.5,1.6) {$\times$};

\end{tikzpicture}

%\end{document}
}
  }  
  \end{figure}
  \end{center}
  The operad $Comp_B$ can be regarded as an operad of 2-colored integer compositions, which, as we argue below, can be considered a type $B$ analogue of ordinary integer compositions. 
  We recall that an $m$-\emph{colored composition} of $n$ is an ordered tuple of positive integers $n_1,n_2,\dots,n_k$ such that~$n=n_1+n_2+\dots+n_k$ and each summand $n_i$ is assigned one of the $m$ labels, or \emph{colors}, $1,2,\dots,m$; cf.\cite{drake-petersen}. For example, $1_1+2_1+1_2$, $1_1+2_2+1_2$ and~$2_1+2_2$ are three different $2$-colored compositions of $4$.
  We adopt the convention \cite[Definition 1(a)]{hopkins-ouvry} that~$n_1$ is always colored by $1$. This mods out the unnecessary, in our case, $\S_m$-action on the colors.

 A bijection between $2$-colored compositions and ternary words can be established as follows. To a {$2$-~colored} composition $n=n_1+n_2+\dots+n_k$ we assign a tiling diagram with the tiles of lengths ${n_1,n_2,\dots,n_k}$ as in Example~\ref{CompAEx}, but with each tile coming now in one of the two colors given by the color of the corresponding summand, adhering to the convention concerning the color of the first part. We read off the corresponding ternary word by going over the inner grid lines of the tiling diagram from left to right and marking such a line by $+$ if it is a tile separator and the tile to the right of it is of color $1$, marking it by $-$ if it is a tile separator and the tile to the right is of color $2$ and marking it by $0$ if it is a non-separating grid line. 
  \begin{center}
  \begin{figure}[H]
    \resizebox{0.25\textwidth}{!} {\begin{tikzpicture}

\fill[fill=red!10,draw=gray] (0,0) rectangle ++(1,1);
\fill[fill=cyan!20,draw=gray] (1,0) rectangle ++(3,1);
\fill[fill=red!10,draw=gray] (4,0) rectangle ++(2,1);

\draw[step=1cm,gray,very thick] (0,0) grid (6,1) rectangle (0,0);

\draw[red, line width=0.8mm] (1,1) -- (1,0);
\draw[red, line width=0.8mm] (3,1) -- (3,0);
\draw[red, line width=0.8mm] (4,1) -- (4,0);

\node at (1,-0.6) {\huge $\mathbf{-}$};
\node at (2,-0.6) {\Large $0$};
\node at (3,-0.6) {\huge $-$};
\node at (4,-0.6) {\huge $+$};
\node at (5,-0.6) {\Large $0$};

\end{tikzpicture}}
  \caption
  {The tiling diagram and the ternary word for the 2-colored composition~ ${6=1_1+2_2+1_2+2_1}$.}
  \end{figure}
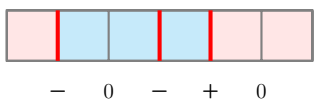
  \end{center} 
  The partial composition $v\circ_i w$ of $2$-colored tiling diagrams amounts to substituting a diagram $w$ into the $i$-th cell of $v$ and flipping the color of the first tile of the inserted diagram if the $i$-th cell of $v$ is of color $2$; otherwise, the original colors are to be retained.
  \begin{center}
  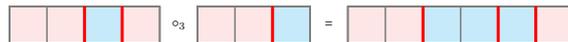
\begin{figure}[H]
    \resizebox{7.5cm}{!} {\begin{tikzpicture}
\fill[fill=red!10,draw=gray] (0,0) rectangle ++(4,1);
\fill[fill=cyan!20,draw=gray] (2,0) rectangle ++(1,1);

\draw[step=1cm,gray,very thick] (0,0) grid (4,1) rectangle (0,0);
\draw[red, line width=0.8mm] (2,1) -- (2,0);
\draw[red, line width=0.8mm] (3,1) -- (3,0);

\node at (4.5, 0.5) {\Large $\circ_3$};

\fill[fill=red!10,draw=gray] (5,0) rectangle ++(2,1);
\fill[fill=cyan!20,draw=gray] (7,0) rectangle ++(1,1);
\draw[step=1cm,gray,very thick] (5,0) grid (8,1) rectangle (5,0);
\draw[red, line width=0.8mm] (7,1) -- (7,0);

\node at (8.5, 0.5) {\Large $=$};

\fill[fill=red!10,draw=gray] (9,0) rectangle ++(2,1);
\fill[fill=cyan!20,draw=gray] (11,0) rectangle ++(3,1);
\fill[fill=red!10,draw=gray] (14,0) rectangle ++(1,1);
\draw[step=1cm,gray,very thick] (9,0) grid (15,1) rectangle (9,0);
\draw[red, line width=0.8mm] (11,1) -- (11,0);
\draw[red, line width=0.8mm] (13,1) -- (13,0);
\draw[red, line width=0.8mm] (14,1) -- (14,0);

%part 2

\end{tikzpicture}} 
    \caption
  {Evaluating $(2_1+1_2+1_1)\circ_3(2_1+1_2)=(2_1+2_2+1_2+1_1)$ in $Comp_B$.}
  \end{figure}  
  \end{center}
 In terms of $2$-colored tiling diagrams, the covering relation that induces the lattice structure of $Comp_B(n)$ is given by $t\prec t'$ iff a diagram $t$ can be obtained by splitting a single tile $p$ of $t'$ into two adjacent pieces $p_{l}$, $p_r$ followed, possibly, by flipping the color of the new tile on the right $p_r$.
  \begin{center}  
  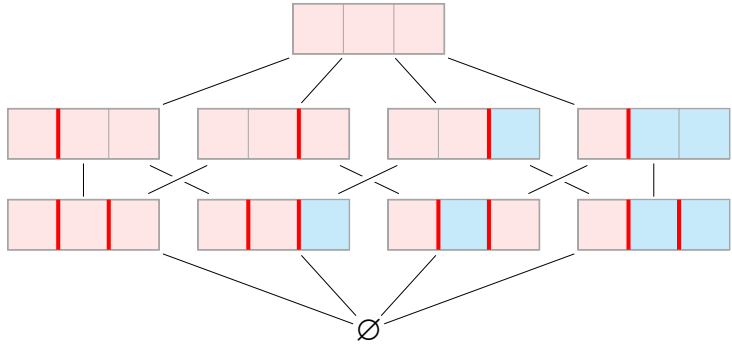
\begin{figure}[H]
    \resizebox{0.6\textwidth}{!} {\begin{tikzpicture}

\newcommand{\tile}[2]
{
 \fill[fill=red!10,draw=gray, thick] (0,0) rectangle ++(3,1);  
 \foreach \i/\j in {#2} \fill[fill=cyan!20] (\i,0) rectangle ++(\j,1);  
 \draw[step=1cm,gray!70] (0,0) grid (3,1) rectangle (0,0); 
 \foreach \i in {#1} \draw[red, line width=0.8mm] (\i,1) -- (\i,0); 
}

\node (top) at (0,3) {\tikz{\tile{}{}}};

\matrix (T) [matrix of nodes, column sep=0.5cm,row sep=0.5cm]
{
   \tikz{\tile{1}{}} & \tikz{\tile{2}{}} & \tikz{\tile{2}{2/1}} & \tikz{\tile{1}{1/2}} \\
   \tikz{\tile{1,2}{}} & \tikz{\tile{1,2}{2/1}} & \tikz{\tile{1,2}{1/1}} & \tikz{\tile{1,2}{1/2}} \\
};

\node (bottom) at (0,-3) {\Huge$\varnothing$};

%%%

\foreach \j in {1,...,4} \draw [shorten <=-2pt, shorten >=-2pt] (top) to (T-1-\j);

\foreach \i/\j in {1/1, 1/2, 2/1, 2/3, 3/2, 3/4, 4/3, 4/4} 
\draw [preaction={draw=white, -,line width=6pt},shorten <=-2pt, shorten >=-3pt]  (T-1-\i) to (T-2-\j);

\foreach \j in {1,...,4} \draw [shorten <=-2pt, shorten >=-2pt] (bottom) to (T-2-\j);

\end{tikzpicture}}
  \caption
  {The lattice $Comp_B(3)$.}
  \end{figure}
  \end{center}  
  To justify the assertion that 2-colored compositions are a type $B$ analogue of integer compositions, we relate them to set partitions of type $B$. Recall that a type $B$ partition $ \pi$ of $[\pm n]$ is a partition of the set~$[\pm n]=\{\pm 1, \pm2,\dots,\pm n\}$ such that (1) for any block $\beta\in \pi$, the block $-\beta$ defined by negating all the elements of $\beta$ is in $\pi$; (2) there is at most one block $\beta$, called the zero block, such that $\beta=-\beta$, cf.\cite{reiner}.

  First, we observe that for any $n\geq 1$, there is a bijection between ordinary integer compositions of $n$ and (type~$A$) partitions of the set $[n]=\{1,2,\dots,n\}$ into $k\leq 2$ blocks. Specifically, given a two-block partition~${[n]\dashv \{\pi_1, \pi_2\}}$ we assign to it the unique block $\pi_i$ such that $n\notin \pi_i$. To the trivial single-block partition we assign the empty set. This maps all partitions of $[n]$ with $k\leq 2$ blocks bijectively to the subsets of~$[n-1]=\{1,2,\dots,n-1\}$. The latter can be identified with integer compositions of $n$ as we have recalled above in Example~\ref{CompAEx}.
  
  By analogy, we take it that the data of a \emph{type $B$ integer composition of $n$} is that of a type $B$ partition $\pi$ of the set $[\pm n]$ with $k\leq 2$ non-zero blocks. Additionally, if the zero block $\pi_0$ exists and $n\in \pi_0$, we require~$\pi_0=[\pm n]$. Now, if $\pi$ is not a trivial single-block partition, there exists a unique non-zero block $\pi_i$ of $\pi$ such that $n\notin \pi_i$. In such a case we construct the ternary word $t=t_1t_2\dots t_{n-1}$ encoding a $2$-colored integer composition of $n$ by setting $t_k=+$ if $k\in \pi_i$, $t_k=-$ if $-k\in \pi_i$ and $t_k=0$ if $k\in\pi_0$ for all $1\leq k\leq n-1$. Otherwise, if $\pi$ is trivial, we assign to it the special word $\varnothing\in Comp_B(n)$. Such a correspondence is bijective.
   \begin{center}
  \begin{figure}[H]
    \resizebox{0.7\textwidth}{!} {\begin{tikzpicture}

\node (partition) at (-10, 0.5) {\huge $\pi=\{\underline{\{-1,3,-4,-5\}},\{-2,2\},\{1,-3,4,5\}\}$};

\fill[fill=red!10] (0,0) rectangle ++(1,1);
\fill[fill=cyan!20] (1,0) rectangle ++(2,1);
\fill[fill=red!10] (3,0) rectangle ++(1,1);
\fill[fill=cyan!20] (4,0) rectangle ++(1,1);

\node (tiles) at (0,0.5) {};

\draw[step=1cm,gray,very thick] (0,0) grid (5,1) rectangle (0,0);

\draw[red, line width=0.8mm] (1,1) -- (1,0);
\draw[red, line width=0.8mm] (3,1) -- (3,0);
\draw[red, line width=0.8mm] (4,1) -- (4,0);

\draw [{Stealth[length=4mm, width=4mm]}-{Stealth[length=4mm, width=4mm]}, shorten >= 0.5cm, shorten <= 0.5cm](partition) -- node [above=2mm] {\huge $t=-0+-$} (tiles);

% part two

\node (partition2) at (-8.4, -1) {\huge 
$\pi=\{\underline{\{-1,2,-3\}},\{1,-2,3\}\}$};

\filldraw[fill=red!10, draw=gray, very thick] (0,-1.5) rectangle ++(1,1);
\filldraw[fill=cyan!20, draw=gray, very thick] (1,-1.5) rectangle ++(1,1);
\filldraw[fill=red!10, draw=gray, very thick] (2,-1.5) rectangle ++(1,1);

\draw[red, line width=0.8mm] (1,-1.5) -- (1,-0.5);
\draw[red, line width=0.8mm] (2,-1.5) -- (2,-0.5);

\node (tiles2) at (0,-1) {};

\draw [{Stealth[length=4mm, width=4mm]}-{Stealth[length=4mm, width=4mm]}, shorten >= 0.5cm, shorten <= 0.5cm](partition2) -- node [above=2mm] {\huge $t=-+$} (tiles2);

\end{tikzpicture}}
  \caption
  {Type $B$ set partitions for $2$-colored compositions ${5=1_1+2_2+1_1+1_2}$ and ~${3=1_1+1_2+1_1}$.}
  \end{figure}
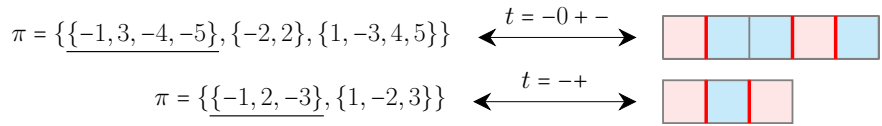
  \end{center}
  The operad $Comp_B$ comes with an involution that flips the color of each, but the first, summand of a given 2-colored composition. In terms of ternary words, it flips the sign of each occurrence of $\pm$. For each~$n\geq 2$, it is an automorphism of the lattice $Comp_B(n)$.
\end{example}

\begin{example}
\label{CompDEx}
  The operad $Comp_B$ contains a lattice suboperad $Comp_D$, which is defined by the following conditions: (1) for all $n \geq 2$, the symbol $0$ appears in a ternary word $t \in Comp_D(n)$ either at least twice or does not appear at all; (2) the subset of $Comp_D(n)$ consisting of ternary words without the symbol $0$ is the largest subset of $Comp_B(n)$ such that any two words in this subset differ in at least two positions, and $++\dots + \in Comp_D(n)$.  
  In particular,
  \begin{align*}
  Comp_D(2)&=\{+,\varnothing\}\\
  Comp_D(3)&=\{++, --, 00,\varnothing\}\\
  Comp_D(4)&=\{+++,+--,-+-,--+,00+,00-,0+0,0-0,+00,-00,000,\varnothing\}.
  \end{align*}  
  The subset of $Comp_D(n)$ defined by condition (2) is unique. It can be characterized as the set of all words in symbols $+$ and $-$ of length $n-1$ that have an even number of $-$'s.
  
  The operad $Comp_D$ can be regarded as an operad of \emph{integer compositions of type $D$} motivated, in view of the correspondence between restricted set partitions and integer compositions outlined in the previous example, by the notion of a set partition of type~$D$~\cite[Section 4]{reiner}. We recall that a set partition $\pi$ is said to be of type $D$ if it is of type $B$ and the zero block $\pi_0$, if present, is not a single pair $\{i,-i\}$. Condition (1) reflects this particular property, while (2) ensures that for all $n\geq 2$, $Comp_D(n)$ is indeed a lattice, rather than just a meet-semilattice that would be the case had we omitted it. 
  \end{example}

\begin{example}
Let $J=[-1,1]$ be considered now as an object in $(\Poly_0,\oplus)$.
The $n$-fold direct sum $J^{\oplus n}$ is the $n$-dimensional hyperoctahedron. 
Furthermore, we may consider $J$ with its ordinary multiplicative structure as a semigroup in $(\Poly_0,\oplus)$. That said, due to $\oplus$ not being cartesian, the word operad construction does not apply here directly. Nevertheless, by polytopal duality between hypercubes and hyperoctahedra, we can obtain the corresponding lattice operad by dualizing the face lattices of Example~\ref{CubeWordEx}. Namely, consider the operad $T^\vee$ in $(\Lat, \utimes)$, where for $n\geq 1$,~$T^\vee(n)$ is the $n$-fold \textit{upper} truncated product of the lattice \eqref{Rhombus} with itself. 
Upon changing the notation for the top and the bottom elements of \eqref{Rhombus} to $\mathbbm{1}$ and $0$ respectively, the elements of~$T^\vee(n)$ can be identified with ternary words of length $n$ consisting of the symbols $+,-,0$ and a special word~$\mathbbm{1}$. A non-special word $w=t_1t_2\dots t_{n}\in T^\vee(n)$ represents the face of the $n$-dimensional hyperoctahedron consisting of all the points $(x_1,x_2,\dots,x_n)\in\mathbb{R}^n$ such that $x_i \geq 0$ if $t_i = +$, $x_i \leq 0$ if $t_i=-$ and $x_i=0$ if $t_i=0$ for $1 \leq i\leq n$, all subject to $\sum\limits_{i=1}^{n}|x_i|=1$.
In particular, the empty face is encoded by a string of zeros. The special word $\mathbbm{1}\in T^\vee(n)$ represents the entire polytope. 
The lattice structure on non-special words is induced by the relations 
$0 \leq +$, $0 \leq -$, and in each arity $n\geq 1$, the special word $\mathbbm{1}\in T^\vee(n)$ is set to be the largest element.  
The compositional structure for non-special words of $T^\vee$ is that of the word operad for the multiplicative semigroup~$\{0,\pm\}$ as in Example~\ref{CubeWordEx}. Any composition with $\mathbbm{1}$ results in $\mathbbm{1}$.
Thus, operads $T$ and $T^\vee$ are isomorphic as set-valued operads and dual to each other as lattice operads. In particular, $T^\vee$ is equipped with hyperoctahedral symmetry, which, in terms of ternary words, amounts to permutting the letters and flipping the signs subject to the equivariance condition \eqref{Bequiv}.
%\vspace{-0.1in}
 \begin{center}
  \begin{figure}[H]
  \subcaptionbox{$-\circ_1 -+=+-$}
  {
  \resizebox{0.38\textwidth}{!} {%\documentclass[tikz,border=10pt]{standalone}
%\usepackage{MnSymbol}

%\begin{document}

\begin{tikzpicture}[z=-6.5]

\coordinate (LA) at (-5.5,1,0);
\coordinate (LB) at (-5.5,-1,0);

\coordinate (SA) at (-3.7,1,0);
\coordinate (SB) at (-2.7,0,0);
\coordinate (SC) at (-3.7,-1,0);
\coordinate (SD) at (-4.7,0,0);

\coordinate (TA) at (0,1,0);
\coordinate (TB) at (-1,0,0);
\coordinate (TC) at (0,-1,0);
\coordinate (TD) at (1,0,0);

\draw[-, line width=2pt, black!30!green!50] (LA)--(LB) node[midway, left, black] {$1$};
\node[fill, circle, red, minimum size=4pt, inner sep=0] at (LB) {};

\draw[opacity=0.8,fill=green!30] (SA)--(SB)--(SC)--(SD)--cycle;
\draw[-, line width=2pt, red] (SA)--(SD);

\draw[->, thick] (-2.4,0)-- (-1.4,0) node[midway, above] {$\circ_1$};
\node at (-2.6,0,0) {$1$};
\node at (-3.7,1.2,0) {$2$};

\node at (-5.1,0) {$\Large\oplus$};

\draw[opacity=0.8,fill=green!30] (TA)--(TB)--(TC)--(TD)--cycle;
\draw[-, line width=2pt, red] (TC)--(TD);
\node at (1.2,0) {$1$};
\node at (0,1.2,0) {$2$};

\end{tikzpicture}

%\end{document}
}
  }    
  \subcaptionbox{$-+\circ_2 ++=-++$}
  {
  \resizebox{0.48\textwidth}{!} {%\documentclass[tikz,border=10pt]{standalone}
%\usepackage{MnSymbol}

%\begin{document}

\begin{tikzpicture}[z=-6.5]
\coordinate (A1) at (2.5,0,-1);
\coordinate (A2) at (1.5,0,0);
\coordinate (A3) at (2.5,0,1);
\coordinate (A4) at (3.5,0,0);
\coordinate (B1) at (2.5,1,0);
\coordinate (C1) at (2.5,-1,0);

\coordinate (SA) at (-4.5,1,0);
\coordinate (SB) at (-3.5,0,0);
\coordinate (SC) at (-4.5,-1,0);
\coordinate (SD) at (-5.5,0,0);

\coordinate (TA) at (-1.7,1,0);
\coordinate (TB) at (-0.7,0,0);
\coordinate (TC) at (-1.7,-1,0);
\coordinate (TD) at (-2.7,0,0);

\draw[opacity=0.8,fill=green!30] (SA)--(SB)--(SC)--(SD)--cycle;
\draw[-, line width=2pt, red] (SA)--(SD);
\node at (-3.3,0,0) {$1$};
\node at (-4.5,1.2,0) {$2$};

\node at (-3,0) {$\Large\oplus$};

\draw[opacity=0.8,fill=green!30] (TA)--(TB)--(TC)--(TD)--cycle;
\draw[-, line width=2pt, red] (TA)--(TB);
\node at (-0.5,0,0) {$1$};
\node at (-1.7,1.2,0) {$2$};

\draw[->, thick] (-0.2,0)-- (1.2,0) node[midway, above] {$\circ_2$};

\draw (A1) -- (A2) -- (B1) -- cycle;
\draw (A4) -- (A1) -- (B1) -- cycle;
\draw (A1) -- (A2) -- (C1) -- cycle;
\draw (A4) -- (A1) -- (C1) -- cycle;
\draw [fill opacity=0.8,fill=red!30] (A2) -- (A3) -- (B1) -- cycle;
\draw [fill opacity=0.8,fill=green!30] (A3) -- (A4) -- (B1) -- cycle;
\draw [fill opacity=0.8,fill=green!30] (A2) -- (A3) -- (C1) -- cycle;
\draw [fill opacity=0.8,fill=green!30] (A3) -- (A4) -- (C1) -- cycle;

\node at (2.2, -0.4, 0) {$3$};
\node at (3.7, 0, 0) {$1$};
\node at (2.5, 1.2, 0) {$2$};

\end{tikzpicture}

%\end{document}
}
  }  
  \end{figure}
  \end{center}

\end{example}

\begin{example}
\label{OctCompBEx}
The dual of Example~\ref{CompBEx} is obtained as follows. Let $J=[-1,-1]$ be considered as object of $(\Poly_0,\oplus)$. Then the corresponding $J$-associative operad $\A J$ has the $(n-1)$-dimensional hyperoctahedron as its component of arity $n$. By taking the face lattices, we obtain the operad $Comp_B^\vee$ in $(\Lat, \utimes)$, where for $n\geq 2$, $Comp_B^\vee(n)$ is the $(n-1)$-fold upper truncated product of the lattice \eqref{Rhombus} with itself. As in the previous example, upon renaming the entries of \eqref{Rhombus}, the elements of~$Comp_B^\vee(n)$ can be identified with the words of length $n-1$ consisting of the symbols $+,-,0$, and a special word~$\mathbbm{1}$. For the non-special words, the compositional structure is that of the $M$-associative operad for $M=\{0,\pm\}$. Any composition with $\mathbbm{1}$ results in $\mathbbm{1}$. 
For each $n\geq 2$, the lattice structure of $Comp_B^\vee(n)$ is dual to that of $Comp_B(n)$, while the operads are isomorphic as operads in $\Sets$. Just like $Comp_B$, the operad $Comp_B^\vee$ comes with an involution that flips the sign of each letter of a ternary word.
   \begin{figure}[H]
  \subcaptionbox{$-\circ_2 +0=-+0$}
  {
  \resizebox{0.43\textwidth}{!} {%\documentclass[tikz,border=10pt]{standalone}
%\usepackage{MnSymbol}

%\begin{document}

\begin{tikzpicture}[z=-6.5]
\coordinate (A1) at (0,0,-1);
\coordinate (A2) at (-1,0,0);
\coordinate (A3) at (0,0,1);
\coordinate (A4) at (1,0,0);
\coordinate (B1) at (0,1,0);
\coordinate (C1) at (0,-1,0);

\coordinate (LA) at (-5.5,1,0);
\coordinate (LB) at (-5.5,-1,0);

\coordinate (SA) at (-3.7,1,0);
\coordinate (SB) at (-2.7,0,0);
\coordinate (SC) at (-3.7,-1,0);
\coordinate (SD) at (-4.7,0,0);

\draw[-, line width=2pt, black!30!green!50] (LA)--(LB) node[midway, left, black] {$1$};
\node[fill, circle, red, minimum size=4pt, inner sep=0] at (LB) {};

\draw[opacity=0.8,fill=green!30] (SA)--(SB)--(SC)--(SD)--cycle;
\node[fill, circle, red, minimum size=4pt, inner sep=0] (V) at (SB) {};

\draw[->, thick] (-2.4,0)-- (-1.4,0) node[midway, above] {$\circ_2$};
\node at (-2.7,0.4,0) {$1$};
\node at (-3.7,1.2,0) {$2$};

\node at (-5.1,0) {$\Large\oplus$};

\draw (A1) -- (A2) -- (B1) -- cycle;
\draw (A4) -- (A1) -- (B1) -- cycle;
\draw (A1) -- (A2) -- (C1) -- cycle;
\draw (A4) -- (A1) -- (C1) -- cycle;
\draw [fill opacity=0.8,fill=green!30] (A2) -- (A3) -- (B1) -- cycle;
\draw [fill opacity=0.8,fill=green!30] (A3) -- (A4) -- (B1) -- cycle;
\draw [fill opacity=0.8,fill=green!30] (A2) -- (A3) -- (C1) -- cycle;
\draw [fill opacity=0.8,fill=green!30] (A3) -- (A4) -- (C1) -- cycle;

\draw[-, line width=2pt, red] (B1)--(A2);

\node at (-0.3, -0.4, 0) {$3$};
\node at (1.2, 0, 0) {$1$};
\node at (0, 1.2, 0) {$2$};

\end{tikzpicture}

%\end{document}
}
  }  
  \vspace{0.5cm}
  \subcaptionbox{$+-\circ_3 +=+-+$}
  {
  \resizebox{0.43\textwidth}{!} {%\documentclass[tikz,border=10pt]{standalone}
%\usepackage{MnSymbol}

%\begin{document}

\begin{tikzpicture}[z=-6.5]
\coordinate (A1) at (0,0,-1);
\coordinate (A2) at (-1,0,0);
\coordinate (A3) at (0,0,1);
\coordinate (A4) at (1,0,0);
\coordinate (B1) at (0,1,0);
\coordinate (C1) at (0,-1,0);

\coordinate (LA) at (-2.7,1,0);
\coordinate (LB) at (-2.7,-1,0);

\coordinate (SA) at (-4.5,1,0);
\coordinate (SB) at (-3.5,0,0);
\coordinate (SC) at (-4.5,-1,0);
\coordinate (SD) at (-5.5,0,0);

\draw[-, line width=2pt, black!30!green!50] (LA)--(LB) node[midway, right, black] {$1$};
\node[fill, circle, red, minimum size=4pt, inner sep=0] at (LA) {};

\draw[opacity=0.8,fill=green!30] (SA)--(SB)--(SC)--(SD)--cycle;
\draw[-, line width=2pt, red] (SB)--(SC);

\draw[->, thick] (-2.2,0)-- (-1.2,0) node[midway, above] {$\circ_3$};
\node at (-3.3,0,0) {$1$};
\node at (-4.5,1.2,0) {$2$};

\node at (-3,0) {$\Large\oplus$};

\draw (A1) -- (A2) -- (B1) -- cycle;
\draw (A4) -- (A1) -- (B1) -- cycle;
\draw (A1) -- (A2) -- (C1) -- cycle;
\draw (A4) -- (A1) -- (C1) -- cycle;
\draw [fill opacity=0.8,fill=green!30] (A2) -- (A3) -- (B1) -- cycle;
\draw [fill opacity=0.8,fill=green!30] (A3) -- (A4) -- (B1) -- cycle;
\draw [fill opacity=0.8,fill=green!30] (A2) -- (A3) -- (C1) -- cycle;
\draw [fill opacity=0.8,fill=red!30] (A3) -- (A4) -- (C1) -- cycle;

\node at (-0.3, -0.4, 0) {$3$};
\node at (1.2, 0, 0) {$1$};
\node at (0, 1.2, 0) {$2$};

\end{tikzpicture}

%\end{document}
}
  }  
  \end{figure} 
\end{example}

\begin{example}
A natural example of an operad with symmetry of the dihedral type $I_2$ arises upon taking the operation of gluing two convex polygons along an edge as an operadic composition. The polygons here are to be considered up to combinatorial equivalence. 
Operads with this particular type of symmetry are embodied by the notion of a \textit{dihedral operad} \cite{dupont-vallette}, interpolating between that of a cyclic operad and a {non-$\Sigma$} cyclic operad. Also, in \cite{markl-envelope} an equivalent of the polygon-gluing operad, albeit only with the rotational symmetry involved, is treated as the terminal non-$\Sigma$ cyclic operad. Note that partial compositions in this framework are double-indexed.
   \begin{center}
  \begin{figure}[H]
    \resizebox{0.8\textwidth}{!} {\def\penta{\draw[black, fill=green!10] (A) -- (B) -- (C) -- (D) -- (E) -- cycle;}
\def\square{\draw[black, fill=green!10] (SA) -- (SB) -- (SC) -- (SD) -- cycle;}

%{{0., 1.}, {0.951057, 0.309017}, {0.587785, -0.809017}, {-0.587785, -0.809017}, % {-0.951057, 0.309017}}
\begin{tikzpicture}

\coordinate (A) at (0,1);
\coordinate (B) at (0.951057, 0.309017);
\coordinate (C) at (0.587785, -0.809017);
\coordinate (D) at (-0.587785, -0.809017);
\coordinate (E) at (-0.951057, 0.309017);

\coordinate (SA) at (0,0);
\coordinate (SB) at (1,0);
\coordinate (SC) at (1,-1);
\coordinate (SD) at (0,-1);

%\node (square1) at (-3,0) {\tikz{\square}};

\node[draw, regular polygon, regular polygon sides=4, inner sep=0pt, minimum height = 1cm, fill=green!10] (square1) at (-3, 0) {};

\node[scale=0.4, below] at (square1.side 3) {$0$};
\node[scale=0.4, left] at (square1.side 2) {$1$};
\node[scale=0.4, above] at (square1.side 1) {$2$};
\node[scale=0.4, right] at (square1.side 4) {$3$};

\draw[blue, thick] (square1.corner 1) -- (square1.corner 2);

\node[draw, regular polygon, regular polygon sides=3, inner sep=0pt, minimum height = 1cm, fill=green!10] (triangle1) at (-1.7, -0.09) {};

\node[scale=0.6] at (-2.3, 0) {$\sqcup$};

\node[scale=0.4, below] at (triangle1.side 2) {$0$};
\node[scale=0.4, left] at (triangle1.side 1) {$1$};
\node[scale=0.4, right] at (triangle1.side 3) {$2$};

\draw[blue, thick] (triangle1.corner 1) -- (triangle1.corner 2);

\draw[-{>[scale=0.7]}] (-1.1, 0)-- (-0.3,0) node[midway, above, scale=0.4] {${}_2\circ_{1}$};

\node[draw, regular polygon, regular polygon sides=5, inner sep=0pt, minimum height = 1cm, fill=green!10] (penta1) at (0.4, 0) {};

\draw[blue, thick, dashed] (penta1.corner 2) -- (penta1.corner 5) node [midway, below, scale=0.4] {$1'$};

\node[scale=0.4, above] at (penta1.side 1) {$2'$};
\node[scale=0.4, left] at (penta1.side 2) {$1$};
\node[scale=0.4, below] at (penta1.side 3) {$0$};
\node[scale=0.4, right] at (penta1.side 4) {$3$};
\node[scale=0.4, above] at (penta1.side 5) {$0'$};

\node[scale=0.6] at (1.2, 0) {$\simeq$};

\node[draw, regular polygon, regular polygon sides=5, inner sep=0pt, minimum height = 1cm, fill=green!10] (penta2) at (1.9, 0) {};

\node[scale=0.4, above] at (penta2.side 1) {$2$};
\node[scale=0.4, left] at (penta2.side 2) {$1$};
\node[scale=0.4, below] at (penta2.side 3) {$0$};
\node[scale=0.4, right] at (penta2.side 4) {$4$};
\node[scale=0.4, above] at (penta2.side 5) {$3$};

\end{tikzpicture}}
  \caption  
  {Gluing polygons up to combinatorial equivalence.}
  \end{figure}
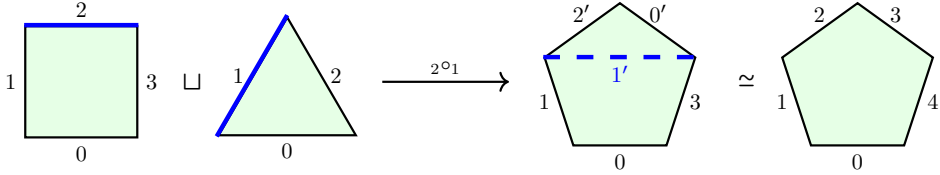
  \end{center}
  \vspace{-0.4in}
Another aspect in which the example that we are about to present differs from the rest is that the associated operad is to be defined in the category of lattices and monotone poset mappings, rather than lattice homomorphisms. 
The monoidal structure that we consider on this category is the disjoint union of lattices as defined in Section \ref{CatLat}.

For $n\geq 3$, let $P(n)$ be the face lattice of a regular $n$-gon. That is, $P(n)$ is a lattice of height four with~$2n+2$ elements, where the elements $\{v_0,\dots,v_{n-1}\}$ at the level $1$ and the elements $\{e_0,\dots,e_{n-1}\}$ at the level two are ordered by setting~$v_i \leq e_j$ iff $i=j-1$ or $i=j$; here, the indicies~$i,j$ are considered as elements of the cyclic group $\mathbb{Z}_n$. In particular, we will assume that $v_i$'s and $e_j$'s are themselves ordered cyclically via~${e_0 \prec e_1 \prec \dots e_{n-1} \prec e_0}$ and~
${v_0 \prec v_1 \prec \dots v_{n-1} \prec v_0}$.
Then, for~$n, m\geq 2$ and $1 \leq a \leq n$, $1 \leq b \leq m$, 
the partial composition ${{}_a\circ_b: P(n)\sqcup P(m)\to P(n+m-2)}$ amounts to merging the cyclic orders for $v_i$'s and $e_j$'s on~$P(n)$ and $P(m)$ in two slightly different ways.

Namely, consider the linear order representatives 
$
 e_{a-n-1}\prec \dots \prec e_{a-1} \prec e_a
$
and
$
 e_b' \prec e_{b+1}' \dots \prec e_{b+m-1}'
$
for the cyclic orders of the corresponding elements of $P(n)$ and $P(m)$. We define the linear order
\begin{align}
\label{CycLin}
e_{a-n-1}\prec \dots \prec e_{a-1} \prec e_{b+1}' \dots \prec e_{b+m-1}' 
\end{align}
on the union of these elements with $e_a$ and $e_b'$ discarded and fix an order-preserving bijection $\phi_{ab}$ from \eqref{CycLin} to the linearly ordered set
\[
 e_{a-n-1}\prec e_{a-n}\prec \dots \prec e_{a+m-2},
\]
where the indicies are considered modulo $n+m-2$. 
We complete the latter to a cyclically ordered set and identify it with the corresponding subset of $P(n+m-2)$.

Similarly, upon picking up the linear order representatives 
${v_{a-1}\prec v_{a}\prec \dots \prec v_{a+n-2}}$ and 
${v_{b}'\prec v_{b+1}'\prec \dots \prec v'_{b-1}}$ 
for the cyclic orders of the corresponding elements of $P(n)$ and $P(m)$,
we consider the linearly ordered set 
\[
 v_{a-1} \prec v'_{b+1} \prec v'_{b+2} \prec \dots \prec v'_{b-2} \prec v_{a} \prec v_{a+1} \prec \dots \prec v_{a+n-2}
\]
with an order-preserving bijection $\psi_{ab}$ to the linearly ordered set
${v_{a-1} \prec v_{a} \prec \dots \prec v_{a+n+m-4}}$,
where the indicies are taken modulo~$n+m-2$. Upon completing it to a cyclically ordered set, we identify it with the corresponding subset of~$P(n+m-2)$.

Now, the partial composition ${{}_a\circ_b: P(n)\sqcup P(m)\to P(n+m-2)}$ is defined as follows:
\begin{align*}
{}_a\circ_b(t)=
\begin{cases}
\mathbf{1},\quad &t=e_a\text{ or }t=e'_b\\
v_{a-1},\quad &t=v'_b\\
v_{a+m-1},\quad &t=v'_{b-1}\\
\phi_{ab}(t),\quad &t=e_j\text{ or }t=e'_j\text{ and }t\neq e_a, t\neq e'_b\\
\psi_{ab}(t),\quad &t=v_i\text{ or }t=v'_i\text{ and }t\neq v'_b, t\neq v'_{b-1}
\end{cases}
\end{align*}
It is a monotonic map of lattices that fails to be a lattice homomorphism due to ${\mathbf{0}={{}_a\circ_b}(e_a \wedge e'_b)\neq {{}_a\circ_b}(e_a) \wedge {{}_a\circ_b}(e'_b)=\mathbf{1}}$.

For $n\geq 3$, the dihedral group $\mathbbm{D}\text{ih}_n$ given in terms of its standard presentation $\langle r, s|\,r^n=s^2=(rs)^2=1 \rangle$ acts on $P(n)$ via $t_{i}\cdot r=t_{i+1}$ and $t_{i}\cdot s=t_{n-i-1}$, where $t_i=e_i, v_i$. The equivariance conditions for ${{}_a\circ_b}$ can be read off from the edge-gluing diagrams.
\begin{center}
  \begin{figure}[H]
    \resizebox{0.9\textwidth}{!} {\begin{tikzpicture}

\node[circle, inner sep=1pt, draw=black, fill=gray!20, minimum size=0.65cm] (top) at (0,2) {$\mathbf{1}$};
\node[circle, inner sep=1pt, draw=black, fill=gray!20, minimum size=0.65cm] (bottom) at (0,-2) {$\mathbf{0}$};
\node[fill=green!10, draw=black,=green!10, draw=black, minimum size=0.65cm, circle, inner sep=1pt](ae0) at (-4, 1) {$e_0$};
\node[fill=green!10, draw=black, minimum size=0.65cm, circle, inner sep=1pt](ae1) at (-3, 1) {$e_1$};
\node[fill=cyan!20, draw=black, minimum size=0.65cm, circle, inner sep=1pt](ae2) at (-2, 1) {$e_2$};
\node[fill=green!10, draw=black, minimum size=0.65cm, circle, inner sep=1pt](ae3) at (-1, 1) {$e_3$};

\node[fill=green!10, draw=black, minimum size=0.65cm, circle, inner sep=1pt](av0) at (-4, -1) {$v_0$};
\node[fill=cyan!20, draw=black, minimum size=0.65cm, circle, inner sep=1pt](av1) at (-3, -1) {$v_1$};
\node[fill=cyan!20, draw=black, minimum size=0.65cm, circle, inner sep=1pt](av2) at (-2, -1) {$v_2$};
\node[fill=green!10, draw=black, minimum size=0.65cm, circle, inner sep=1pt](av3) at (-1, -1) {$v_3$};

\foreach \i in {0,...,3} \draw (ae\i) -- (av\i);
\foreach \i [evaluate=\i as \prev using \i-1] in {1,...,3} \draw (ae\i) -- (av\prev);
\draw [preaction={draw=white, -, line width=3pt}] (ae0) -- (av3);

%%% 

\node[fill=green!10, draw=black,=green!10, draw=black, minimum size=0.65cm, circle, inner sep=1pt](be0) at (1, 1) {$e_0'$};
\node[fill=cyan!20, draw=black, minimum size=0.65cm, circle, inner sep=1pt](be1) at (2, 1) {$e_1'$};
\node[fill=green!10, draw=black, minimum size=0.65cm, circle, inner sep=1pt](be2) at (3, 1) {$e_2'$};

\node[fill=cyan!20, draw=black, minimum size=0.65cm, circle, inner sep=1pt](bv0) at (1, -1) {$v_0'$};
\node[fill=cyan!20, draw=black, minimum size=0.65cm, circle, inner sep=1pt](bv1) at (2, -1) {$v_1'$};
\node[fill=green!10, draw=black, minimum size=0.65cm, circle, inner sep=1pt](bv2) at (3, -1) {$v_2'$};

\foreach \i in {0,...,2} \draw (be\i) -- (bv\i);
\foreach \i [evaluate=\i as \prev using \i-1] in {1,...,2} \draw (be\i) -- (bv\prev);
\draw [preaction={draw=white, -, line width=3pt}] (be0) -- (bv2);
%%% 

\foreach \i in {0,...,2} 
{\draw (be\i) -- (top);
 \draw (bottom) -- (bv\i);}

\foreach \i in {0,...,3} 
{\draw  (ae\i) to  (top);
 \draw  (av\i) to (bottom);
 %\draw [in=180, out=-90] (av\i) to (bottom);
 }

 \draw[-{>[scale=0.7]}] (3.5, 0)-- (4.5,0) node[midway, above, scale=1] {${}_2\circ_{1}$};

 %%%

\node[fill=green!10, draw=black,=green!10, draw=black, minimum size=0.65cm, circle, inner sep=1pt](ae0) at (5, 1) {$e_0$};
\node[fill=green!10, draw=black, minimum size=0.65cm, circle, inner sep=1pt](ae1) at (6, 1) {$e_1$};
\node[fill=green!10, draw=black, minimum size=0.65cm, circle, inner sep=1pt](ae2) at (7, 1) {$e_2'$};
\node[fill=green!10, draw=black, minimum size=0.65cm, circle, inner sep=1pt](ae3) at (8, 1) {$e_0'$};
\node[fill=green!10, draw=black, minimum size=0.65cm, circle, inner sep=1pt](ae4) at (9, 1) {$e_3$};

\node[fill=green!10, draw=black, minimum size=0.65cm, circle, inner sep=1pt](av0) at (5, -1) {$v_0$};
\node[fill=green!10, draw=black, minimum size=0.65cm, circle, inner sep=1pt](av1) at (6, -1) {$v_1$};
\node[fill=green!10, draw=black, minimum size=0.65cm, circle, inner sep=1pt](av2) at (7, -1) {$v_0'$};
\node[fill=green!10, draw=black, minimum size=0.65cm, circle, inner sep=1pt](av3) at (8, -1) {$v_2$};
\node[fill=green!10, draw=black, minimum size=0.65cm, circle, inner sep=1pt](av4) at (9, -1) {$v_3$};

\node[circle, inner sep=1pt, draw=black, fill=gray!20, minimum size=0.65cm] (top) at (7,2) {$\mathbf{1}$};
\node[circle, inner sep=1pt, draw=black, fill=gray!20, minimum size=0.65cm] (bottom) at (7,-2) {$\mathbf{0}$};

\foreach \i in {0,...,4} 
{\draw  (ae\i) to  (top);
 \draw  (av\i) to (bottom);
 %\draw [in=180, out=-90] (av\i) to (bottom);
 }

\foreach \i in {0,...,4} \draw (ae\i) -- (av\i);
\foreach \i [evaluate=\i as \prev using \i-1] in {1,...,4} \draw (ae\i) -- (av\prev);

\draw [preaction={draw=white, -, line width=3pt}] (ae0) -- (av4);

%%%

\node at (10, 0) {\Large$\simeq$};

\node[fill=green!10, draw=black,=green!10, draw=black, minimum size=0.65cm, circle, inner sep=1pt](ae0) at (11, 1) {$e_0$};
\node[fill=green!10, draw=black, minimum size=0.65cm, circle, inner sep=1pt](ae1) at (12, 1) {$e_1$};
\node[fill=green!10, draw=black, minimum size=0.65cm, circle, inner sep=1pt](ae2) at (13, 1) {$e_2$};
\node[fill=green!10, draw=black, minimum size=0.65cm, circle, inner sep=1pt](ae3) at (14, 1) {$e_3$};
\node[fill=green!10, draw=black, minimum size=0.65cm, circle, inner sep=1pt](ae4) at (15, 1) {$e_4$};

\node[fill=green!10, draw=black, minimum size=0.65cm, circle, inner sep=1pt](av0) at (11, -1) {$v_0$};
\node[fill=green!10, draw=black, minimum size=0.65cm, circle, inner sep=1pt](av1) at (12, -1) {$v_1$};
\node[fill=green!10, draw=black, minimum size=0.65cm, circle, inner sep=1pt](av2) at (13, -1) {$v_2$};
\node[fill=green!10, draw=black, minimum size=0.65cm, circle, inner sep=1pt](av3) at (14, -1) {$v_3$};
\node[fill=green!10, draw=black, minimum size=0.65cm, circle, inner sep=1pt](av4) at (15, -1) {$v_4$};

\node[circle, inner sep=1pt, draw=black, fill=gray!20, minimum size=0.65cm] (top) at (13,2) {$\mathbf{1}$};
\node[circle, inner sep=1pt, draw=black, fill=gray!20, minimum size=0.65cm] (bottom) at (13,-2) {$\mathbf{0}$};

\foreach \i in {0,...,4} 
{\draw  (ae\i) to  (top);
 \draw  (av\i) to (bottom);
 %\draw [in=180, out=-90] (av\i) to (bottom);
 }

\foreach \i in {0,...,4} \draw (ae\i) -- (av\i);
\foreach \i [evaluate=\i as \prev using \i-1] in {1,...,4} \draw (ae\i) -- (av\prev);

\draw [preaction={draw=white, -, line width=3pt}] (ae0) -- (av4);

\end{tikzpicture}}
  \caption
  {Defining ${}_2\circ_1: P(4) \sqcup P(3) \to P(5)$.}
  \end{figure}
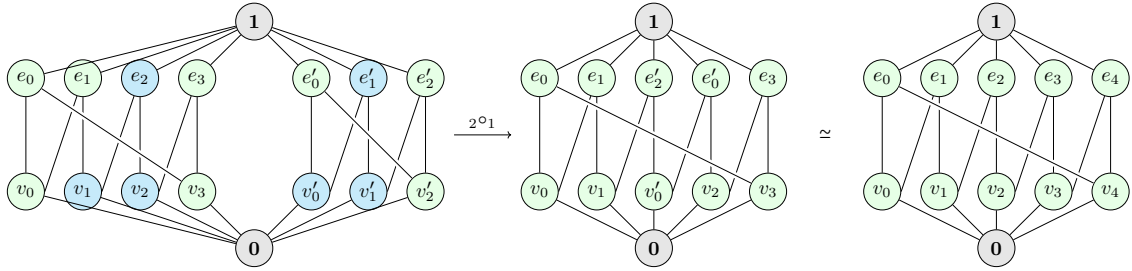
  \end{center}
\noindent

\end{example}
\begin{example}
\label{AssLatOp}
As established in \cite{diagonal1}, Loday's realization of associahedra provides a polytopal model for the topological $A_\infty$-operad $\oA_\infty:=\{\oK^{n-2}\}_{n\geq 2}$ in the symmetric monoidal category $(\cPoly, \times)$. 
In particular, the partial compositions ${\circ_i: \oK^{n-2}\times \oK^{m-2} \to \oK^{n+m-3}}$ can be taken to be codimension-one subpolytope embeddings ensuring that the face lattice map $\L$ has the desired functorial properties.
The corresponding face lattice operad can be identified with the non-$\Sigma$ operad of reduced planar rooted trees with grafting as the operadic composition, together with a designated empty tree $\varnothing$ in each arity. Specifically, for each $n\geq 2$, we have $\L(\oK^{n-2})=PT(n)\cup \{\varnothing\}$, where $\oK^{n-2}$ is Loday's realization of the $(n-2)$-dimensional associahedron, $PT(n)$ is the set of all planar rooted trees with $n$ leaves and no vertices having less than two descendants. The partial order on $PT(n)$ is induced by the covering relation $t\prec t'$ iff a tree $t'$ can be obtained from $t$ by contracting an edge, and $\varnothing$ is set to be the smallest element with respect to this order.

The lattice structure on $PT(n)\cup \{\varnothing\}$ may be elucidated upon passing to its combinatorial equivalent, that being the lattice of subdivisions of a regular oriented $(n+1)$-gon $P_{n+1}$ with one marked edge. Recall that a subdivision $d$ of $P_{n+1}$ is a subset of the set of all diagonals of $P_{n+1}$ such that no two diagonals from $d$ intersect. Given two subdivisions $d$ and $d'$, the join $d\vee d'$ is defined to be $d\cap d'$, while the meet $d\wedge d'$ is $d\cup d'$ if $d\cap d'=\varnothing$ and $d\wedge d'=\varnothing$ otherwise. In the latter equality, $\varnothing$ denotes the designated smallest element of the lattice. 
 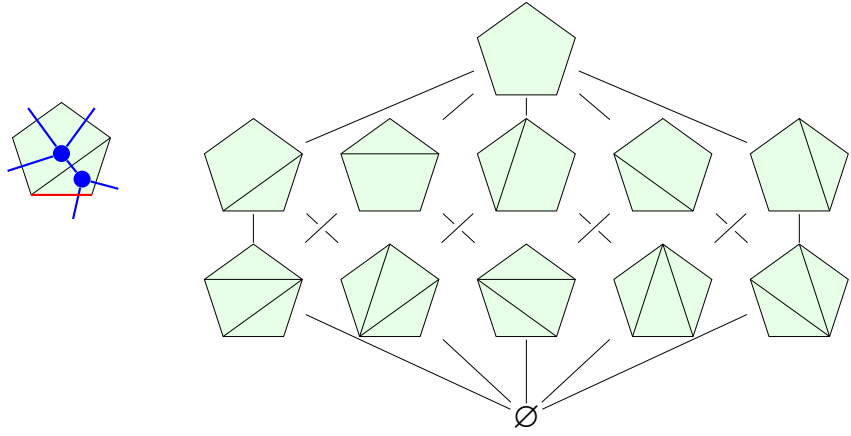
\begin{figure}[H]
  \centering
    \resizebox{0.7\textwidth}{!} {\def\penta{\draw[black, fill=green!10] (A) -- (B) -- (C) -- (D) -- (E) -- cycle;}

% Macro for drawing polygon diagonals. 
% Example \slice{A/C,C/E,E/G,C/G}

\newcommand{\slice}[1]{%    
    \penta;
    \draw \foreach \x/\y in {#1} {(\x)--(\y)};     
}

\begin{tikzpicture}
    \coordinate (A) at (0,1);
    \coordinate (B) at (0.951057, 0.309017);
    \coordinate (C) at (0.587785, -0.809017);
    \coordinate (D) at (-0.587785, -0.809017);
    \coordinate (E) at (-0.951057, 0.309017);

\node (top) at (0,3.5) {\tikz{\penta}};

\matrix (P) [matrix of nodes, column sep=0.5cm,row sep=0.4cm]
{
    \tikz{\slice{B/D}}& \tikz{\slice{B/E}} & \tikz{\slice{A/D}}& \tikz{\slice{E/C}} & \tikz{\slice{A/C}} \\

    \tikz{\slice{B/D, B/E};}& \tikz{\slice{B/D, D/A}} & 
    \tikz{\slice{C/E, B/E}}& \tikz{\slice{D/A, C/A}} & \tikz{\slice{E/C, A/C}} \\
};

\foreach \j in {1,...,5} \draw [shorten <=-2pt, shorten >=-2pt] (top) to (P-1-\j);

\foreach \i/\j in {1/1, 1/2, 2/1, 2/3, 3/2, 3/4, 4/3, 4/5, 5/4, 5/5} 
\draw [preaction={draw=white, -,line width=6pt},shorten <=-2pt, shorten >=-3pt]  (P-1-\i) to (P-2-\j);

\node (bottom) at (0,-3.7) {\Huge$\varnothing$};

\foreach \j in {1,...,5} \draw [shorten <=-2pt, shorten >=-2pt] (bottom) to (P-2-\j);

\node (left) at (-9, 1.3) 
{\tikz{
    \slice{B/D};
    \node[circle,fill, blue] (center) at (0,0) {};
    \draw[very thick, blue] (center)--(126:1.1);
    \draw[very thick,blue] (center)--(54:1.1);
    \draw[very thick,blue] (center)--(198:1.1);
    \node[circle,fill, blue] (right) at (0.4, -0.5) {};
    \draw[very thick,blue] (center)--(right);
    \draw[very thick,blue] (right)--(-32:1.3);
    \draw[very thick,blue] (right)--(-80:1.3);
    \draw[very thick,red] (C)--(D);
}
};
\end{tikzpicture}}  
  \caption{The rooted tree - polygon subdivision correspondence and the lattice of subdivisions of a pentagon.}
  \end{figure}
Note that the atoms of the lattice $PT(n)$ correspond to planar binary rooted trees with $n$ leaves. Thus, $PT$ contains the operad $PBT$ discussed in example \ref{TamariOp} as a set-valued suboperad. Equivalently, this is the suboperad of the topological operad $\mathcal{A}_\infty$ consisting of the vertices of the associahedra. Clearly, $PBT$ is not a lattice suboperad of $PT$ due to non-triviality of the Tamari order.
Also, while $PT$ is free as a non-$\Sigma$ operad in $\Sets$ with a single generator in each arity, it is not free as a poset-valued operad due to a non-trivial order structure on each component $PT(n)$ for $n>2$.
\end{example}

%\subsection{Generating series of ranked lattice operads}
Given a set-valued or a $\bbk$-linear operad $\oP$, let $|\oP(n)|$ denote the cardinality or, respectively, the dimension of the $n$-th component of $\oP$. Here, and in the following, we assume that all such quantities are finite. The \emph{generating}, or the \emph{Hilbert-Poincar{\'e}}, series of $\oP$ is the formal series $f_{\oP}(z):=\sum\limits_{n\geq 1}\frac{|\oP(n)|}{n!}z^n$ if $\oP$ is symmetric, and it is 
$f_{\oP}(z):=\sum\limits_{n\geq 1}|\oP(n)|z^n$ in the non-$\Sigma$ case. In case $\oP$ is a ranked lattice, or more generally, a poset-valued, operad, where each component $\oP(n)$ is ranked, these can be refined to bi-variate series, upon replacing each occurrence of $|\oP(n)|$ with the corresponding rank-generating polynomial.
Namely, if $\rho_{\oP(n)}:\oP(n)\to \mathbb{N}_0$ is the rank function of $\oP(n)$, then the corresponding rank-generating polynomial is $\sum\limits_{p\in\oP(n)}t^{\rho_{\oP(n)}(p)}=\sum\limits_{k\geq 0}c_{n,k} t^k$, where $c_{n,k}$ is the number of elements of $\oP(n)$ of rank $k$,
and we get
\[
f_{\oP}(z, t)=\sum\limits_{n\geq 1}\sum\limits_{k\geq 0}\frac{c_{n,k}}{n!}t^k z^n
\]
 for a symmetric operad $\oP$, while removing $n!$ in the above formula yields the bi-variate generating series for the non-$\Sigma$ case.
The lattice operads induced by an operads of polytopes are naturally ranked, with the rank generating polynomial in each arity being the $f$-polynomial of the corresponding polytope. In particular, we have
\[
f_{Comp}(z,t)=\sum\limits_{n=2}^{\infty}(1+t)^{n-1}z^n=\frac{z^2(1+t)}{(1-z-tz)}
\]
and
\[
f_{Comp_B}(z,t)=\sum\limits_{n=2}^{\infty}(2+t)^{n-1}z^n
=\frac{z^2(2+t)}{(1-2z-tz)}
\]
as the generating functions for the non-$\Sigma$ operads of compositions and type-B compositions respectively.

\section{Filtrations by lattice operads}
\label{filtrations}
The following is a generalization of the standard notion of a filtration of an associative algebra indexed by an ordered monoid. 
\begin{definition}
 Let $L$ be a lattice operad.
  An (increasing) \emph{$L$-filtration} 
  $F\oP =\{F_{p}\oP(n)\}_{p \in L(n), n \geq 1}$ of a $\bbk$-linear operad $\oP$
is a collection of $\bbk$-subspaces of $F_{p}\oP(n)\subset
\oP(n)$ indexed by the elements  $p\in L(n)$ for all $n
\geq 1$ subject to the following conditions:
\begin{itemize}
 \item[(i)] Monotonicity:
$F_{p'}\oP(n) \subset F_{p''}\oP(n)$
for all $p', p''\in L(n)$ such that $p'\leq p''$.
 \item[(ii)] Equivariance: 
$F_{p}\oP(n)\cdot \sigma = F_{p\cdot \sigma}\oP(n)$
for all $p\in L(n)$ and $\sigma\in \S_n$.
\item[(iii)] Compositional compatibility: 
$F_{a}\oP(m) \circ_i F_{b}\oP(n) \subset F_{a \circ_i
  b}\oP(m+n-1)$
for all $a\in L(m)$,  $b\in L(n)$ and $1 \leq i \leq m$. 
\item[(iv)] Unitality: if $L(1)$ has the smallest element $\mathbf{0}$ and $e\in\mathcal{P}(1)$
is the operadic unit, if it exists, then $\bbk\cdot e\subset F_\mathbf{0}\mathcal{P}(1)$.
\end{itemize}
\end{definition}
\begin{remark}
 There is an obvious dualization of the above definition covering the case of \textit{decreasing} filtrations. A typical use-case thereof is a filtration of an operad by the powers of an operadic ideal; a compelling example can be found in, for instance, \cite{dotsenkoprelie}.
\end{remark}

Some notions can be readily extrapolated from the realm of associative algebras to the operadic context. Specifically, an $L$-filtration $F\oP$ is said to be \emph{exhaustive} if 
$
\oP(n)=\bigcup\limits_{p\in L(n)} F_{p}\oP(n),
$
for all \ $n\geq 1$ and is said to be \emph{bounded below} if for any $n\geq 1$, there exists $b_n \in L(n)$ such that $F_{p}\oP(n) = 0$ for all $p \leq b_n$. Dually, an $L$-filtration $F\oP$ is called \emph{stabilized} if 
for any $n\geq 1$, there exists $c_n \in L(n)$ such that $F_p \oP=F_{p\wedge c_n}\oP$ for any $p \geq c_n$.

Additionally, we say that an $L$-filtration $F\oP$ is \emph{saturated} if 
\[
 F_{p'}\oP(n)\cap F_{p''}\oP(n) \subset F_{p'\wedge p''}\oP(n)
\]
for all $n\geq 1$, $p', p''\in L(n)$.
This property is vacuous for $\bbZ$-indexed filtrations and, more generally, for filtrations indexed by totally ordered monoids, of associative algebras, but it is non-trivial in the framework of filtrations by lattice operads.

The first few examples of filtrations by lattice operads are immediate:
\begin{example}
\label{FiltEx}
 \noindent
 \\
 \begin{enumerate}[(a)]
  \item 
  The ordinary notion of a $\bbZ$-indexed filtration of an associative algebra can be recovered from the above definition. Indeed, let $C\bbZ$ be the lattice operad of Example ~\ref{LatOpEx}\eqref{LZEx}.
  Then an (increasing) filtration $FA=\{F_n A\}_{n\in \mathbb{Z}}$ of an associative $\bbk$-algebra $A$ 
  can be identified with the corresponding $C\bbZ$-filtration of $A$, when the latter is regarded as an $\bbk$-linear operad concentrated in arity $1$.
  
  \item 
  \label{TrivFilt}
  For any lattice operad $L$ and any $\bbk$-linear operad $\oP$, there is a trivial exahaustive $L$-filtration $\mathbf{1}\mathcal{P}$ of $\mathcal{P}$ defined by $\mathbf{1}_p\mathcal{P}(n)=\mathcal{P}(n)$ for all $n\geq 1$ and $p\in L(n)$.
  
  \item For any $\bbk$-linear operad $\oP$ there is a \emph{tautological} filtration $\tau\oP$
  of $\oP$ by the lattice operad $Sub(\oP)$ as per Example \ref{LatOpEx}(\ref{SubLatOp}). 
  Namely, for all $n\geq 1$ and $U\in Sub(\oP)(n)$, we set $F_U\oP(n):=U$.
 \end{enumerate}
 \end{example}
  
  \begin{example}
  \label{GenCntEx}
  Let $E:=\{\mu_i|i\in I\}$ for some $I\neq \varnothing$ and $\Free(E)$ be the free $\bbk$-linear operad on the generators $E$. No conditions are imposed on the arities of $\mu_i$'s. 
  Given a tuple $p=(p_i)_{i\in I}\in \bbZ^{I}$ and $n\geq 1$, we define $F_p\Free(E)(n)$ to be the linear span of all partial-composition monomials $\nu$ in generators $E$, of arity $n$ and such that each $\mu_i$ appears as a factor in $\nu$ no more than $p_i$ times. In particular, $F_p\Free(E)(n)=0$ whenever $p_i\leq 0$ for all $i\in I$.
  Since $\Free(E)$ is free, it follows quite easily that the collection $\{F_p\Free(E)(n)\}_{n\geq 1, p\in \bbZ^I}$ is a bounded $C\bbZ^I$-filtration of $\Free(E)$, where $C\bbZ^I$ is the counting operad introduced in Example \ref{LatOpEx}\eqref{CIEx}. Here, each tuple $p\in C\bbZ^I(n)$ \textit{counts}, or rather gives an upper bound for, the number of occurrences of each of the generators $\mu_i$ in the monomial expressions appearing in $F_p\Free(E)(n)$.

  A certain modification of the above construction allows one to introduce in a similar way a $C\bbZ^I$-filtration for a general $\bbk$-linear operad $\oP$ upon choosing a presentation $\oP=\Free(E)/(\mathcal{R})$. This is to be addressed in Corollary \ref{StdFiltQuot} below. 
  \end{example}
  
  \begin{example}
  \label{DiffFilt}
  Let $A$ be a commutative associative $\bbk$-algebra $A$.
  A \textit{multilinear differential operator} is a linear map $\delta: A^{\otimes n} \to A$ such that $\delta=\delta(x_1,\dots, x_n)$ is a differential operator of order $\leq d_i$ in each variable $x_i$ for the fixed values of all the remaining arguments. A prototypical example to have in mind is that of a Poisson or, more generally, a Jacobi algebra $A$, whose bracket $\{-,-\}:A\otimes A\to A$ is a differential operator of order one in each variable. Any such operator $\delta$ is assigned a tuple $(d_1,\dots, d_n)\in \mathbb{Z}_{\geq 0}^n$.

  All multilinear differential operators on $A$ taken together form a suboperad ${\it Diff}_A$ of the endomorphism operad $\mathcal{E}nd(A)$, which carries a natural $\MZ$-filtration by the tuple of orders $(d_1,\dots, d_n)$, $n \geq 1$. In particular, the combinatorics of the lattice suboperad of $\MZ$ generated by the element $(1,1)\in \MZ(2)$ arises in connection to Poisson and Jacobi algebras \cite{superbig}.
   
  %\item 
  %Let $\mathcal{Q}$ be a $\bbk$-linear operad with a fixed space %of generators $E$.
  %Let $\mu\in E(m)$ for some $m$. For each $n\geq 1$ and $p\in %\mathbb{Z}$, 
  %define $F^\mu_p\mathcal{Q}(n)$ to be the set of all $\alpha\in %\mathcal{Q}(n)$
  %such that $\alpha$ is representable as a sum $\beta_1+\beta_2+%%%%\dots+\beta_t$, where each $\beta_j$ is a $\circ_i$-monomial in %elements of $E$ and $\mu$ appears as a factor of $\beta_j$ at %most $p$ times. Such a $F^\mu_p\mathcal{Q}(n)$ is well-defined %and is, in fact, a $\bbk$-subspace of $\mathcal{Q}(n)$.   
  %The assignment $p\mapsto F^\mu_p\mathcal{Q}(n)$ for each $n\geq %1$ and $p\in \mathbb{Z}$
  %defines a $C\bbZ$-filtration.
%   , where $C\bbZ$ is the lattice operad of Example ~\ref{LatOpEx}(1).
%   It could be unbounded (below) if $E$ is strictly larger than $\bbk\cdot \mu$.  
\end{example}

%\begin{example}
%TODO: a filtration of the algebraic $A_\infty$-operad by by the operad of %Example~\ref{AssLatOp}.
%\end{example}

Given a lattice operad $L$ and a $\bbk$-linear operad $P$, let ${\it Filt}(L, \mathcal{P})$
be the set of all $L$-filtrations of $\mathcal{P}$. This assignment is functorial contravariantly in $L$ and covariantly in $\oP$.
Indeed, first, we have the following
\begin{theorem}
\label{FiltFuncL}
 Let $\oP$ be a $\bbk$-linear operad and $f:K\to L$ be a lax morphism of lattice operads.
 Define $f^*: Filt(L, \oP) \to Filt(K, \oP)$ by setting
 \[
  (f^*(F\oP)_p\oP)(n):= F_{f(p)}\oP(n)
 \]
 for any $L$-filtration $F\oP$, all $n\geq 1$ and $p\in K(n)$. 
 Then 
 \begin{enumerate}[(a)]
  \item 
  $f^*(F\oP)\oP$ is a well-defined $K$-filtration.
  \item
  $(gf)^*=f^*g^*$ for any pair of lax morphisms $K\overset{f}{\to }L\overset{g}\to M$.
 \end{enumerate}
\end{theorem}
\begin{proof}
 \begin{enumerate}[(a)]
  \item
  To simplify the notation, we denote $f^*(F\oP)\oP$ by $G\oP$.
  %Monotonicity, equivariance and unitality of $G\oP$ are %immediate. 
  To check monotonicity of $G\oP$, let $p, q~\in~K(n)$ be such that $p\leq q$.
  Then $f(p) \leq f(q)$, since $f:K(n)\to L(n)$, being a lattice homomorphism, is monotonic. Now, by monotonicity of $F\oP$,
  \[
  G_p\oP(n)=F_{f(p)}\oP(n) \subset  F_{f(q)}\oP(n) = G_q\oP(n).
  \]
  To check the compositional compatibility, let $n\geq 1$, $1\leq i\leq m$ and $p\in K(m)$, $q\in K(n)$.
  Then
  \begin{align*}
   G_p\oP(m)\circ_i G_q\oP(n)&=
   F_{f(p)}\oP(m)\circ_i F_{f(q)}\oP(n)\\
   &\subset F_{f(p)\circ_i f(q)}\oP(n)\\
   &\subset 
   F_{f(p\circ_i q)}\oP(m+n-1)=G_{p\circ_i q}\oP(m+n-1)
  \end{align*}
  Here, the first inclusion follows from compositional compatibility for $F\oP$ and the second one follows by the defining property of a lax morphism \eqref{LaxMorph} and monotonicity 
  Equivariance and unitality of $G\oP$ are straightforward.
  \item
  Let $F\oP$ be a $M$-filtration of $\oP$, $n\geq 1$ and $p\in K(n)$. 
  Denote $g^*(F\oP)$ by $G\oP$. 
  Then $G_p\oP=F_{g(p)}\oP$ and, assuming arity $n$ everywhere, we have 
  \reqnomode
  \begin{align}
   ((gf)^*F\oP)_p\oP=F_{g(f(p))}\oP= (f^*G_p\oP)\oP=(f^*(g^*F_p\oP))\oP=((f^*g^*)F\oP)_p\oP
  \end{align}

 \end{enumerate} 
\end{proof}
\begin{corollary}
\label{totMZisZ}
By virtue of Example \ref{TotEx}, any $M\mathbb{Z}$-filtration of a $\bbk$-linear operad $\oP$ gives rise to a $C\bbZ$-filtration of $\oP$. 
\end{corollary}

\begin{theorem}
\label{FiltPush}
 Let $L$ be a lattice operad and $\phi: \oP\to\oQ$ be a morphism of $\bbk$-linear operads.
 Define $\phi_*:Filt(L, \oP)\to Filt(L, \oQ)$ by setting
 \[
  (\phi_*(F\oP)_p\oQ)(n):=\phi(F_p\oP(n))
 \]
 for any $L$-filtration $F\oP$, all $n\geq 1$ and $p\in L(n)$. 
 Then 
 \begin{enumerate}
  \item 
  $\phi_*(F\oP)\oQ$ is a well-defined $L$-filtration.
  \item
  $(\psi\phi)_*=\psi_*\phi_*$ 
  for any pair of $\bbk$-linear operad morphisms $\oP\overset{\phi}{\to}\oQ\overset{\psi}\to \mathcal{R}$.
 \end{enumerate}
 \end{theorem}
 \begin{proof}
  \noindent
    First, to simplify the notation, we denote $\phi_*(F\oP)\oQ$ by $G\oQ$.
   %Monotonicity, equivariance and unitality of $G_\oP$ are %immediate. 
   To check monotonicity of $G\oQ$, let $p, q \in L(n)$ be such that $p \leq q$. Then by monotonicity of $F\oP$,
   \[
   G_p\oP = \phi(F_p\oP) \subset \phi(F_q\oP) = G_q\oP. 
   \]
   To check the compositional compatibility, let $n\geq 1$, $1\leq i\leq m$ and $p\in L(m)$, $q\in L(n)$.
   Then
   \begin{align*}
    G_p\oQ(m)\circ_i G_q\oQ(n)&=\phi(F_p\oP(m))\circ_i\phi(F_q\oP(n))
    = \phi(F_p\oP(m)\circ_i F_q\oP(n))\\
    &\subset \phi(F_{p\circ_i q}\oP(m+n-1))=G_{p\circ q}\oQ(m+n-1)  
   \end{align*}
   Equivariance follows from the definition of an operad morphism and the equivariance property of $F\oP$. Unitality is straightforward.
   \item
   To get the second statement, let $p\in L(n)$. Omitting arities for the sake of brevity, we have
   \[
   ((\psi\phi)_*(F\oP)_p)\oQ=
   (\psi\circ \phi)(F_p\oP)=\psi(\phi(F_p\oP))=\psi(\phi_*(F\oP)_p\oQ)=
   \psi_*(\phi_*(F\oP)_p\oQ).
   \]

 \end{proof}
 \begin{corollary}
 \label{StdFiltQuot}
 Let $\oP$ be a $\bbk$-linear operad defined in terms of a presentation $\oP=\Free(E)/(\mathcal{R})$ with respect to a generating set $E=\{\mu_i|i\in I\}$.
 Recall (cf. Example \ref{LatOpEx}\eqref{CIEx}) that $\Free(E)$ carries a canonical $C\bbZ^I$-filtration. Then it induces a natural $C\bbZ^I$-filtration on $\oP$ along the projection $\Free(E)\twoheadrightarrow \oP$.
 \end{corollary}

 \begin{proposition}
 \label{reprprop}
  Let $\oP$ be a $\bbk$-linear operad. 
  The functor $Filt(-,\oP): \LatOper^{op} \to \Sets$ is representable.
  In fact,
  \[
   Filt(L,\oP) \simeq {{\rm Hom}}_{\LatOper}(L, Sub(\oP))
  \]
  functorially in $L$.
  Here, $\LatOper$ is the category of lattice operads with lax morphisms and $Sub(\oP)$ 
  is the lattice operad of Example \ref{LatOpEx}\eqref{SubLatOp}.
 \end{proposition}
 \begin{proof}
  In one direction, let $F\oP$ be an $L$-filtration of $\oP$. 
  Then $p\mapsto F_p\oP$ for all $p\in L$ defines a lax operad morphism $L\to Sub(\oP)$ as follows from the definition of an $L$-filtration.
  
  Conversely, given a lax operad morphism $f:L\to Sub(\oP)$,  
  we get an $L$-filtration of $\oP$ by considering the image of the tautological
  filtration $\tau\oP$ under
  \[
   f^*:Filt(Sub(\oP),\oP) \to Filt(L, \oP).
  \]
  provided by Theorem \ref{FiltFuncL}. 
  
  One readily verifies that these correspondences are indeed inverse to each other and are natural.
 \end{proof}
% 
% TODO: Inducing filtrations on the Hadamard product of operads.

% TODO: Filtrations on algebras via transfers along $\oP \to \mathcal{E}nd(V)$.
 \begin{lemma}
  For any lattice operads $K, L$, the set ${\rm Hom}_{\LatOper}(L,K)$ acquires a lattice structure by setting
  \[
   (f\wedge g)(p):= f(p)\wedge g(p)
  \]
  and
  \[
   (f\vee g)(p):= f(p)\vee g(p)
  \]
  for all $n\geq 1$, $p\in L(n)$.
 \end{lemma} 
 \begin{proof}
  First, we check that $f\wedge g$ is a well-defined lax lattice operad morphism.   
  To this end, let $m, n\geq 1$, $1\leq i\leq m$, $p\in L(m)$ and $q\in L(n)$.
  Then
  \begin{align*}
   (f\wedge g)(p\circ_i q) &=f(p\circ_i q)\wedge g(p\circ_i q)\\
   &\geq (f(p)\circ_i f(q))\wedge (g(p)\circ_i g(q)) \\   
   &\geq (f(p)\circ_i f(q))\wedge (g(p)\circ_i g(q)) \wedge (f(p)\circ_i g(q))\wedge (g(p)\circ_i f(q))\\
   &\geq [(f(p)\wedge g(p))\circ_i f(q)] \wedge [(f(p)\wedge g(p))\circ_i g(q)]\\
   &\geq (f(p)\wedge g(p))\circ_i (f(q)\wedge g(q))=(f\wedge g)(p)\circ_i(f\wedge g)(q)
  \end{align*}
  Here, the first comparison follows from the definition of a lax morphism applied to $f$ and $g$, 
  the second one is due to the natural order on the lattice $K(m+n-1)$,
  the third and the fourth one are an application of the distributive properties of Lemma \ref{MonoDist}.
  Equivariance of $f\wedge g$ is clear.
  
  A similar argument shows that $f\vee g$ is a lax operad morphism as well. 
  The lattice axioms for $\wedge$ and $\vee$ hold true, since they hold in each $K(n)$ for $n\geq 1$.
 \end{proof} 
 \begin{remark}
  If each component $K(n)$ of a lattice operad $K=\{K(n)\}_{n\geq 1}$ is a complete lattice, then so is ${\rm Hom}_{\LatOper}(L,K)$ for any lattice operad $L$.
 \end{remark}

 \begin{corollary}
 \label{FiltIsALat}
  $Filt(L,\oP)$ is a (complete) lattice with the trivial $L$-filtration $\mathbf{1}\oP$ as its largest element.  
 \end{corollary} 
  Indeed, the statement follows from Proposition \ref{reprprop}, the above remark and the fact that $Sub(\oP)(n)$ is a complete lattice for each $n\geq 1$.
  Explicitly, the lattice structure on $Filt(L,\oP)$ corresponds to taking (arity-wise) intersections and $\bbk$-linear sums of $L$-filtrations of $\oP$.
  
 \section{Filtrations with special properties}
\label{extra}
 Given a lattice operad $L$ and a $\bbk$-linear operad $\oP$, we define 
 ${\it PreFilt}(L, \oP)$ to be the set of all collections $F\oP=\{F_p\oP\}_{p\in L}$ of $\bbk$-linear subspaces $F_p\oP\subset \oP(n)$
 for all $n\geq 1$ and $p\in L(n)$. For any $F\oP, G\oP\in {\it PreFilt}(L, \oP)$ we define $(F\cap G)\oP\in PreFilt(L, \oP)$ by setting
 \[
  (F\cap G)_p\oP:=F_p\oP\cap G_p\oP\subset \oP(n)
 \]
and define $(F+G)\oP\in PreFilt(L, \oP)$ by setting
 \[
  (F+G)_p\oP:=F_p\oP+G_p\oP\subset \oP(n)
 \]
 for all $n\geq 1$ and $p\in L(n)$. 
 One verifies that $PreFilt(L, \oP)$ is a (complete) lattice with respect to these operations,
 and $Filt(L, \oP)$ is its sublattice.
 
 Now, let $\mathcal{C}$ be a subset of $PreFilt(L, \oP)$ such that 
 \begin{align}
 \label{ClOpReq}
  &\mathbf{1}\oP\in \mathcal{C}\notag\\
  &F\oP\in PreFilt(L, \oP)\Rightarrow \bigcap\limits_{G\oP\in \mathcal{C}: F\oP\leq G\oP} G\oP\in \mathcal{C}
 \end{align}
 where $\mathbf{1}\oP$ is the trivial filtration as per Example~\ref{FiltEx}\eqref{TrivFilt}.
 
 Then, as it is the case for any complete lattice \cite{everett}, under these assumptions, $\mathcal{C}$ induces the \textit{closure operator} $\gamma_\mathcal{C}:PreFilt(L,\oP) \to \mathcal{C}$ on $PreFilt(L, \oP)$
 that sends $F\oP$ to $\bigcap\limits_{G\oP\in \mathcal{C}: F\oP\leq G\oP} G\oP$.
 The closure operator satisfies
 \begin{enumerate}
  \item $F\oP\leq \gamma_\mathcal{C}(F\oP)$ for any $F\oP\in PreFilt(L, \oP)$; 
  \item $\gamma_\mathcal{C}(\gamma_\mathcal{C}(F\oP))=\gamma_\mathcal{C}(F\oP)$ for any $F\oP\in PreFilt(L, \oP)$;   
  \item $F\oP\leq G\oP\Rightarrow \gamma_\mathcal{C}(F\oP)\leq \gamma_\mathcal{C}(G\oP)$ for any $F\oP, G\oP\in PreFilt(L, \oP)$.  
 \end{enumerate}
For a given lattice operad $L$ and a $\bbk$-linear operad $\oP$, the closure operator $\gamma_\mathcal{C}$ returns the smallest $L$-filtration of type $\mathcal{C}$ out of the inital data provided in the form of a $L$-indexed collection of subspaces of $\oP$. This is illustrated in the following examples.
\begin{example}
Let $L$ be a lattice operad. Then by Corollary \ref{FiltIsALat}, $\mathcal{C}=Filt(L, \oP)$ satisfies the conditions \eqref{ClOpReq} above. 
  The corresponding closure operator returns for a given collection $E\oP\in PreFilt(L, \oP)$  the smallest $L$-filtration $G_E\oP$ such that 
  $E_p\oP\subset G_p\oP$ for all $n\geq 1$, $p\in L(n)$. 
  In other words, one may think of $G_E\oP$ as of the $L$-filtration of $\oP$ \textit{generated} by $E$.
  
  As an application, this allows one to introduce a certain analogue of the standard filtration of an associative algebra with respect to a chosen set of generators \cite[Section 1.6.2]{mcconnell}.
  Namely, let $\oP$ be an operad with a generating collection $E=\{E(n)\}_{n\geq 1}$.  
  Then, specializing to $L=M\mathbb{Z}$ in the above construction, we set $E_p\oP(n):=E(n)$ if $p\geq (1,\dots,1)$ and $E_p\oP(n):=0$ otherwise for all $n\geq 1$.
  In the terminology of \cite{superbig}, the corresponding closure $E\oP\mapsto G_E\oP$ 
  is the \emph{prestandard multifiltration} of $\oP$ with respect to $E$.
  
  Similarly, by taking $L=C\bbZ$ and setting $E_p\oP(n):=E(n)$ if $p\geq 1$ and $E_p\oP(n):=0$ otherwise for all $n\geq 1$, we get 
  a filtration of $\oP$ by the total degrees of the partial-composition polynomials in $E$.   
  In fact, the latter $C\bbZ$-filtration of $\oP$ is the $\rho^*$-image of the
  former $M\bbZ$-filtration, cf. Example \ref{TotEx} and Corollary \ref{totMZisZ}.
\end{example}

\begin{example}
Let $L$ be a lattice operad such that each $L(n)$ for $n\geq 1$ is complete.
Then the set ${\mathcal{S}:=\{G\oP\in Filt(L, \oP)|G\oP\hbox{ is saturated}\}}$ satisfies conditions \eqref{ClOpReq} above.
We call the corresponding closure operator $F\oP\mapsto \overline{F\oP}$ the \emph{saturation} of~$F\oP$.
\end{example}

\begin{example}
\label{DFilt}
  Recall that an increasing $\bbZ$-indexed filtration $\{F_iB\}_{i\in\bbZ}$ of an associative non-commutative algebra $B$ is called a\textit{ $D$-filtration}
  if $F_iB=0$ for $i<0$ and $[F_i B, F_j B]\subset F_{i+j-1}B$ for all $i, j\geq 0$. Here, the bracket denotes the ordinary commutator bracket $[a, b]=a\cdot b-b\cdot a$. 
  
  An operadic enhancement of the commutator bracket can be introduced as follows \cite{superbig}.
  For ${a \in \oP(m)}$, $b \in \oP(n)$, $1\leq i\leq m$, $1 \leq j \leq n$, we set \begin{align*}
[a,b]_{ij} := (a \circ_i b) \cdot (\tau_{i-1,j-1} \times \id_{m+n-i-j+1})
- (-1)^{|a||b|} (b \circ_j a) \cdot(
\id_{i + j -1} \times \tau_{m-i,n-j}),  
\end{align*}
where $\tau_{i-1,j-1} \times \id_{m+n-i-j+1} \in \S_{m+n-1}$ is the
permutation that swaps the block consisting of the first $i-1$ entries with the subsequent block of 
$j-1$ entries and leaves the remaining entries intact. The action of 
${\id_{i + j -1}\times \tau_{m-j,n-i}}$ is defined in a similar way. The flow chart shown below illustrates the~idea.  
 \begin{center}
  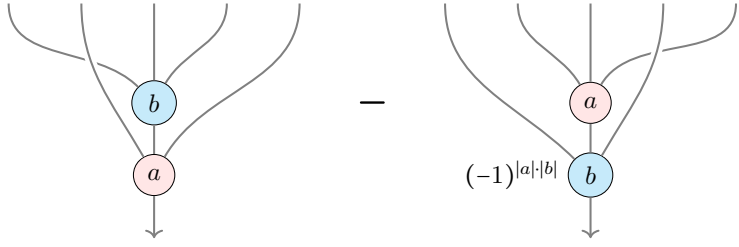
\begin{figure}[H]
    \resizebox{0.6\textwidth}{!} {\begin{tikzpicture}
\node (LB1) at (0, 0) {};
\node (LA1) at (1, 0) {};
%\node [label=above:$\tau_{i-1,j-1} \times \id_{m+n-i-j+1}$] (LB2) at (2, 0) {};
\node (LB2) at (2, 0) {};
\node (LB3) at (3, 0) {};
\node (LA3) at (4, 0) {};
\node (LA0) at (2, -3.5) {};

\node[circle, fill=cyan!20, draw=black] (LB) at (2, -1.5) {$b$};
\node[circle, fill=red!10, draw=black] (LA) at (2, -2.5) {$a$};

\draw[gray, thick] (LB1) to [out=-90] (LB);
\draw[gray, thick, preaction={draw=white, line width=4pt}] (LA1) to [out=-90, in=120] (LA);
\draw[gray, thick] (LB2) to [out=-90, in=90] (LB);
\draw[gray, thick] (LB3) to [out=-90, in=60] (LB);
\draw[gray, thick] (LA3) to [out=-90, in=60] (LA);
\draw[gray, thick] (LB) to (LA);
\draw[gray, thick, ->] (LA) to (LA0);

\node (LB1) at (6, 0) {};
\node (LA1) at (7, 0) {};
%\node[label=above:${\id_{i + j -1}\times \tau_{m-j,n-i}}$] (LA2) at (8, 0) {};
\node(LA2) at (8, 0) {};
\node (LB3) at (9, 0) {};
\node (LA3) at (10, 0) {};
\node (LB0) at (8, -3.5) {};

\node[circle, fill=cyan!20, draw=black, label=left:$(-1)^{|a|\cdot |b|}$] (LB) at (8, -2.5) {$b$};
\node[circle, fill=red!10, draw=black] (LA) at (8, -1.5) {$a$};

\draw[gray, thick] (LB1) to [out=-90] (LB);
\draw[gray, thick] (LA1) to [out=-90, in=120] (LA);
\draw[gray, thick] (LA2) to [out=-90, in=90] (LA);
\draw[gray, thick] (LA3) to [out=-90, in=60] (LA);
\draw[gray, thick,preaction={draw=white, line width=4pt}] (LB3) to [out=-90, in=60] (LB);
\draw[gray, thick] (LB) to (LA);
\draw[gray, thick, ->] (LB) to (LB0);

\node at (5,-1.5) {\Huge$-$};

\end{tikzpicture}}
  \caption
  {The operadic commutator $[a, b]_{2,2}$ of two ternary operations. }
  \end{figure}
  \end{center}  
\noindent
Next, we define~
$
[-,-]_{ij} :\MZ(m)
\times \MZ(n) \to \MZ(m+n-1),\ 1\leq i \leq m,\ 1 \leq j \leq n,
$
by setting
\[
[\vec a , \vec b]_{ij} :=
(\vec b_L + a_i, \vec a_L + b_j,a_i + b_j -1,\vec b_R + a_i, \vec a_R+b_j),
\]
for $\vec a = (\vec a_L, a_i, \vec a_R) \in \MZ(m)$ and $\vec b = 
(\vec b_L, b_j, \vec b_R) \in \MZ(n)$. Here, an expression of the form $\vec p + q$ for a vector $\vec p=(p_1,\dots, p_k)$ and a scalar $q$ stands for $(p_1+q, \dots, p_k + q)$.

We say that a $\MZ$-filtration $F\oP$ of a $\bbk$-linear operad $\oP$
is a $D$-\textit{filtration} if
\begin{align}
\label{DBracket}
[F_{\vec p}\oP(m),F_{\vec q}\oP(n)]_{ij} \subseteq  F_{[\vec p, \vec q]_{ij}}\oP(m+n-1)
\end{align}
for all $\vec p \in \MZ(m)$,  $\vec q \in \MZ(n)$, $1 \leq i \leq
m$, $1 \leq j \leq n$. 
In arity one this agrees with the associative-algebraic case recalled above.
In the operadic context, the prototypical example of a $D$-filtration is the $\MZ$-filtration of the operad of multilinear differential operators outlined in example \ref{DiffFilt}.

For a given operad $\oP$, the set of all $D$-filtrations $\mathcal{D}\subset PreFilt(\MZ, \oP)$ satisfies the conditions of \eqref{ClOpReq} giving rise to the closure operator $F\oP\mapsto (F\oP)^D$.  Indeed, this follows upon noting that 
$$[\bigcap\limits_{\mathcal{D}} G_p\oP(m),
\bigcap\limits_{\mathcal{D}} G_q\oP(n)]_{ij}\subseteq 
\bigcap\limits_{\mathcal{D}}[G_p\oP(m), G_p\oP(n)]_{ij},$$
where the intersection is taken over all $G\oP\in PreFilt(\MZ, \oP)$ containing $F\oP$.
In the terminology of~\cite{superbig}, the \emph{standard $D$-filtration} associated to a generating collection $E$ of a $\bbk$-linear operad $\oP$ is obtained as the composite $\overline{(G_E\oP)^D}$ of the three closure operators defined above. The construction can be made explicit in the form of a certain iterative algorithm as done in op.cit.
\end{example}

\begin{remark}
As a non-example, the property of being a stabilized filtration does not arise from a closure operator. 
  Indeed, the collection $\mathcal{B}:=\{G\oP\in Filt(L, \oP)|G\oP\hbox{ is stabilized}\}$ does not, in general, satisfy~\eqref{ClOpReq}.
\end{remark}

\section{Gradings and associated graded operads}
 
\begin{definition}
\label{GrOpDef}
 Let $G$ be an operad in $\Sets$ and $\oQ$ be a $\bbk$-linear operad.
 We say that $\oQ$ is \emph{$G$-graded} if
 \begin{itemize}
  \item[(i)]
  for all $n\geq 1$, we have a decomposition $\oQ(n)=\bigoplus\limits_{p\in G(n)}\oQ(n)_p$; % for some $\bbk$-linear subspaces $\oQ(n)_p$ of $\oQ(n)$;
  \item[(ii)]
  for any $m, n\geq 1$, $1\leq i\leq m$, $p\in G(m)$ and $q\in G(n)$,
  \[
   \oQ(m)_p\circ_i\oQ(n)_q\subset \oQ(m+n-1)_{p\circ_i q}
  \]
  \item[(iii)]
  If both $G$ and $\oQ$ are symmetric operads, we require $\oQ(n)_p\cdot\sigma=\oQ(n)_{p\cdot\sigma}$  for all $n\geq 1$, $\sigma\in \S_n$ and $p\in G(n)$
 \end{itemize}
 \begin{example}
 Any $\mathbb{Z}$-graded associative algebra can be regarded as an operad concentrated in arity $1$ graded by the operad $C\bbZ$ of Example \ref{LatOpEx}\eqref{LZEx}, when we forget about its order structure.
 \end{example}

 \begin{example} 
   Let $V=V_0\oplus V_1$ be a $\mathbb{Z}_2$-graded $\bbk$-vector space.
   For each $n\geq 1$ and $(d_1,\dots,d_n)\in \mathbb{Z}_2^n$ let~ $\mathcal{E}_V(n)_{(d_1,\dots,d_n)}$ be the 
   space of all $\bbk$-linear mappings $f:V^{\otimes n}\to V$ such that $$v\mapsto f(u_1,\dots,u_{i-1}, v, u_{i+1},\dots, u_n)$$ is homogeneous of degree $d_i$ for all $1\leq i\leq n$ and any choice of homogeneous elements $u_j\in V$.
   Then $\mathcal{E}_V=\{\mathcal{E}_V(n)\}_{n\geq 1}$, where 
   $\mathcal{E}_V(n):=\bigoplus\limits_{(d_1,\dots,d_n)\in \mathbb{Z}_2^n}\mathcal{E}_V(n)_{(d_1,\dots,d_n)}$
   is a suboperad of the endomorphism operad $\mathcal{E}nd(V)$ graded by the word operad $\W\mathbb{Z}_2$.  
 \end{example} 
\end{definition}
 
We say that a lattice operad $L$ is \emph{strictly monotonic} if 
$q<q'$ implies $p\circ_i q<p\circ_i q'$ and $q\circ_j p<q'\circ_j p$  for any $m, n\geq 1$, 
$1\leq i\leq m$, $1\leq j\leq n$, $p\in L(m)$, $q, q'\in L(n)$.
\begin{example}
 Lattice operads $C\bbZ$ and $M\mathbb{Z}$ are strictly monotonic. 
 To produce an example of a non-strictly monotonic lattice operad, one can modify $C\bbZ$ by replacing $\bbZ$ with the totally ordered monoid $\overline{\bbZ}:=\bbZ\cup\{\infty\}$, where $\infty$ is the designated largest element such that $a+\infty=\infty$ for all $a\in \overline{\bbZ}$.
 \end{example}

 \begin{definition}
 Let $L$ be a strictly monotonic lattice operad and $\oP$ be a $\bbk$-linear operad with an $L$-filtration $F\oP$.
 For all $n\geq 1$ and $p\in L(n)$, we define
 \[
  gr \oP(n)_p := F_p\oP(n)/\sum\limits_{r\in L: r<p} F_{r}\oP(n),
 \]
 where the quotient is a quotient of vector spaces.
 \end{definition}
  Let $m, n \geq 1$,
  and $\overline{\alpha}\in gr\oP(m)_p$, $\overline{\beta}\in gr\oP(n)_q$ for some $p\in L(m)$, $q\in L(n)$.
  For $1\leq i\leq m$, we define $\overline{\alpha}\circ_i\overline{\beta}:=\overline{\alpha\circ_i\beta}$,
  where $\alpha$, $\beta$ are any representatives of the respective cosets. 
  We set $gr\oP(n):=\bigoplus\limits_{p\in L(n)} \oP(n)_p$ and extend the map $(-\circ_i-)$ onto the entire $gr\oP$ in both arguments by linearity.
 \begin{proposition}
 The family $gr \oP=\{\oP(n)\}_{n\geq 1}$ with the partial compositions introduced above is a well-defined $L$-graded $\bbk$-linear operad. We call it the \emph{associated graded of $\oP$ with respect to the filtration $F\oP$}.
 \end{proposition}
 
 \begin{proof}  
  First, to check that $\circ_i$-compositions of $\oP$ descend properly to $gr\oP$, 
  let $\overline{\alpha}$, $\overline{\beta}$ be some homogeneous elements as above,  
  where $\alpha\in F_p\oP(m)$, $\beta\in F_q\oP(n)$.
  Then verifying that $\overline{\alpha}\circ_i\overline{\beta}:=\overline{\alpha\circ_i\beta}$ is well-defined amounts to checking that      
  \[
   \alpha\circ_i r_\beta + r_\alpha\circ_i \beta + r_\alpha\circ_i r_\beta \in \sum\limits_{r\in L: r<p\circ_i q} F_r\oP(n)
  \]
  for any $r_\alpha\in F_{p'}\oP(m)$, where $p'<p$, and any $r_\beta\in F_{q'}\oP(m)$, where $q'<q$.
  This follows from the monotonicity property of an $L$-filtration $F\oP$ and strict monotonicity of $L$.  Furthermore, since $\alpha\circ_i\beta\in F_{p\circ_i q}\oP(m+n-1)$, then condition (ii) of definition \ref{GrOpDef} holds for $gr\oP$.
  
  The parallel and the sequential associativity of $(-\circ_i-)$'s follow from $\overline{(\alpha\circ_i\beta)\circ_j \gamma}=(\overline{\alpha}\circ_i\overline{\beta})\circ_j \overline{\gamma}$
  and ${\overline{\alpha\circ_i(\beta \circ_j \gamma)}=\overline{\alpha}\circ_i(\overline{\beta}\circ_j \overline{\gamma})}$
  holding for any relevant combination of $i$ and $j$. 
  Finally, if both $L$ and $\oP$ are symmetric, then the equivariance property of the filtration $F\oP$ gives rise to condition (iii) of definition \ref{GrOpDef}.
 \end{proof}
  As an immediate observation we get the following
  \begin{lemma}
   Let $F\oP$ be a $\MZ$-filtration of a $\bbk$-linear operad $\oP$ such that it is a $D$-filtration in the sense of Example \ref{DFilt}.   
   Then all operadic commutators $[-,-]_{ij}$ in $gr \oP$ vanish.
  \end{lemma}
 
 \begin{example}
 Let $V$ be a vector space, $Sym(V)$ be symmetric algebra over $V$ and ${\it Diff}_{S(V)}$ be the operad of multilinear differential operators taken with a $M\bbZ$-filtration as in Example~\ref{DiffFilt}. Then $gr {\it Diff}_{S(V)}=\W S(V)$, where the induced $M\bbZ$-grading is determined by the vector of polynomial degrees $(deg(p_1),\dots,deg(p_n))$ for $(p_1,\dots,p_n)\in \W S(V)(n)$, where $1\leq i\leq n$, $n\geq 1$.
 \end{example}
 
 \begin{example}
 For $d\geq 1$, let $S\subset \mathbb{Z}^d$ be a finite non-empty set. The $M$-associative operad $\A S$ can be regarded as the operad of grid walks\footnote{A more common term is a \textit{lattice walk}, but we would like to avoid a clash of terminology.}, possibly with self-intersections, starting at the origin $\mathbf{0}\in \bbZ^d$ and with $S$ as the set of elementary steps. Namely, for $n\geq 1$, a word $(s_1,s_2,\dots,s_n)\in\A S(n+1)$ is interpreted as the walk of $n$ steps 
 passing through the points
 $$
 \mathbf{0},\,s_1,\,s_1+s_2,\dots, s_1+s_2+\dots+s_n.
 $$ 
 The partial composition $p \circ_i q$ amounts to inserting path $q$ into path $p$ at its $i$-th node. 
\begin{figure}[H]
{
 \resizebox{0.7\textwidth}{!} {\begin{tikzpicture}[>={Stealth}]
\tikzset{dot/.style={fill=black,circle}}
\draw[step=1cm,gray,very thin] (0,0) grid (3,3);
\node[dot] (A1) at (0,0) {};
\node[dot] (A2) at (1,0) {};
\node[dot, blue] (A3) at (1,1) {};
\node[dot] (A4) at (2,2) {};
\node[dot] (A5) at (3,3) {};
\draw[line width=0.2em] (A1) -- (A2) -- (A3) -- (A4) -- (A5);
\draw[line width=0.2em, ->] (A4) -- (A5);

\draw[step=1cm,gray,very thin] (5,0) grid (7,2);
\node[dot] (B1) at (5,0) {};
\node[dot] (B2) at (6,1) {};
\node[dot] (B3) at (6,2) {};
\node[dot] (B4) at (7,2) {};    
\draw[line width=0.2em] (B1) -- (B2) -- (B3) -- (B4);
\draw[line width=0.2em, ->] (B3) -- (B4);

\draw[step=1cm,gray,very thin] (10,-1) grid (15,4);
\node[dot] (C1) at (10,-1) {};
\node[dot] (C2) at (11,-1) {};
\node[dot, blue] (C3) at (11,0) {};

\node[dot,blue] (C4) at (12,1) {};
\node[dot,blue] (C5) at (12,2) {};
\node[dot,blue] (C55) at (13,2) {};
\node[dot] (C6) at (14,3) {};
\node[dot] (C7) at (15,4) {};
\draw[line width=0.2em] 
(C1) -- (C2) -- (C3);

\draw[line width=0.2em] 
 (C55) -- (C6) -- (C7);
\draw[line width=0.2em, ->] (C6) -- (C7);

\draw[line width=0.3em, blue] 
 (C3) -- (C4) -- (C5) -- (C55);

\node at (4,1.5) {$\circ_3$};

\node at (8.5,1.5) {$=$};
\end{tikzpicture}}
}
\caption{A partial composition in the operad $\A S$ of Delannoy walks; ${S=\{(1,0),(0,1), (1,1)\}}$.}
\end{figure}
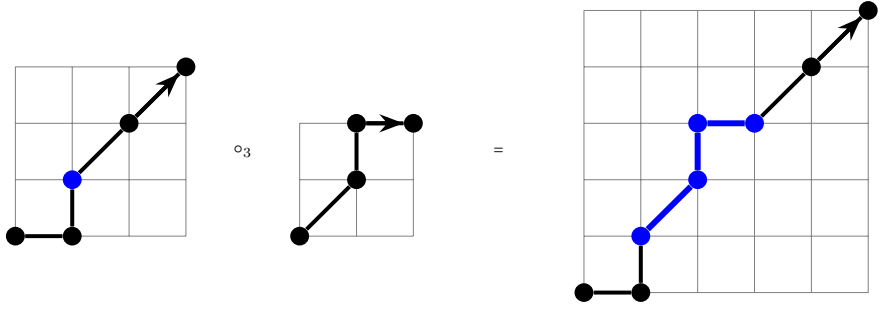
We define a $C\mathbb{Z}^d$-filtration on the $\bbk$-linearization ${Path_S}$ of $\A S$ as follows. 
For $p=(p_1,\dots,p_d)\in C\mathbb{Z}^d(n)$, the component $F_p{Path_S}(n)$ is the linear span of all the walks with $n$ steps from $S$, confined to the region $x_i\leq p_i,\quad i=1,2\dots d$. If we further assume that for any elementary step $s=(t_1,t_2,\dots, t_d)\in S$ we have $t_i\geq 0$, $i=1,2,\dots, d$, then the associated graded $gr Path_S$ can be regarded as the operad of the (linear combinations of) end-points of walks. Its homogeneous components with respect to the corresponding $C\bbZ^k$-grading are given by
\[
(gr Path_S(n))_p=\bbk \langle w\,|\, w\text{ is a walk of }n\text{ steps ending at }p\rangle. 
\]
For instance, in the particular case of $S=\{(1,0),(0,1)\}$ (the \emph{NE walks}), we have 
$$\text{dim}(gr Path_S(n)_{(p_1,p_2)})=
\begin{cases}
{n \choose p_1},\quad &p_1+p_2=n\\
0,\quad &\text{otherwise}
\end{cases}.$$

\end{example}

By virtue of bijections between certain families of combinatorial objects and grid walks, the operad $\A S$, for different choices of $S$, serves as a container of various combinatorial operads. For example, the operad of integer partitions $Part$ presented in example \ref{YoungLatEx} can be identified with the suboperad of $\A S$ for $S = \{(1,0), (0,1)\}$ generated by a countable family of left-turning hook-shaped walks. Specifically, this well-known correspondence is established by drawing a Young diagram in the first quadrant of the $xy$-plane, as shown below, and tracing its boundary by going from the bottom-left to the top-right corner of the diagram in the counterclockwise direction. 
The composition of partitions \eqref{PartComp}, presented graphically by insertions of Young diagrams, translates to path insertions.
 For each $a, b\geq 1$, the hook-shaped walk $\text{J}_{a,b}$, consisting of $a$ right-moving horizontal unit steps followed by $b$ upward-moving vertical unit steps, serves as a generator of $Part$ in arity $a+b+1$. 
 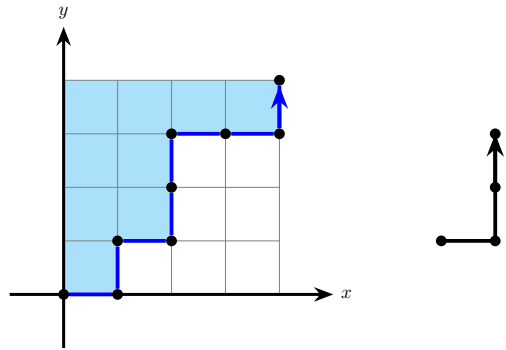
\begin{figure}[H]
{
 \resizebox{0.4\textwidth}{!} {\begin{tikzpicture}[>={Stealth}]
\tikzset{dot/.style={fill=black,circle,inner sep=0pt,minimum size=2mm}}

\fill[cyan!30] (0,0) rectangle (1,4);
\fill[cyan!30] (1,1) rectangle (2,4);
\fill[cyan!30] (2,3) rectangle (4,4);

\draw[step=1cm,gray,very thin] (0,0) grid (4,4);
\node[dot] (A1) at (0,0) {};
\node[dot] (A2) at (1,0) {};
\node[dot] (A3) at (1,1) {};
\node[dot] (A4) at (2,1) {};
\node[dot] (A5) at (2,2) {};
\node[dot] (A6) at (2,3) {};
\node[dot] (A7) at (3,3) {};
\node[dot] (A8) at (4,3) {};
\node[dot] (A9) at (4,4) {};

\draw[->,ultra thick] (-1,0)--(5,0) node[right]{$x$};
\draw[->,ultra thick] (0,-1)--(0,5) node[above]{$y$};

\draw[line width=2pt,draw=blue] \foreach \x [remember=\x as \lastx (initially 1)] in {2,...,9}{(A\lastx) -- (A\x)};
\draw[line width=2pt, ->,draw=blue] (A8) -- (A9);

\draw[line width=2pt] (7,1) node[dot] {} -- (8,1) node[dot] {} -- (8,2) node[dot] {} -- (8,3) node[dot] {};
\draw[line width=2pt, draw=black, ->] (8,2) -- (8,3);

\end{tikzpicture}}
}
\caption{The grid walk for partition (4,2,2,1) and a hook-shaped walk $\text{J}_{1,2}$.}
\end{figure}
\noindent
The generators satisfy 
 \[
  \text{J}_{a,b}\circ_1\text{J}_{c,d}=
  \text{J}_{a+i-1,b}\circ_i\text{J}_{c-i+1,d}
 \]
 for all $1\leq i\leq c$, and
 \[
  \text{J}_{a,b+i-1}\circ_{a+b+1}\text{J}_{c,d-i+1}=
  \text{J}_{c+j-1,d}\circ_{j}\text{J}_{a-j+1,b}
 \]
 for all $1\leq i\leq d$, $1\leq j\leq a$. 
 
As a special case, the suboperad of $\A S$ generated by all equilateral hook-shaped walks $J_{a,a}$, or equivalently, the suboperad of $\A S$ consisting of all NE walks with $2n$ steps from $(0,0)$ to $(n,n)$ that do not rise above the line $y=x$, can be viewed as an operad of binary trees, owing to a well-known bijection between Dyck paths of length $2n$ and planar rooted binary trees with $n+1$ leaves. This operad is not isomorphic to the operad $PBT$ considered in example \ref{TamariOp}. More generally, for a cone $C$ in $\bbZ^d$ and a set of elementary steps $S\subset \bbZ^d$, the walks staying within $C$ form a suboperad of $\A S$.

\subsection*{Acknowledgements}
The author is grateful to Anton Khoroshkin and Guillaume Laplante-Anfossi for helpful communications, and to Martin Markl for numerous discussions, helpful suggestions and support. The work is supported by RVO:67985840 of the Institute of Mathematics of the Czech Academy of Sciences and the Praemium Academiae grant of M. Markl. 

\printbibliography
\addresseshere

\newpage
\pagestyle{plain}
\section*{Appendix}
In what follows, $\mathcal{N}_p$ denotes the suboperad (in $\Lat$) of the non-$\Sigma$ version of $\MZ$ generated by a single element $p\in \MZ(n)$.

\begin{tabular}{p{0.4\textwidth} p{0.6\textwidth}}
  \vspace{40pt} \includegraphics[width=0.36\textwidth]{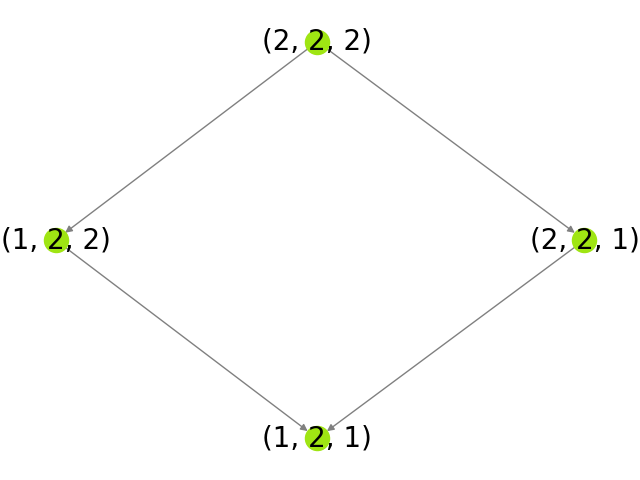} &
  \vspace{0pt} \includegraphics[width=0.6\textwidth]{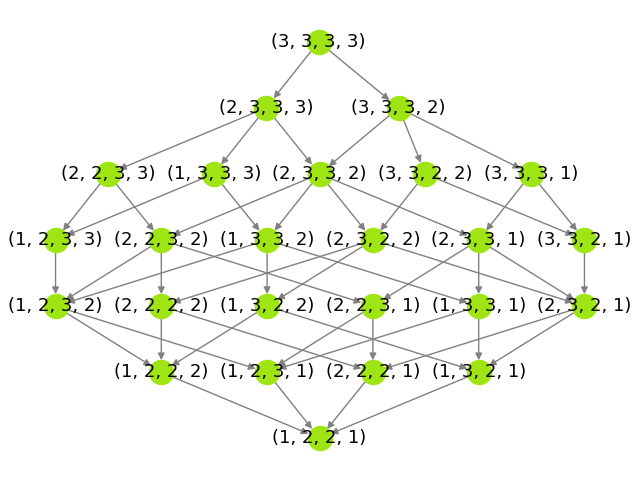}
\end{tabular}
\begin{center}
\textsc{Figure 3.} 
$\mathcal{N}_{(1,1)}(3)$ and $\mathcal{N}_{(1,1)}(4)$
\end{center}
\begin{tabular}{p{0.4\textwidth} p{0.6\textwidth}}
  \vspace{40pt} \includegraphics[width=0.36\textwidth]{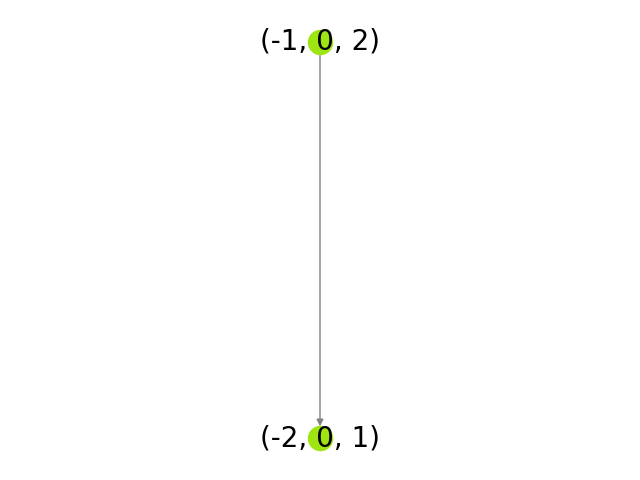} &
  \vspace{0pt} \includegraphics[width=0.5\textwidth]{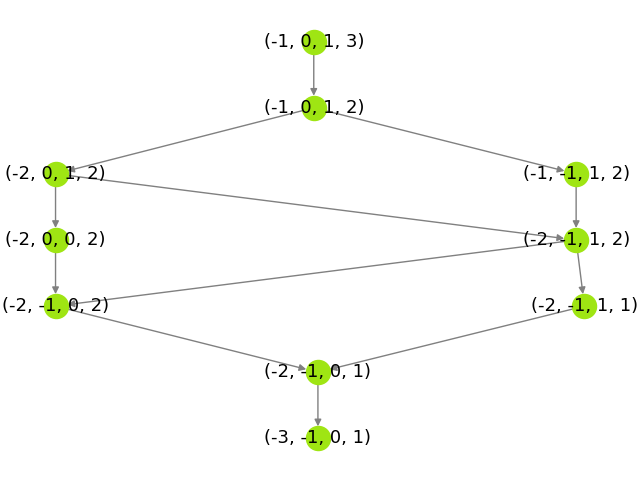}
\end{tabular}
\begin{center}
\textsc{Figure 4.} 
$\mathcal{N}_{(-1,1)}(3)$ and $\mathcal{N}_{(-1,1)}(4)$
\end{center}
\vspace{0.2in}
\begin{center}
\includegraphics[height=1.8in]{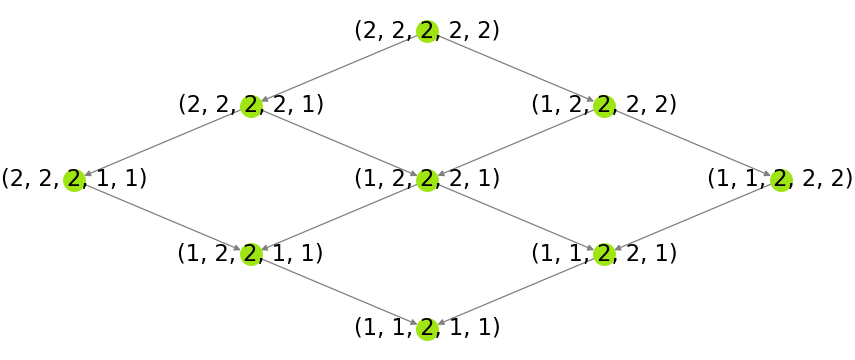}\\
\textsc{Figure 5.} 
$\mathcal{N}_{(1,1,1)}(5)$
\end{center}

\newpage
Let $\mathcal{M}_{(1,1)}$ be the suboperad of the symmetric lattice operad $\MZ$ generated by $(1,1)\in \MZ(2)$.
\begin{center}
\includegraphics[height=1.8in]{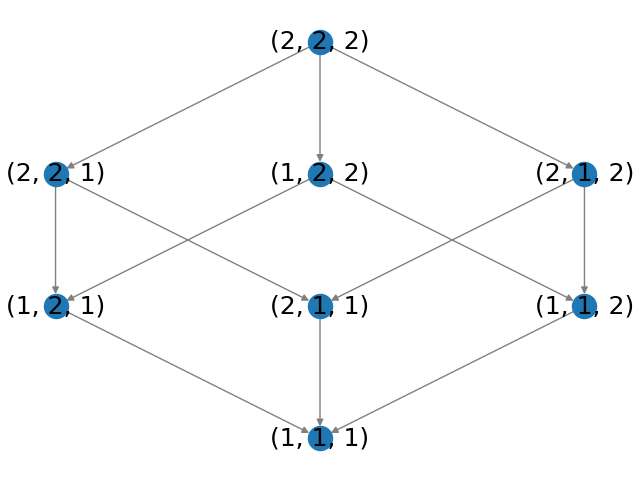}\\
\textsc{Figure 6.} 
$\mathcal{M}_{(1,1)}(3)$
\end{center}
\vspace{0.5in}
A computer-assisted calculation indicates that $|\mathcal{M}_{(1,1)}{(n)}|=(n-1)^n$
for $2 \leq n\leq 6$ (OEIS: A065440).
\end{document}